%% file: main.tex
\documentclass{imsart}

\usepackage{amsthm,amsmath}
\RequirePackage[numbers]{natbib}
\RequirePackage[colorlinks,citecolor=blue,urlcolor=blue]{hyperref}

\usepackage[T1]{fontenc}
\usepackage[utf8]{inputenc}
\usepackage{microtype}
\usepackage[many]{tcolorbox}
\usepackage{epsfig,subfig}
\usepackage{amssymb,amsfonts,amsmath,multirow,rotating}
\usepackage{dsfont}
\usepackage{booktabs}
\usepackage[bottom]{footmisc}
\usepackage{indentfirst}
\usepackage{lscape}
\usepackage{color}
\usepackage{graphicx}
\usepackage{wrapfig}
\usepackage{mathtools}
\usepackage{tabularx}
\usepackage{adjustbox}
\usepackage{enumerate}

\usepackage{amsthm}
\newtheorem{theorem}{Theorem}

\newtheorem{condition}{Condition}

\newtheorem{lemma}[theorem]{Lemma}

\newtheorem{proposition}[theorem]{Proposition}

\theoremstyle{remark}

\numberwithin{equation}{section}
\numberwithin{theorem}{section}
\DeclareMathOperator\symdif{\triangle}

\startlocaldefs

\input{0macros.tex}
\endlocaldefs

\begin{document}

\input{0front.tex}

\tableofcontents

\section{Introduction}
\input{1intro.tex}

\section{Asymptotic theory: sample drawn from the GP distribution}
\label{sec:GP}
\input{2asygp.tex}

\section{Asymptotic theory: sample drawn from the DoA of a GEV distribution}
\label{DoA_bias}
\input{3asydoa.tex}

\section{Estimating the extreme value index}
\label{Application}
\input{4estim.tex}

\section{Finite-sample performance}
\label{simulation}
\input{5simu.tex}

\section{Discussion and open problems}
\label{sec:disc}
\input{6disc.tex}

\appendix 

\section{Preliminary distributional representations} 
\label{sec:rep}
\input{appAprelim.tex}

\section{Proof of Proposition \ref{Prop:variance} }
\label{proof:variance}
\input{appBprop22.tex}

\section{Proof of Proposition \ref{Prop:Hajek_distribution}}
\label{sec:proof_Hajek}
\input{appCprop23.tex}

\section{Proof of Theorem~\ref{Prop:bias}}
\input{appDprop3.tex}

\section{Proofs for Section~\ref{Application}} 
\label{proof:application}
\input{appEsec4.tex}

\begin{acks}[Acknowledgments]
We are grateful to the constructive comments by the referees and the associate editor, which stimulated us to improve the structure of the paper.
\end{acks}

%%%%%%%%%%%%%%%%%%%%%%%%%%%%%%%%%%%%%%%%%%%%%%
%% Supplementary Material, if any, should   %%
%% be provided in {supplement} environment  %%
%% with title and short description.        %%
%%%%%%%%%%%%%%%%%%%%%%%%%%%%%%%%%%%%%%%%%%%%%%
%\begin{supplement}
%\stitle{???}
%\sdescription{???.}
%\end{supplement}

%% if your bibliography is in bibtex format, uncomment commands:
\bibliographystyle{imsart-number} % Style BST file (imsart-number.bst or imsart-nameyear.bst)

\bibliography{reference_XUstats}       % Bibliography file (usually '*.bib')

%% or include bibliography directly:
% \begin{thebibliography}{}
% \bibitem{b1}
% \end{thebibliography}

\end{document}

%% file: 0macros.tex
\newcommand{\expec}{\operatorname{\mathbb{E}}}
\newcommand{\prob}{\operatorname{\mathbb{P}}}

\newcommand{\cov}{\operatorname{cov}}
\newcommand{\GP}{\operatorname{GP}}

\newcommand{\var}{\operatorname{var}}
\renewcommand{\P}{\operatorname{P}}
\renewcommand{\mathbf}{\boldsymbol}

\newcommand{\Real}{\mathbb R}
\newcommand{\reals}{\mathbb{R}}
\newcommand{\eps}{\varepsilon}
\newcommand{\eqd}{\overset{d}{=}}
\newcommand{\dto}{\xrightarrow{d}}

\newcommand{\td}{\dto}

\newcommand{\pto}{\overset{p}{\longrightarrow}}
\newcommand{\1}{\operatorname{\mathds{1}}}
\newcommand{\diff}{\mathrm{d}}
\newcommand{\oh}{\operatorname{o}}
\newcommand{\Oh}{\operatorname{O}}

\newcommand{\Km}{\hat{K}_{m}}

\definecolor{myorange}{rgb}{0.83,0.37,0.10}
\definecolor{darkteal}{rgb}{0,0.35,0.35}

\newcommand{\KP}{K_{\mathrm{P}}}
\newcommand{\KPm}{K_{\mathrm{P},m}}
\newcommand{\UP}{U_{n,\mathrm{P}}^m}
\newcommand{\UPdrie}{U_{n,\mathrm{P}}^3}

%% file: 0front.tex
\begin{frontmatter}
	
	% "Title of the Paper"
	\title{Tail inference using extreme U-statistics}
	\runtitle{Extreme U-Statistics}
	\dedicated{Dedicated to the memory of James Pickands III (1931--2022)}
	
	% indicate corresponding author with \corref{}
	% \author{\fnms{John} \snm{Smith}\thanksref{t1}\corref{}\ead[label=e1]{smith@foo.com}\ead[label=e2,url]{www.foo.com}}
	% \thankstext{t1}{Thanks to somebody} 
	% \address{line 1\\ line 2\\ \printead{e1}\\ \printead{e2}}
	
	\author{\fnms{Jochem} \snm{Oorschot}\ead[label=e1]{oorschot@ese.eur.nl}}
	\address{Econometric Institute, Erasmus University Rotterdam, 3000 DR Rotterdam, The Netherlands. \printead{e1}}
	%\and
	\author{\fnms{Johan} \snm{Segers}\ead[label=e2]{johan.segers@uclouvain.be}}
	\address{LIDAM/ISBA, UCLouvain, Voie du Roman Pays 20,
		B-1348 Louvain-la-Neuve, Belgium. \printead{e2}}
	%\and
	\author{\fnms{Chen} \snm{Zhou}\ead[label=e3]{zhou@ese.eur.nl}}
	\address{Econometric Institute, Erasmus University Rotterdam, 3000 DR Rotterdam, The Netherlands. \printead{e3}}

	\runauthor{Oorschot, Segers and Zhou}
	
	\begin{abstract}
		Extreme U-statistics arise when the kernel of a U-statistic has a high degree but depends only on its arguments through a small number of top order statistics. As the kernel degree of the U-statistic grows to infinity with the sample size, estimators built out of such statistics form an intermediate family in between those constructed in the block maxima and peaks-over-threshold frameworks in extreme value analysis. The asymptotic normality of extreme U-statistics based on location-scale invariant kernels is established. Although the asymptotic variance coincides with the one of the H\'{a}jek projection, the proof goes beyond considering the first term in Hoeffding's variance decomposition.
		We propose a kernel depending on the three highest order statistics leading to 
		a location-scale invariant estimator of the extreme value index resembling the Pickands estimator. This \emph{extreme Pickands U-estimator} is asymptotically normal and its finite-sample performance is competitive with that of the pseudo-maximum likelihood estimator.
	\end{abstract}
	
	%\begin{keyword}[class=MSC]
	%\kwd[Primary ]{}
	%\kwd{}
	%\kwd[; secondary ]{}
	%\end{keyword}
	
	\begin{keyword}
		\kwd{U-statistic}
		\kwd{Generalized Pareto distribution}
		\kwd{H\'{a}jek projection}
		\kwd{Extreme value index}
	\end{keyword}
	
	% history:
	% \received{\smonth{1} \syear{0000}}
	
\end{frontmatter}

%% file: 1intro.tex
Let $X_1, \ldots, X_n$ be independent and identically distributed random variables drawn from a common distribution $F$. Consider the generalized average
\begin{equation}
	\label{eq:Unm}
	U^{m}_n= \binom{n}{m}^{-1}  \sum_{I \subset [n], |I| = m} K_{m}(X_I),
\end{equation}
where $K_m: \Real^m\to\Real$ with $1\leq m\leq n$ is a permutation invariant function called kernel,  $[n] := \{1,\ldots,n\}$ and $X_I$ is the vector $(X_i)_{i \in I}$ indexed by the subset $I$. The generalized average $U_n^m$ is called a U-statistic and is an estimator of its corresponding ``kernel parameter'' $\theta_m$,
\[ 
\theta_m 
= \expec [ U_{n}^m ]
= \expec [ K_m(X_1,\ldots, X_m) ].
\]  
By efficiently exploiting the information in the sample, U-statistics corresponding to kernels that are not linked to a particular parametric model can attain uniformly minimum variance among all unbiased estimators of $\theta_m$,  see \cite{halmos1946theory} for the earliest theoretical development and \cite{lee1990Ustat},  \cite[Chapter~5]{serfling2009approximation} or \cite[Chapter~12]{van2000asymptotic} for an overview of U-statistics.

In this work we adapt the theory of U-statistics to the setting of extreme value analysis, which is concerned with the tail of a distribution function. If the block size $m$ is held fixed, as in the classic setup of U-statistics, there are many blocks containing observations that cannot be considered as part of the tail of the distribution. Therefore, we require that the block size $m$ tends to infinity as $n \to \infty$ and we let the kernel function $K_m$ depend only on the top-$q$ observations in a block of $m$ observations, i.e., there is $K: \{ (x_1,\ldots,x_q) \in \reals^q : x_1 \ge \ldots \ge x_q \} \to \reals$ satisfying
\begin{equation}
	\label{eq:KmK}
	K_{m}(X_1,\ldots, X_m)=K( X_{m:m}, \ldots, X_{m-q+1:m}),
\end{equation}
where $X_{m:m } \geq X_{m-1:m} \geq \ldots \geq X_{1:m}$ denote the order statistics of $X_1, \dots, X_m$.
Note that while $m$ grows, the number of upper order statistics $q$ and the kernel $K$ on $\reals^q$ remain fixed. U-statistics corresponding to kernel functions as in \eqref{eq:KmK} where $m:=m(n)\to \infty$ and $m=\oh(n)$ as $n \to \infty$ are called \emph{extreme U-statistics}.
The likelihood equation for the all-block maxima estimator in \citep[eq.~(2.3)]{oorschot2020all} can be seen as a particular instance of an extreme U-statistic with $q=1$.

Following usual practice in extreme value analysis, we assume that the observations are drawn from a common distribution function $F$ in the domain of attraction (DoA) of a Generalized Extreme Value (GEV) distribution with parameter $\gamma \in \Real$. For such $F$, there exist sequences $a_n$ and $b_n$ such that
\begin{align}\label{DoA}
	\lim_{n \to \infty } F^{n}(a_nx+b_n)=G_{\gamma}(x)=\exp(-(1+\gamma x)^{-1/\gamma})
\end{align}
for $x \in \Real$ such that $1+\gamma x > 0$; notation $F \in \mathcal{D}(G_\gamma)$.
The parameter $\gamma$ is called the extreme value index and governs the tail behaviour of $F$. We will show that extreme U-statistics can be valid estimators of $\lim_{m \to \infty} \theta_m=\mu(\gamma)$ for some function $\mu$. If the function $\mu$ is invertible, we can use the extreme U-statistic to construct estimators of $\gamma$. 

Extreme U-statistics corresponding to location-scale invariant kernels were first explored in \cite{segers2001}. It was shown there that under suitable conditions the extreme U-statistic is consistent for $\mu(\gamma)$. In this work, we study the asymptotic distribution of extreme U-statistics and we will show that, like classic U-statistics \citep{hoeffding1948class}, they are asymptotically normally distributed.
The theory is subdivided into two parts: in Section~\ref{sec:GP}, to focus on the asymptotic variance, we treat the case of an independent random sample drawn from the Generalized Pareto (GP) distribution. In Section~\ref{DoA_bias}, our main result, Theorem~\ref{Main-XU-result}, handles random samples drawn from a distribution $F$ in the DoA of a GEV distribution, in which case an asymptotic bias term arises.
In Theorem~\ref{Prop:bias}, we fix $q=3$, the minimum number of top order statistics required to obtain a location-scale invariant kernel, and we provide easily verifiable kernel integrability conditions under which Theorem~\ref{Main-XU-result} applies. 

As an application of Theorem~\ref{Prop:bias}, we consider in Section~\ref{Application} a location-scale invariant kernel $\KP:\reals^3 \to \reals$ that we call the \emph{Pickands kernel} in view of its resemblance to the Pickands estimator \citep{pickands1975statistical}.
The corresponding extreme U-statistic $\UP$ is called the \emph{extreme U-Pickands estimator} and is shown to be an asymptotically normal estimator of~$\gamma$. 
In a simulation study, we find that $\UP$ is competitive with the 
Maximum Likelihood (ML) estimator 
in the peaks-over-threshold approach based on the Generalized Pareto (GP) distribution~\citep{davison1990models}.

A technical challenge for extreme-U statistics is that their asymptotic distribution cannot be obtained in the same way as for classical U-statistics. A new approach is needed, which constitutes our main theoretical contribution. Recall that, classically, one truncates the variance of the U-statistic to the first term of the Hoeffding variance decomposition \citep[equation~(5.12)]{hoeffding1948class} and exploits that the U-statistic's H\'ajek projection has the same variance. In contrast, 
in Proposition~\ref{Prop:variance}
we need to take into account more than the first term of the Hoeffding decomposition to deal with the asymptotic variance of the extreme U-statistic.

Estimators derived from extreme U-statistics are more efficient in the mean square error sense than estimators based on other block techniques in extreme value analysis. In the context of extreme value analysis, block techniques are often used to extract large observations, for example, in the classical block maxima approach, which has received renewed interest in, for instance, \citep{dombry2019maximum, ferreira2015block}. In addition, recent studies have proposed estimators based on the sliding block technique \citep{buecher2018, robert2009sliding}. To make the comparison, note that $U^{m}_{n} = \expec[ K_m(X_1,\ldots,X_m) \mid X_{1:n},\ldots,X_{n:n}]$, see for instance \cite[p.~161]{van2000asymptotic}. This means that the conditional expectation of disjoint and sliding block statistics given the order statistics is equal to $U^{m}_{n}$, while, by Jensen's inequality for conditional expectations, the variance of $U_n^m$ is not higher.

In Section~\ref{sec:GP}, we develop the asymptotic theory of extreme U-statistics based on samples of the GP distribution. The extension to samples from a distribution in the DoA of some GEV distribution is treated in Section~\ref{DoA_bias}. The Pickands kernel leading to a novel estimator of the extreme value index is constructed in Section~\ref{Application}, the finite-sample performance of which is reported in Section~\ref{simulation}. After a brief discussion in Section~\ref{sec:disc}, the proofs of the results are given in a series of appendices.

%% file: 2asygp.tex
\subsection{Conditions and theorem}

In Theorem~\ref{main_GP_thm} we state the asymptotic normality of the extreme U-statistic 
\begin{equation}
	\label{U-stat_GP}
	U^{m}_{n, Z} = \binom{n}{m}^{-1}  \sum_{I \subset [n], |I| = m} K_{m}(Z_I),
\end{equation}
for $K_m$ as in \eqref{eq:KmK} and for
independent random variables $Z_1,\ldots,Z_n$ drawn from the $\GP(\gamma)$ distribution, i.e., $\P(Z_i \le z) = 1 - (1+\gamma z)^{-1/\gamma}$ for $z \ge 0$ such that $1 + \gamma z > 0$, with the limiting interpretation $\P(Z_i \le z) = 1 - \exp(-z)$ if $\gamma = 0$. Usually, the GP family also includes a location and scale parameter, but these will play no role here since the kernel $K$ will be assumed to be location-scale invariant.
We start by listing the required conditions on the kernel $K$ and the degree sequence $m:=m_n$.

Condition~\ref{Condition_kernel_scale-loc_integrability} is an integrability condition on the kernel $K$ to be checked by analysis. Condition~\ref{degree_sequence} provides upper and lower bounds for the block sizes $m$. Both are rather weak and are satisfied for $m$ of the order $n^\kappa$ for any $0 < \kappa < 1$.

\begin{condition}
	\label{Condition_kernel_scale-loc_integrability} 
	The kernel $K$ in \eqref{eq:KmK} is location-scale invariant and not everywhere constant, implying $q \geq 3$. Moreover, for all $p>0$ it holds that
	\[
	\expec\bigl[ \bigl| 
	K\bigl((Z_{q-j:q-1})_{j=1}^{q}\bigr)
	\bigr|^{p} \bigr] 
	< \infty,
	\]
	for independent $\GP(\gamma)$ random variables $Z_1,\ldots,Z_{q-1}$ and $Z_{0:q-1}:=0$.
\end{condition}

\begin{condition}
	\label{degree_sequence}
	The degree sequence $m:=m_n$ satisfies $m=\oh(n^{1-\epsilon})$ and $\ln(n)=\oh(m^{1-\delta})$ as $n \to \infty$ for some $\epsilon, \delta>0$.
\end{condition}

Because of the threshold-stability property of the GP distribution, the value of $\theta_m = \expec[ K_m(Z_1,\ldots,Z_m) ]$ does not depend on $m \ge q$, see \eqref{def: expectation}; we write $\theta = \theta_m$ in the remainder of this section.

\begin{theorem} 
	\label{main_GP_thm}
	If Conditions \ref{Condition_kernel_scale-loc_integrability}--\ref{degree_sequence} are satisfied, then the extreme U-statistic $U_{n, Z}^{m}$ in~\eqref{U-stat_GP} satisfies
	\[
	\sqrt{k} \left( U_{n, Z}^m- \theta  \right) 
	\dto N \left(0, \sigma^2_{K}(\gamma) \right),
	\qquad n \to \infty,
	\]
	where $k = n / m$ and $\sigma^2_{K}(\gamma)$ depends only on $K$ and $\gamma$ and is defined in~\eqref{eq:sigmaK} below.
\end{theorem}

\subsection{Background theory and sketch of the proof of Theorem~\ref{main_GP_thm}}

We start with a sketch of the derivation of the asymptotic distribution of a U-statistic in the classical case where $m$ is fixed \citep{serfling2009approximation,van2000asymptotic}, which will provide guidance in the case where $m \to \infty$. 
It will turn out that the same proof technique can be applied to the case where $m \to \infty$ but $m^2 = \oh(n)$ as $n \to \infty$, leading to the asymptotic distribution of the extreme U-statistics. By contrast, if $m^2 = \oh(n)$ does not hold, a novel approach is required: in the Hoeffding variance decomposition, a growing number of terms needs to be taken into account.
The reader only interested in the main results can skip this subsection and jump immediately to Section~\ref{DoA_bias}.

\paragraph{Classical theory for fixed $m$.}

By Hoeffding's decomposition,
\begin{equation}
	\label{eq:Hoeffding}
	\var U_{n,Z}^m = \sum_{\ell=1}^m p_{n,m}(\ell) \, \zeta_{m,\ell}
	\quad \text{with } 
	\left\{
	\begin{array}{r@{\;=\;}l}
		p_{n,m}(\ell) & \displaystyle \binom{n}{m}^{-1} \binom{m}{\ell}\binom{n-m}{m-\ell}, \\[1em]
		\zeta_{m,\ell} & \cov \bigl( K_m(X_S), K_m(X_{S'}) \bigr),
	\end{array}
	\right.
\end{equation}
where $S$ and $S'$ are sets of positive integers with $|S| = |S'| = m$ and $|S \cap S'| = \ell$.
For fixed $m$, we get from \eqref{eq:Hoeffding} the asymptotic relation
\begin{align}\label{combinatory_weight}
	\frac{p_{n,m}(\ell+1)}{ p_{n,m}(\ell)}=\Oh \left(n^{-1} \right), \qquad \text{ as } n \to \infty.
\end{align} Provided $\zeta_{m, 1}\neq 0$, it follows that the first term of Hoeffding's decomposition, $p_{n, m}(1)\zeta_{m, 1}$, asymptotically dominates the value of the sum $ \sum_{\ell=1}^m p_{n,m}(\ell) \zeta_{m,\ell}$. For $p_{n,m}(1)$, we find
\begin{equation}
	p_{n,m}(1)	= \frac{m! (n-m)!}{n!} \cdot m \cdot \frac{(n-m)!}{(n-2m+1)! (m-1)!} 
	= \frac{m^2}{n} a_{m,n} \label{first_weight_p}
\end{equation}
where
\[
a_{m,n} 
= \frac{(n-m)!}{(n-1)!} \frac{(n-m)!}{(n-2m+1)!}
= \frac{(n-m) (n-m-1) \cdots (n-2m+2)}{(n-1)(n-2) \cdots (n-m+1)}.
\]
Because the kernel degree $m$ is fixed, we have $a_{m,n} \to 1$ as $n \to \infty$. It follows that
\begin{align}\label{var_approximation}
	\var U_{n,Z}^m \sim \left(m^2/n \right) \zeta_{m, 1}, \qquad \text{ as } n \to \infty,
\end{align}
where $x_n \sim y_n$ signifies that the ratio $x_n /y_n$ converges to $1$ as $n \to \infty$.

Recall that $\theta=\expec[ K_m(Z_1,\ldots, Z_m) ]$. 
Define the H\'ajek projection of $U_{n, Z}^m - \theta$ as
\begin{align}
	\hat{U}_{n, Z}^m
	&:= \sum_{j=1}^{n}  \sum_{I \subset [n], |I| = m}  \binom{n}{m}^{-1}   \expec \left[   K_m(Z_I)  -\theta \mid Z_j  \right] \nonumber \\
	&= \sum_{j=1}^{n}  \binom{n-1}{m-1} \binom{n}{m}^{-1}    \Km(Z_{j}) %\nonumber \\ &= 
	= \frac{1}{k} \sum_{j=1}^{n} \Km(Z_{j})
	\label{eq:Hajek}
\end{align}
where $k = n/m$ and
\[
\Km(x)= \expec \left[  K_m(Z_1, \ldots, Z_{m-1}, x) \right] -\theta.
\]
Because 
the asymptotic variance of ${U}_{n, Z}^m$, i.e., $(m^2/n) \zeta_{m, 1}$, is %exactly 
equal to the variance of $\hat{U}_{n, Z}^m$, the asymptotic distribution of $U_{n, Z}^m-\theta$ is the same as that of $\hat{U}_{n, Z}^m$, i.e., as $n \to \infty$, $$\frac{\hat{U}_{n, Z}^m}{\sqrt {\var \hat{U}_{n,Z}^m}}-\frac{U_{n, Z}^m-\theta}{\sqrt{\var U_{n,Z}^m}} \overset{p}{\to} 0.$$ The H\'ajek projection $ \hat{U}_{n,Z}^m$ is asymptotically normal by the central limit theorem, and thus 
\begin{align*}
	\frac{k}{\sqrt{n \zeta_{m,1}}} \left( U_{n, Z}^m-\theta \right) \td N(0, 1), \qquad n \to \infty,
\end{align*}
provided that the kernel is square-integrable and $\zeta_{m, 1}>0$. For more details, see the proof of Theorem~12.3 in \cite{van2000asymptotic}. Note again that we assumed that $m$ is fixed.

\paragraph{Letting $m \to \infty$: first attempt.}

We aim to adapt the above proof to an extreme U-statistic, that is, when $m \to \infty$ in such a way that $m=\oh(n)$ as $n \to \infty$. Such an adaptation is non-trivial because $\zeta_{m, 1}$ varies as $m\to\infty$ and the combinatorial numbers $p_{n, m}(\ell)$ no longer satisfy equation~\eqref{combinatory_weight} in general, which makes it difficult to establish the key variance approximation in \eqref{var_approximation}. 

If we consider again the first term $p_{n,m}(1)$ in \eqref{first_weight_p}, we need a further restriction on the range of $m$ to ensure $a_{m,n}\to 1$ as $n\to\infty$. Clearly,
\begin{equation}
	\label{eq:amn1}
	\left( 1 - \frac{m-1}{n-m+1} \right)^{m-1}
	\le a_{m,n}
	\le \left(1 - \frac{m-1}{n-1}\right)^{m-1}.
\end{equation}
The lower and upper bound tend to $1$ if and only if $m^2 = \oh(n)$ as $n \to \infty$. The condition $m^2 = \oh(n)$ is also necessary to ensure that $p_{n,m}(1)$ is the dominant term in the sum $ \sum_{\ell=1}^m p_{n,m}(\ell)$ as $n \to \infty$. Taken together, these facts lead to the variance approximation~\eqref{var_approximation} but only under the additional restriction $m^2 = \oh(n)$ as $n \to \infty$.

\paragraph{Case $m \to \infty$: novel approach.}

We overcome the restriction $m^2=\oh(n)$ and obtain the desired asymptotic equivalence relation in \eqref{var_approximation} by further analyzing the subsequent terms $\zeta_{m,\ell}$ for $\ell>1$ in Hoeffding's decomposition~\eqref{eq:Hoeffding}. For integers $a$ and $b$ such that $a \le b$ write $[a{:}b] = [a, b] \cap \mathbb{Z} = \{a, \ldots, b\}$.

Without loss of generality, let the two blocks in \eqref{eq:Hoeffding} be $S = [1{:}m]$ and $S' = [(m-\ell+1){:}(2m-\ell)]$. Their overlap $S \cap S' = [(m-\ell+1){:}m]$ contains exactly $\ell$ elements.
It follows that, for $\ell \in [1{:}m]$,
\[
\zeta_{m,\ell} = \expec \left[ 
K_m(Z_{[1:m]}) \,
K_m(Z_{[(m-\ell+1):(2m-\ell)]})
\right] - \theta^2.
\]

The dependence between $K_m(Z_S)$ and $K_m(Z_{S'})$ stems from observations $Z_i$ with $i \in S \cap S'$ in the top-$q$ of both $Z_S$ and $Z_{S'}$. If $\ell$ is small compared to $m$, then usually there will be at most one such observation. We can thus expect that $\zeta_{m,\ell}$ can be approximated by $\ell \zeta_{m,1}$.
Recognizing the term $p_{n,m}(\ell)$ as the probability mass function of a hypergeometric random variable, we use the expectation formula to find
\[
\sum_{\ell=1}^m p_{n,m}(\ell) \ell \zeta_{m,1} 
= \frac{m^2}{n} \zeta_{m,1} 
= \frac{m \zeta_{m,1}}{k}.
\]
We obtain the bound
\begin{equation}
	\label{eq:approxvarUnm}
	\left| \var U_{n, Z}^{m} - \frac{m \zeta_{m,1}}{k} \right|
	\le
	\sum_{\ell=1}^m p_{n,m}(\ell) 
	\left| \zeta_{m,\ell} - \ell \zeta_{m,1} \right|.
\end{equation}

In order to formally derive \eqref{var_approximation} as well as the asymptotic distribution of the extreme U-statistic $U^m_{n, Z}$, we show three limit relations. First, we show that the bound in \eqref{eq:approxvarUnm} is $\oh(1/k)$ as $n \to \infty$, that is,
\begin{equation}
	\label{eq:wish}
	\lim_{n \to \infty} \left| k \var U_{n, Z}^{m} - m \zeta_{m,1} \right| = 0.
\end{equation}
Second, we show that
\begin{equation}
	\label{eq:limmzetam1}
	\lim_{m \to \infty} m \zeta_{m,1} = \sigma^2 > 0.
\end{equation}
Note then that \eqref{eq:wish} and \eqref{eq:limmzetam1} imply \eqref{var_approximation} as well as
\[
\frac{\var U_{n, Z}^m}{\var \hat{U}_{n, Z}^m}
= \frac{\var U_{n, Z}^m}{(n/k^2) \zeta_{m,1}}
= \frac{k \var U_{n, Z}^m}{m \zeta_{m,1}}
\to 1, \qquad \text{ as } n \to \infty.
\]
Finally, we show the asymptotic normality of the H\'ajek projection $\hat{U}_{n, Z}^m$, which, in the extreme value setting, is a centered and reduced row sum of a triangular array of row-wise independent and identically distributed random variables.

Limit relation~\eqref{eq:wish} is the content of Proposition~\ref{Prop:variance}. Limit relation~\eqref{eq:limmzetam1} and the asymptotic normality of the H\'ajek projection $\hat{U}_{n, Z}^m$ are the content of Proposition~\ref{Prop:Hajek_distribution}. Together, Propositions~\ref{Prop:variance}~and~\ref{Prop:Hajek_distribution} lead to the main result of this section in Theorem \ref{main_GP_thm}.

\begin{proposition}\label{Prop:variance}
	If Conditions \ref{Condition_kernel_scale-loc_integrability} and \ref{degree_sequence} are satisfied, we have
	\[
	\left| \var U_{n,Z}^{m} - \frac{m}{k} \zeta_{m,1} \right|=
	\oh(1/k), \qquad n \to \infty.
	\]
\end{proposition}

Let $\overline{K}_m = K_m - \theta$ denote the centered kernel. Recall that for integer $q \ge 1$, the Erlang($q$) distribution is the one of the sum $E_1+\cdots+E_q$ of $q$ independent unit exponential random variables.

\begin{proposition}
	\label{Prop:Hajek_distribution}
	If Condition~\ref{Condition_kernel_scale-loc_integrability} holds, then
	\begin{equation}
		\label{eq:sigmaK}
		\lim_{m \to \infty} m \zeta_{m, 1}=
		\int_0^\infty
		\left( \expec \left[
		\overline{K}_{q+1} \bigl(
		(Z_{q-j:q-1})_{j=1}^q, h_\gamma(S_q/x)
		\bigr)
		\right] \right)^2 \diff x 
		:=\sigma^2_{K}(\gamma),
	\end{equation}
	where $S_q$ is an $\operatorname{Erlang}(q)$ random variable, independent of the independent $\GP(\gamma)$ random variables $Z_1,\ldots,Z_{q-1}$, and where $Z_{0:q-1}:=0$.
	If, additionally, Condition~\ref{degree_sequence} holds, the H\'ajek projection $\hat{U}_{n, Z}^m$ in~\eqref{eq:Hajek} satisfies
	\begin{align}\label{Part2}
		\sqrt{k} \, \hat{U}_{n, Z}^m \dto N\left(0, \sigma^2_{K}(\gamma) \right), \qquad n \to \infty.
	\end{align}
\end{proposition}

The proofs of Propositions~\ref{Prop:variance} and~\ref{Prop:Hajek_distribution} are long and are given in Sections~\ref{proof:variance} and~\ref{sec:proof_Hajek}, respectively. 
Together, the two propositions are the required ingredients for deriving the asymptotic distribution of $U_{n, Z}^m$ in Theorem~\ref{main_GP_thm}. 

%% file: 3asydoa.tex
\subsection{Conditions and general result}
Consider again the U-statistic $U_{n}^m$ in \eqref{eq:Unm} based on a location-scale invariant kernel,
but, in contrast to Section~\ref{sec:GP}, let $X_1,\ldots,X_n$ be a random sample with common distribution function $F$ that satisfies the DoA condition \eqref{DoA} with shape parameter $\gamma \in \reals$.
Our goal now is to prove the asymptotic normality of $U_{n}^{m}$ in full generality, not only for generalized Pareto distributed samples. 

The technique relies on a coupling construction, leading to a $\GP(\gamma)$ sample $Z_1,\ldots,Z_n$ associated to $X_1,\ldots,X_n$. The extreme U-statistic $U_{n,Z}^m$ of the $\GP(\gamma)$ sample will be close to $U_{n}^m$, up to an asymptotic bias term.

For a continuous distribution function $F$, define the tail quantile function as 
\begin{equation}
	\label{eq:tqf}
	U(y) := \inf \left\{ x \in \Real : (1-F(x))^{-1} \geq y \right\}, \qquad y > 1. 
\end{equation}
For instance, the tail quantile function of $\GP(\gamma)$ distribution is 
\begin{align}
	\label{GP_function}
	h_\gamma(y) = \int_1^y u^{\gamma-1} \diff u =
	\begin{dcases}
		(y^\gamma - 1)/\gamma & \text{if $\gamma \ne 0$,} \\
		\ln y & \text{if $\gamma = 0$.}
	\end{dcases}
\end{align}
To establish asymptotic theory, we construct the i.i.d.\ sample as 
\begin{equation}
\label{eq:Y2X}
	(X_1,\ldots,X_n)=(U(Y^*_1),\ldots,U(Y^*_n)),
\end{equation}
for independent Pareto(1) random variables $Y^*_1,\ldots,Y^*_n$, i.e., $\P(Y_i^* \le y) = 1 - 1/y$ for $y \ge 1$.

Since the distribution function $F$ is in the DoA of a GEV distribution with shape parameter $\gamma \in \reals$, there exists a measurable auxiliary function $a:(1, \infty) \to (0, \infty)$ such that
\begin{align} \label{tail_quantile_DoA} 
	\lim_{x \to \infty } \frac{U(xy)-U(x)}{a(x)}= h_{\gamma}(y), \qquad \text{ locally uniformly in $y \in (0, \infty)$,}
\end{align}
see for instance \cite[Theorem~1.1.6]{deHaanFerreira2006extreme}.
Define 
\[
(Z_1,\ldots,Z_n)=(h_\gamma(Y^*_1),\ldots,h_\gamma(Y^*_n)).
\]
for the same Pareto(1) random variables as in \eqref{eq:Y2X}.
As in Section~\ref{sec:GP}, let the U-statistic with kernel $K_m$ in \eqref{eq:KmK} corresponding to the i.i.d.\ $\GP(\gamma)$ random variables $Z_1,\ldots, Z_n$ be denoted by
\[
	U^{m}_{n,Z}
	= \binom{n}{m}^{-1} \sum_{I \subset [n], |I| = m} K_{m}(Z_I).
\]

We make the following assumptions on the tail quantile function $U$ and the intermediate sequence $k$.

\begin{condition}
	\label{derivative_tail_quantile}	
	The tail quantile function $U$ in~\eqref{eq:tqf} is twice differentiable. Furthermore, the function $A$ defined by 
	\[
	A(t)=\frac{t \, U''(t)}{U'(t)}-\gamma+1
	\]
	is eventually positive or negative for large $t$ and satisfies $\lim_{t \to \infty}A(t)=0$, while the function $|A|$ is regularly varying with index $\rho < 0$. 
\end{condition}

\begin{condition}
	\label{balance_bias}
	The sequence $k:=n/m$ satisfies $\sqrt{k} \, A(m) \to \lambda \in \mathbb{R}$ as $n \to \infty$.
\end{condition}
Condition~\ref{derivative_tail_quantile} is slightly stronger than the typical second-order condition; see \cite[Theorem~2.3.12]{deHaanFerreira2006extreme}. Condition~\ref{balance_bias} is a typical assumption in extreme value statistics leading to a finite asymptotic bias.

Recall that $I$ is a set of $m$ positive integers. Define
\begin{align*}
	\xi^I_m=\frac{K_m(X_I)-K_m(Z_I)}{A(m)}
\end{align*}
for $A$ as in Condition~\ref{derivative_tail_quantile} and simply write $\xi_m$ whenever only the marginal distribution is of concern. We find
\begin{align*}
	\sqrt{k}\left(  { U}^{m}_{n}-{ U}^{m}_{n,Z} \right)
	&=\sqrt{k}   \binom{n}{m}^{-1}  \sum_{I \subset [n], |I| = m} \left( K_m\left(X_I\right)-  K_m\left(Z_I\right) \right)  \\
	&=\sqrt{k}A(m ) \,  U^m_{n, \xi},
\end{align*}
where, for fixed $m$,
\[
U^m_{n, \xi}=\binom{n}{m}^{-1}  \sum_{I \subset [n], |I| = m} \xi^I_m
\]
is another U-statistic. 

The following condition assumes the weak convergence of the U-statistic $U^m_{n, \xi}$:
\begin{condition}
	\label{biasterm_with_limit_BK}
	As $n\to\infty$, we have $U^m_{n, \xi}\stackrel{p}{\to}B_{K}$ for some constant $B_{K}\in \Real$.
\end{condition}

\begin{theorem}
	\label{Main-XU-result}
	Suppose that $X_1,\ldots,X_n$ are i.i.d.\ random variables with a common distribution function $F$ that satisfies the DoA condition \eqref{DoA} with shape parameter $\gamma \in \reals$. Under Conditions~\ref{Condition_kernel_scale-loc_integrability}--\ref{biasterm_with_limit_BK}, 
	the corresponding extreme U-statistic, $U_{n}^m$, satisfies 
	\[
		\sqrt{k}\left(U_{n}^m  -\theta  \right) 
		\dto N \left(
			\lambda  B_{K}, \sigma^2_{K}(\gamma)
		\right),
	\]
	where $\sigma^2_{K}(\gamma)$ is as in Proposition~\ref{Prop:Hajek_distribution} and $\theta=\mu(\gamma)$ is as in Theorem~\ref{main_GP_thm}.
\end{theorem}

\begin{proof}[Proof of Theorem \ref{Main-XU-result}]
Write
\begin{align*}
	\sqrt{k} \left( U_{n}^m - \theta \right)
	&= \sqrt{k}\left( U_n^m-U_{n, Z}^m \right)
	+ \sqrt{k}\left( U^m_{n, Z} -\theta \right)\\
	&= \sqrt{k}A(m ) \,  U^m_{n, \xi}
	+ \sqrt{k}\left( U^m_{n, Z} -\theta \right)\\
	&=: I_1+I_2.
\end{align*}
With Condition~\ref{Condition_kernel_scale-loc_integrability}  and~\ref{degree_sequence}, we can apply Theorem \ref{main_GP_thm} to get that $I_2\stackrel{d}{\to}N \left(0, \sigma^2_{K}(\gamma)\right)$. Condition~\ref{derivative_tail_quantile}-\ref{biasterm_with_limit_BK} imply that $I_1\dto\lambda  B_{K}$, which completes the proof.
\end{proof}

Note that $\sigma^2_{K}(\gamma)$ does not depend on the distribution of $X$. Therefore, a practical consequence of Theorem \ref{Main-XU-result} is that, for kernels $K$ satisfying the required conditions, one can do a parametric bootstrap using $\GP(\hat{\gamma})$ distributed random variables to obtain an estimate of $\sigma^2_{K}(\gamma)$, where $\hat{\gamma}$ can be any reasonable estimator of $\gamma$, provided that $\sigma^2_{K}(\gamma)$ is a continuous function of $\gamma$.

To make use of Theorem \ref{Main-XU-result}, one needs to verify that the kernel $K$ and the tail quantile function $U$ satisfy Condition~\ref{biasterm_with_limit_BK}. In Section~\ref{Location-scale invariant kernels}, we deal with kernels where $q=3$ --- the minimum value required to obtain a location-scale invariant kernel that is not everywhere constant --- and proceed to provide verifiable sufficient conditions under which Condition~\ref{biasterm_with_limit_BK} holds, giving in addition an explicit expression for $B_K$.

\subsection{Location-scale invariant kernels where $q=3$}
\label{Location-scale invariant kernels}

For a location-scale invariant kernel $K:\{(x_1,x_2,x_3) \in \Real^3 : x_1 > x_2 > x_3\} \to \Real$, we may write 
\begin{align}
\nonumber
	K(x_1,x_2,x_3)
	&= K(x_1-x_2, 0, x_3-x_2) \\
	\nonumber &= K\left(\frac{x_1-x_2}{x_2-x_3}, 0, -1 \right) \\
\label{eq:Kg}
	&=: g\left( \ln\left(\frac{x_1-x_2}{x_2-x_3}\right) \right)
\end{align}
for some function $g: \Real \to \Real$, with $x_1>x_2>x_3$. We assume that $F$ is continuous, so that ties do not occur with probability one.
Further, note that a location-scale function of one or two real arguments is necessarily constant, which is why we consider $q = 3$.

Next we assume the existence of first-order partial derivatives of $K$, which is equivalent to the differentiability of $g$. Let $g'$ denote the derivative of $g$. We assume the following condition on $g'$.

\begin{condition}
	\label{kernel_derivatives}
	The derivative $g'$ is continuous and satisfies one of the following two properties:
	\begin{itemize}
		\item[(a)] The function $g'$ is monotone and  there exists $\epsilon>0$ such that
		\[
		\limsup _{m \geq 3 } \expec \left[ \left| 
		g'\left( 
		\ln \left(
		\frac{X_{m:m}-X_{m-1:m}}{X_{m-1:m}-X_{m-2:m}}
		\right)  
		\right) 
		\right|^{2+\epsilon} \right] 
		< \infty,
		\]
		\item[(b)] The function $g'$ is bounded.
	\end{itemize}
\end{condition}

In addition, we assume that the moments of the kernel function do not grow too quickly as $n\to\infty$.

\begin{condition}
	\label{kernel_moment}
	There exist $\epsilon>0$ and $a>0$ such that
	\begin{equation}
		\label{weaker}
		\expec \left[
			\left|K_m(X_1,\ldots,X_m)\right|^{2+\epsilon} 
		\right]
		= \Oh(m^a), 
		\qquad m \to \infty.
	\end{equation} 
\end{condition}

We obtain the following consequence of Theorem~\ref{Main-XU-result}.

\begin{theorem}\label{Prop:bias}
	For a kernel $K:\Real^3 \to \reals$, Conditions~\ref{derivative_tail_quantile}, \ref{kernel_derivatives} and~\ref{kernel_moment} imply Condition~\ref{biasterm_with_limit_BK} with
	\begin{equation}
		\label{eq:BK}
		B_K=B_{K}(\rho, \gamma) 
		:= \expec \left[ S_3^{-\rho} \right] 
		\expec \left[ 
		\sum_{j=1}^{2}
		\dot{K}_{j}\left( 
		\{h_{\gamma}\left( Y_{3-j:2}\right) \}_{j=1}^3
		\right)
		H_{\gamma, \rho}(Y_{3-j: 2})
		\right],
	\end{equation}
	where $S_3$ is an $\operatorname{Erlang}(3)$ random variable and where $H_{\gamma,\rho}$ is defined in \eqref{eq:Hgammarho}.
	
	Under Conditions~\ref{Condition_kernel_scale-loc_integrability}--\ref{balance_bias}, ~\ref{kernel_derivatives} and \ref{kernel_moment}, the extreme U-statistic $U_{n}^m$ satisfies
	\[
	\sqrt{k}\left(U_{n}^m  -\theta  \right) 
	\dto N \left(\lambda  B_{K}(\rho, \gamma), \sigma^2_{K}(\gamma)\right), 
	\qquad n \to \infty
	\]
	where $\sigma^2_{K}(\gamma)$ is the same as in Proposition~\ref{Prop:Hajek_distribution} and for $\theta=\mu(\gamma)$ as defined prior to Theorem~\ref{main_GP_thm}.
\end{theorem}

%% file: 4estim.tex
We develop a kernel depending on the top-3 observations in a block producing an extreme U-statistic that is an unbiased estimator of the shape parameter $\gamma \in \Real$ when the observations are independently sampled from a $\GP(\gamma)$ distribution.
For integer $m \ge 3$, define the permutation and location-scale invariant kernel $\KPm : \reals^m \to \reals$ by
\begin{equation}
\label{eq:KPm}
	\KPm(x_1,\ldots,x_m) = \KP(x_{m:m}, x_{m-1:m}, x_{m-2:m}),
\end{equation}
with $\KP : \{(y_1, y_2, y_3) \in \reals^3 : y_1 > y_2 > y_3 \} \to \reals$ given by
\begin{equation}
	\label{eq:K*}
	\KP(y_1, y_2, y_3) 
	= 2 \ln \left( \frac{y_1-y_2}{y_2-y_3} \right) 
	- \ln \left( \frac{y_1-y_3}{y_2-y_3} \right)
	= \ln \left( \frac{(y_1-y_2)^2}{(y_1-y_3)(y_2-y_3)} \right).
\end{equation}
Because of the resemblance with the estimator of the extreme value index defined in \cite{pickands1975statistical}, we call $\KP$ the \emph{Pickands kernel}. It consists of two terms: the first one keeps the so-called Galton ratios as small as possible, see~\cite{segers2000testing}, while the second one was considered in~\cite{segers2001}.

The extreme U-statistic based on $\KP$ has a particularly attractive property.

\begin{proposition}
	\label{unbiased_estimator}
	For $\gamma \in \reals$, integer $n \ge m \ge 3$, and independent $\GP(\gamma)$ random variables $Z_1,\ldots,Z_n$, we have
	\[
		\expec \left[ \KPm(Z_1,\ldots,Z_m)  \right]=\gamma.
	\]
	Consequently, the U-statistic 
	\begin{equation*}
		\binom{n}{m}^{-1} \sum_{I \subset [n], |I| = m} \KPm(Z_I)
	\end{equation*}
	is an unbiased estimator of $\gamma$.
\end{proposition}

For a random sample $X_1,\ldots,X_n$ from a distribution $F$ with extreme value index~$\gamma$, we need a weak additional condition on $F$ to guarantee the asymptotic normality of $\sqrt{k} \left(\UP - \gamma\right)$, where
\begin{equation}
	\label{eq:Un*m}
	\UP := \binom{n}{m}^{-1} \sum_{I \subset [n], |I| = m} \KPm(X_I)
\end{equation}
is the \emph{extreme U-Pickands estimator}.

\begin{condition}
	\label{cond:f}
	The distribution function $F$ admits a density function $f$ with support $\mathcal{E} = \{x \in \reals : f(x) > 0 \}$, say. There exist $\epsilon>0$ and $p>2$ such that
	\begin{align}
		\nonumber
		\int_{-\infty}^{-1} \left(\ln(-x)\right)^p f(x) \, \diff x 
		&< \infty, \\
		\label{conditie:CDF}
		\sup_{x\in \mathcal{E}}
		\left|F(x)-F(x+e^{-r^{1/p}}) \right|
		&= \Oh(r^{-(1+\epsilon)}), \qquad \text{ as } r \to \infty.
	\end{align}
\end{condition}

Note that \eqref{conditie:CDF} holds as soon as the density $f$ is bounded, since $F(y)-F(x) = \int_x^y f(t) \, \diff t \le (y-x) \sup_{t \in \mathcal{E}} f(t)$ for real $x < y$.

\begin{theorem}
	\label{Prop_kernel_density}
	If Conditions~\ref{degree_sequence}, \ref{derivative_tail_quantile}, \ref{balance_bias} and~\ref{cond:f} are fulfilled, then
	\[
	\sqrt{k} \left(\UP - \gamma\right) \dto
	N \left( \lambda B_{\KP}(\gamma, \rho), \sigma_{\KP}^2(\gamma) \right),
	\qquad n \to \infty,
	\]
	where the asymptotic variance $\sigma_{\KP}^2$ and the bias $B_{\KP}$ are given in equations~\eqref{eq:sigmaK} and~\eqref{eq:BK}, respectively, with $K$ to be replaced by $\KP$ defined in~\eqref{eq:K*}.
\end{theorem}

The finite-sample performance of $\UP$ as an estimator of $\gamma$ will be explored in Section~\ref{simulation} by a simulation study. For calculation purposes, it is convenient to express $\UP$ explicitly as a linear combination of the log-spacings.

\begin{lemma}
	\label{weights_estimator}
	The U-statistic $\UP$ based on a random sample $X_1,\ldots,X_n$ can be written in terms of the order statistics as
	\begin{equation}
		\label{eq:Un*mexplicit}
		\UP = \sum_{j=2}^{n-m+3}
		\frac{\binom{n-j}{m-3}}{\binom{n}{m}} 
		\left( 2 \frac{n-j+1}{m-2} - j \right) 
		\sum_{i=1}^{j-1} \ln(X_{n-i+1:n} - X_{n-j+1:n}).
	\end{equation}
\end{lemma}

For given $n$ and $m$, the coefficients on the right-hand side of \eqref{eq:Un*mexplicit} lend themselves well for recursive calculation, without the need of invoking factorials:
\begin{align*}
	\frac{\binom{n-2}{m-3}}{\binom{n}{m}}
	&= \frac{m(m-1)(m-2)}{n(n-1)(n-m+1)}, \\
	\frac{\binom{n-j}{m-3}}{\binom{n}{m}}
	&= \frac{\binom{n-(j-1)}{m-3}}{\binom{n}{m}} \cdot \frac{n-j-m+4}{n-j+1}, \qquad j \in [3{:}(n-m+3)].
\end{align*}

%% file: 5simu.tex
We compare the performance of the extreme U-Pickands estimator $\UP$ in \eqref{eq:Un*m} of $\gamma \in \Real $ to the location-scale invariant GP Maximum Likelihood (ML) estimator \citep{davison1990models} based on the $k$ excesses over the $(k+1)$-largest order statistic. While the latter is also considered an estimator of $\gamma \in \reals$, the algorithm to compute the ML estimator is theoretically constrained to $\gamma >-1$ \citep{grimshaw1993computing}, see also \citep{zhou2009existence}.
We simulate independent random samples of size $n$ from a number of distributions and for each sample we apply both estimators of $\gamma$ over a range of block sizes $m$ and threshold parameters $k$. To facilitate comparison between the two estimators we fix the number of tail observations to be used in GP ML estimation to $k = 3n/m$, where $m$ is the block size of the kernel $\KPm$. The motivation for this choice is as follows: theoretically, for estimators using the top-$3$ observations of a disjoint partitioning of the dataset in blocks of size $m$, the number of data points used is equal to $3n/m$ too; moreover, when $m=3$, both $\UP$ and the GP ML estimator use all data.

In Section~\ref{sec:simu:ABM}, we also compare the extreme U-Pickands estimator with the All Block Maxima (ABM) estimator in \cite{oorschot2020all}. This estimator extracts the maximum out of every block of size $m$ and then fits a two-parameter Fr\'echet distribution to the sample of all block maxima thus obtained; the two parameters are the shape parameter $\gamma > 0$ and the scale parameter $\sigma > 0$. The likelihood equation~(2.3) in that paper is a scale-invariant function of the block maxima. Although the estimator of $\gamma$ is not an extreme U-statistic in the sense of our paper, it is still close in spirit to it with $q = 1$.

\subsection{Samples from a GP distribution}

We start by comparing the MSE of the GP ML estimator to $\UP$ in the case that the data are sampled from the GP distribution with shape parameter $\gamma \in \{-0.5, 0, 0.5\}$. Effectively, the comparison will be on the basis of estimator variance, since, in the GP case, $\UP$ is unbiased and the GP ML estimator is asymptotically unbiased.

\begin{figure}[h]
	\caption{MSE comparison of $\UP$ and the GP ML estimator for $100$ samples of size $n=10\,000$ from $\GP(\gamma)$}
	\label{GP_MSE_plots}
	\subfloat[GP($\gamma=0.5$)]{\includegraphics[width=0.475\textwidth]{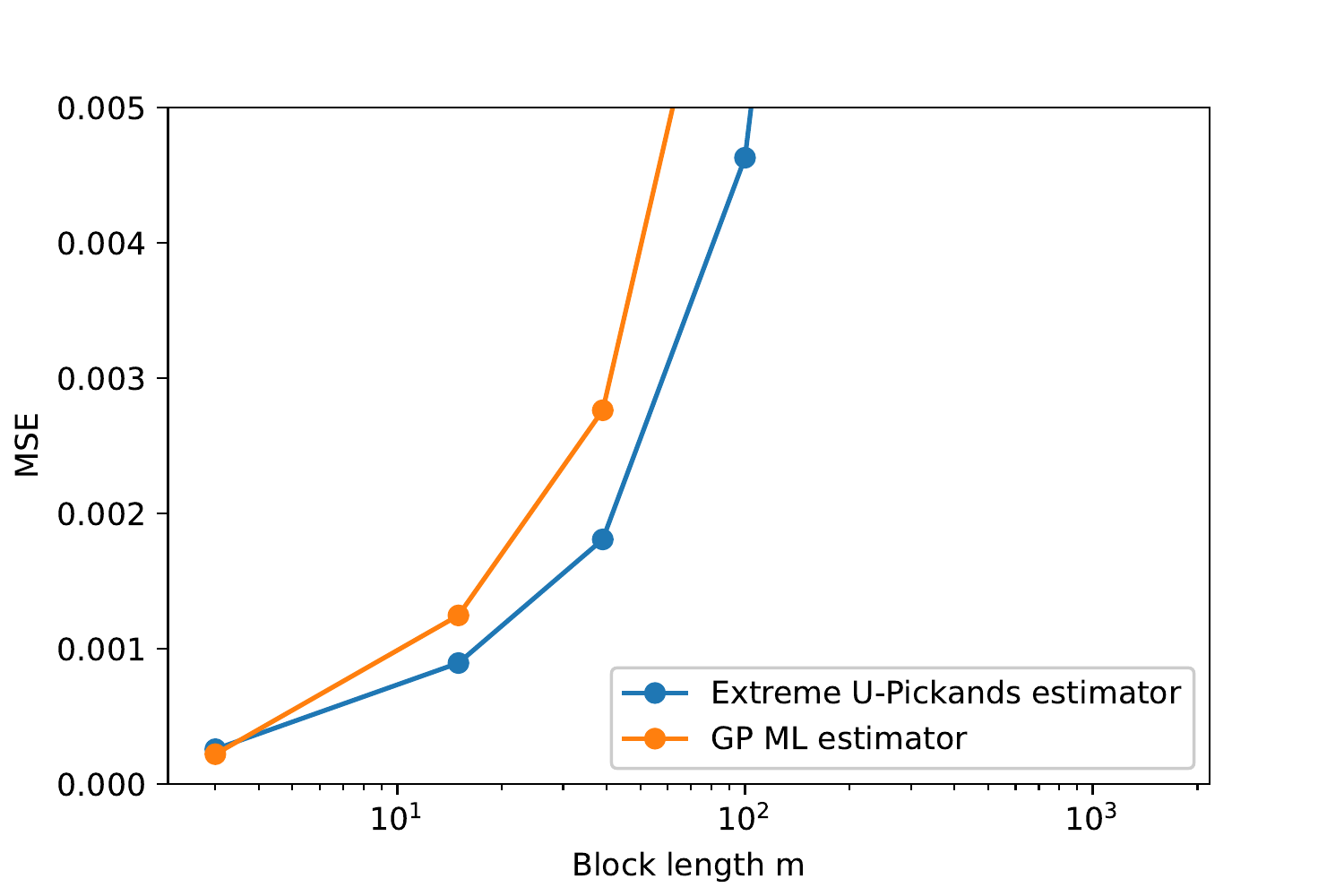}}
	\subfloat[GP($\gamma=0$)]{\includegraphics[width=0.475\textwidth]{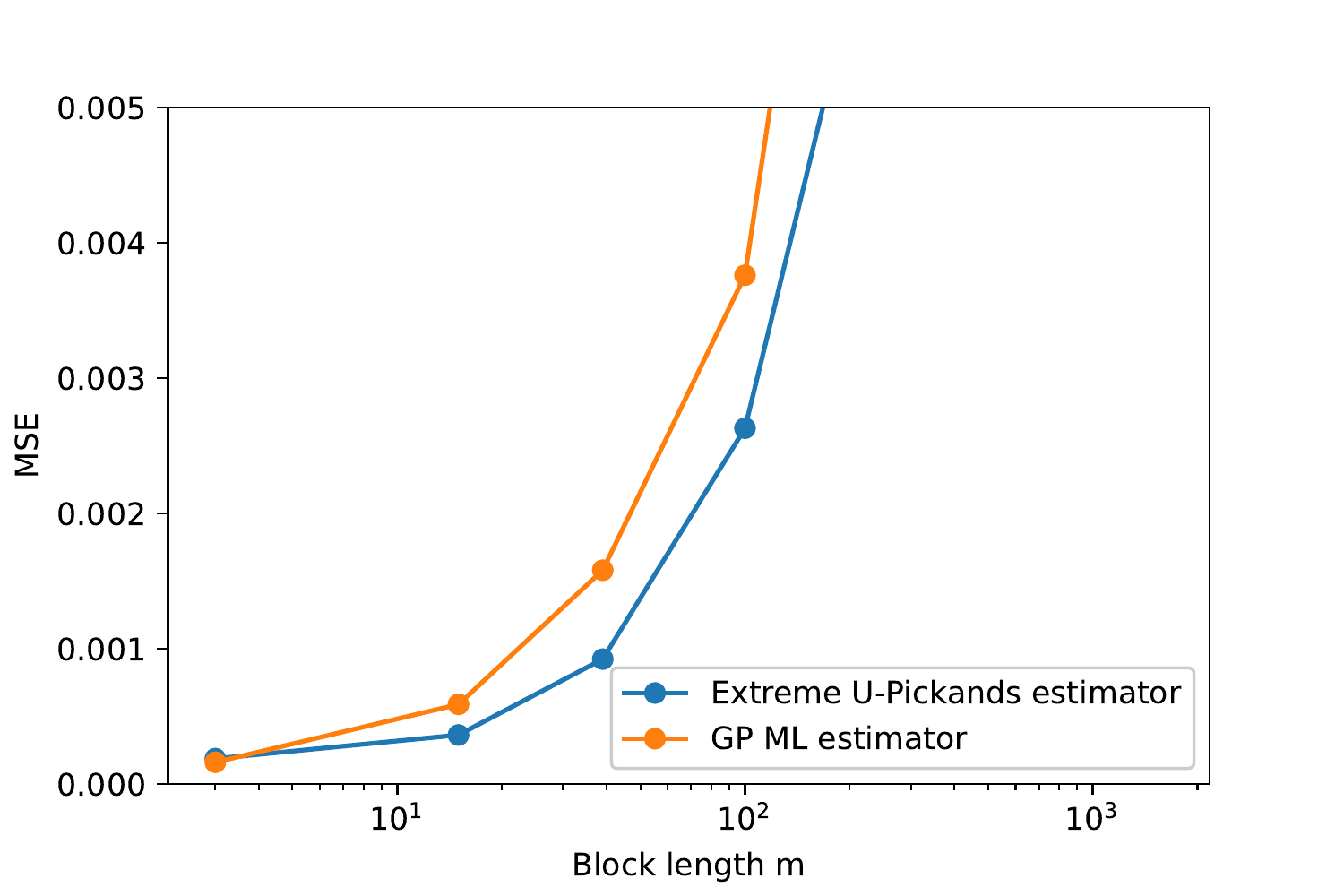}}\\
	\subfloat[GP($\gamma=-0.5$)]{\includegraphics[width=0.475\textwidth]{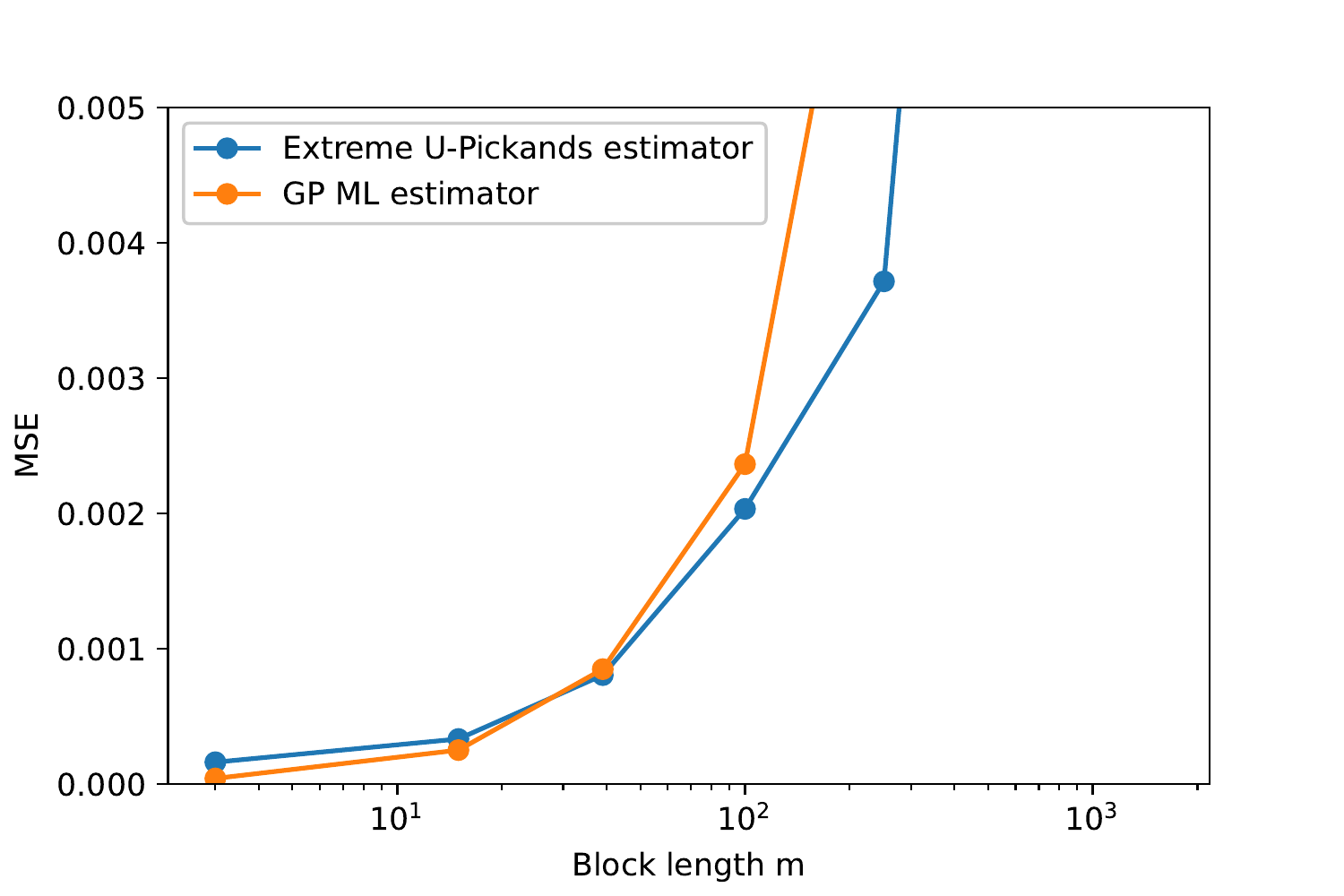}}
\end{figure}

Figure \ref{GP_MSE_plots} shows that the MSE for both estimators is minimal when we use all available observations in estimation (block length $m=3$). 
Moreover, the MSE of $\UP$ is significantly lower than that of the GP ML estimator over the largest range of the block length $m$, except when $\gamma=-0.5$. The apparent discrepancy between the two cases $\gamma \in \{0, 0.5\}$ on the one hand and the case $\gamma=-0.5$ on the other hand can be understood by a comparison of the asymptotic variance.

Recall that the GP ML estimator has asymptotic variance $(1+\gamma)^2$, provided $\gamma> -1/2$ \citep{drees2004maximum}. For the asymptotic variance of $\UP$, we fix a grid of 50 equally spaced values of $\gamma$ between $-1$ and $1$ and approximate $\sigma^2_{\KP}(\gamma)$ by the sample variance of $\UP$ based on $1000$ $\GP(\gamma)$ samples of size $n=10\,000$, multiplied with $k=n/m=100$. The normalized asymptotic variance of the GP ML estimator $(1+\gamma)^2/3$ (like before, we divide by three to facilitate comparison) together with the estimated asymptotic variance $\hat{\sigma}^2_{\KP}(\gamma)$ corresponding to $\UP$ are depicted in Figure~\ref{asymptotic_variance}.
The asymptotic variances intersect around $\gamma \approx -0.25$, with $\sigma^2_{\KP}(\gamma)$ indeed being higher for lower values of $\gamma$ but lower for all higher values of $\gamma$. In this context it is possible for the asymptotic variance $\sigma^2_{\KP}( \gamma)$ to be lower than the one of the GP ML estimator because
$\UP$ uses more observations than the GP ML estimator. For completeness the approximate values of $\sigma^2_{\KP}( \gamma)$ for $\gamma \in [-1,1]$ can be found in Table~\ref{asymptotic_variance_table}. These can be used for constructing asymptotic confidence intervals when using the extreme U-Pickands estimator, at least when the bias term is negligible.

\begin{figure}
	\caption{Normalized asymptotic variance $\sigma_{\KP}^2(\gamma)$ of $\UP$ and $(1+\gamma)^2/3$, corresponding to the one of the GP ML estimator}
	\label{asymptotic_variance}
	\centering
	\includegraphics[height=0.25\textheight, width=0.5\textwidth]{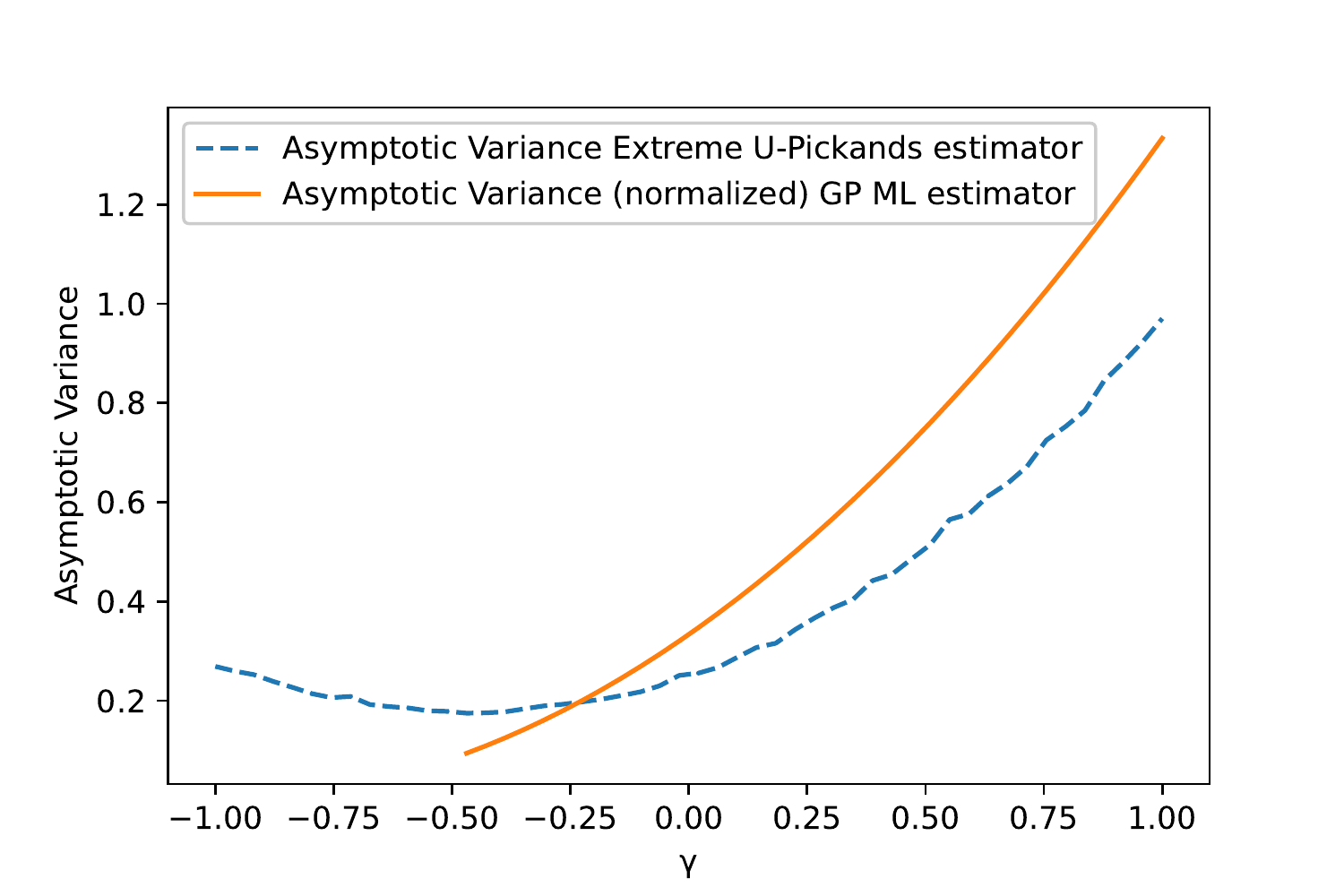}
\end{figure} 

\begin{table}
	\caption{Asymptotic variance $\sigma^2_{\KP}(\gamma)$ for $\gamma \in [-1,1]$ }\label{asymptotic_variance_table}
	\begin{adjustbox}{width=\columnwidth,center}
		\begin{tabular}{r  r | r  r| r r | r r | r r}
			\toprule
			$\gamma$ & $\sigma^2_{\KP}(\gamma)$ & $\gamma$ & $\sigma^2_{\KP}(\gamma)$ & $\gamma$ & $\sigma^2_{\KP}(\gamma)$ & $\gamma$ & $\sigma^2_{\KP}(\gamma)$ & $\gamma$ & $\sigma^2_{\KP}(\gamma)$\\
			\midrule
			$-1.000$ & 0.269 &  $-0.959$ & 0.259 & $-0.918$ &0.252 & $-0.878$ & 0.239 & $-0.837$ &0.227 \\
			$^-0.796$ &0.214 & $-0.755$ & 0.206 &  $-0.714$ & 0.208 & $-0.673$&0.192 &  $-0.633$ & 0.188  \\
			$-0.592$ & 0.185 &  $-0.551$ &0.179 & $-0.510$ &0.178 & $-0.469$ &0.175 & $-0.429$ & 0.175 \\
			$-0.388$ &0.177 & $-0.347$ & 0.183 &  $-0.306$ & 0.189& $-0.265$&0.193 &  $-0.224$ & 0.198  \\
			$-0.184$ & 0.203 &  $-0.143$ & 0.210 & $-0.102$ & 0.218 & $-0.061$ & 0.230 & $-0.020$ & 0.251 \\
			$0.020$ & 0.255 & $0.061$ &0.267 &  $0.102$ & 0.287 & $0.143$&0.307 &  $0.184$ & 0.316  \\
			$0.225$ & 0.343 &  $0.265$ & 0.366 & $0.306$ & 0.387 & $0.347$ & 0.404 & $0.388$ & 0.442 \\
			$0.429$ &0.454 & $0.469$ & 0.485 &  $0.510$ & 0.514 & $0.551$& 0.565 &  $0.592$ & 0.576  \\
			$0.633$ &0.612 &  $0.674$ &0.638 & $0.714$ & 0.672 & $0.755$ &0.725 & $0.796$ & 0.752 \\
			$0.837$ &0.785 & $0.878$ &0.845 &  $0.918$ & 0.883 & $0.959$&0.924 &  $1.000$ & 0.970  \\
			\bottomrule
		\end{tabular}
	\end{adjustbox}
\end{table}
\clearpage

\subsection{Random variables in the DoA of a GEV distribution}

Next we consider data drawn from a distribution function in the DoA of a GEV distribution. The optimal block length $m$ in a MSE sense will then not be equal to three, but will in general be a higher value that better balances bias and variance.

In Figure~\ref{MSE_DoA} we illustrate the relative performance of $\UP$ in terms of MSE for a couple of well known distributions.
The conclusions from the previous section remain; the minimal MSE of $\UP$ is competitive with that of the GP ML estimator (sometimes even lower) and the performance of $\UP$ is usually better over a large range of block lengths.

\begin{figure}
	\caption{MSE comparison of $\UP$ and the GP ML estimator at $k = 3n/m$ for $100$ samples of size $n=10\,000$ from the stated distributions, with extreme value index $\gamma \in \reals$ and second-order parameter $\rho \leq 0$ (Condition~\ref{derivative_tail_quantile}) as indicated}
	\label{MSE_DoA}
	\subfloat[Student-t($2$), $\gamma=0.5$, $\rho=-1$]{\includegraphics[height=0.25\textheight, width=0.5\textwidth]{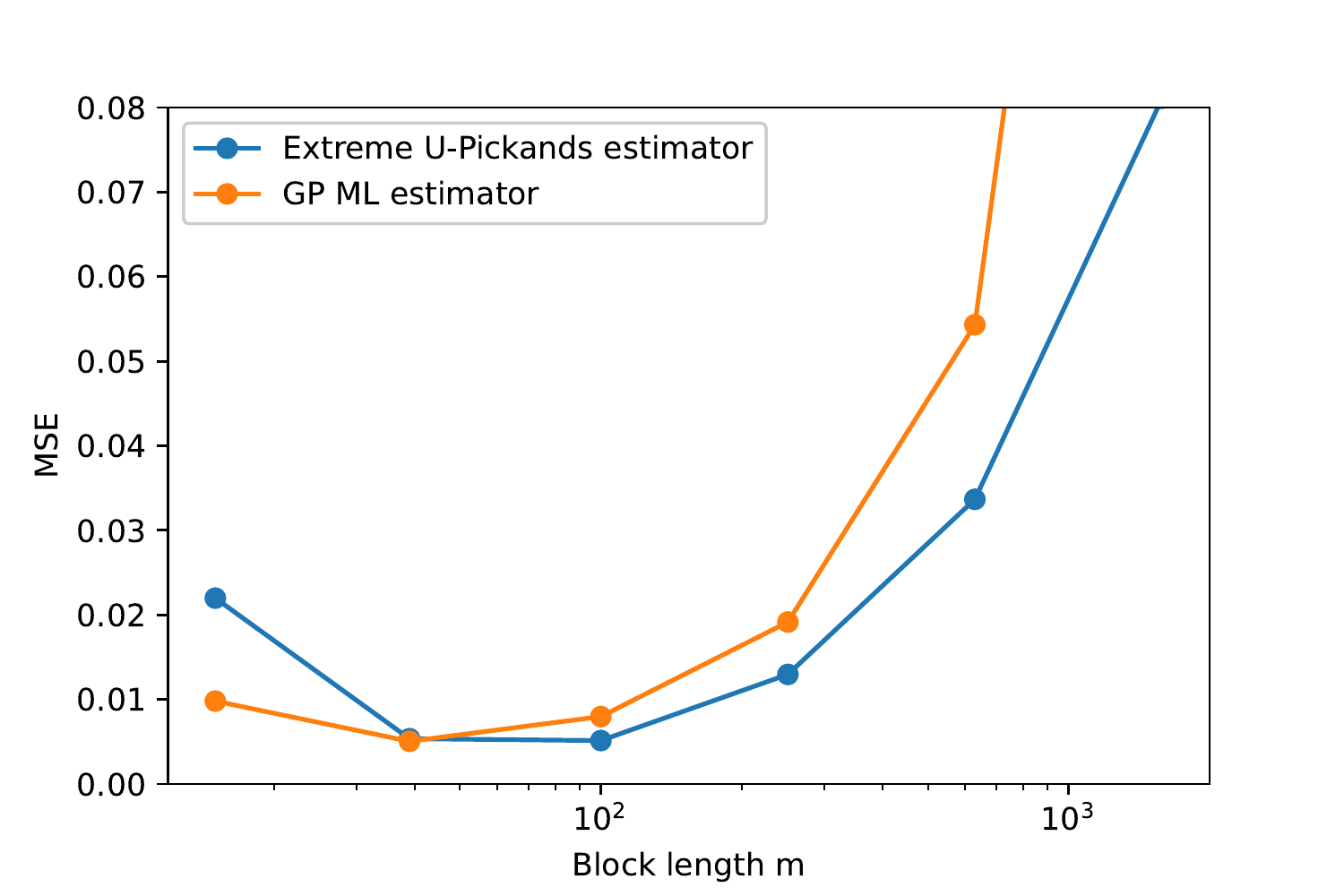}}
	\subfloat[Normal, $\gamma=0$, $\rho=0$]{\includegraphics[height=0.25\textheight, width=0.5\textwidth]{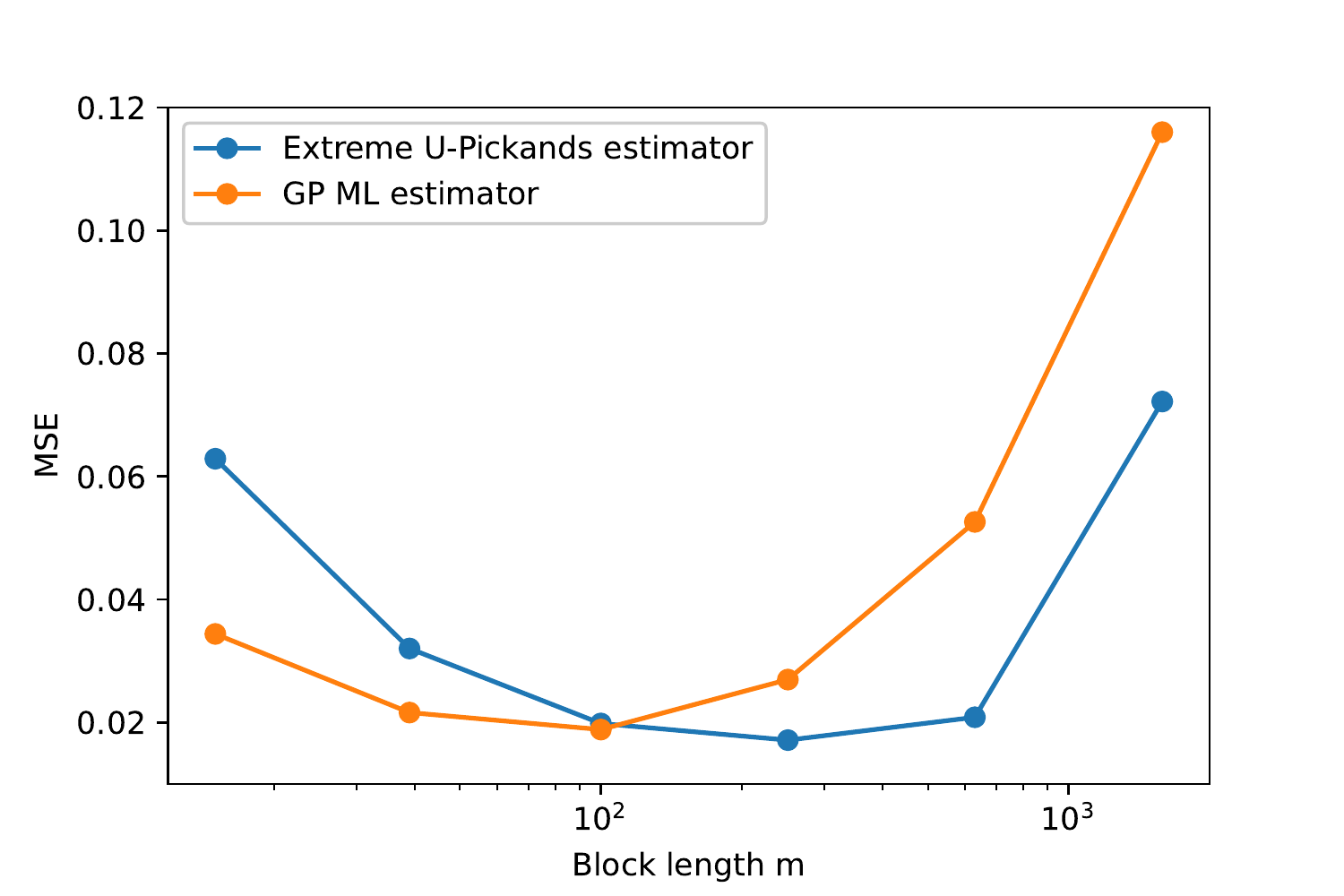}}\\
	\subfloat[Beta($2,2$), $\gamma=-0.5$ ]{\includegraphics[height=0.25\textheight, width=0.5\textwidth]{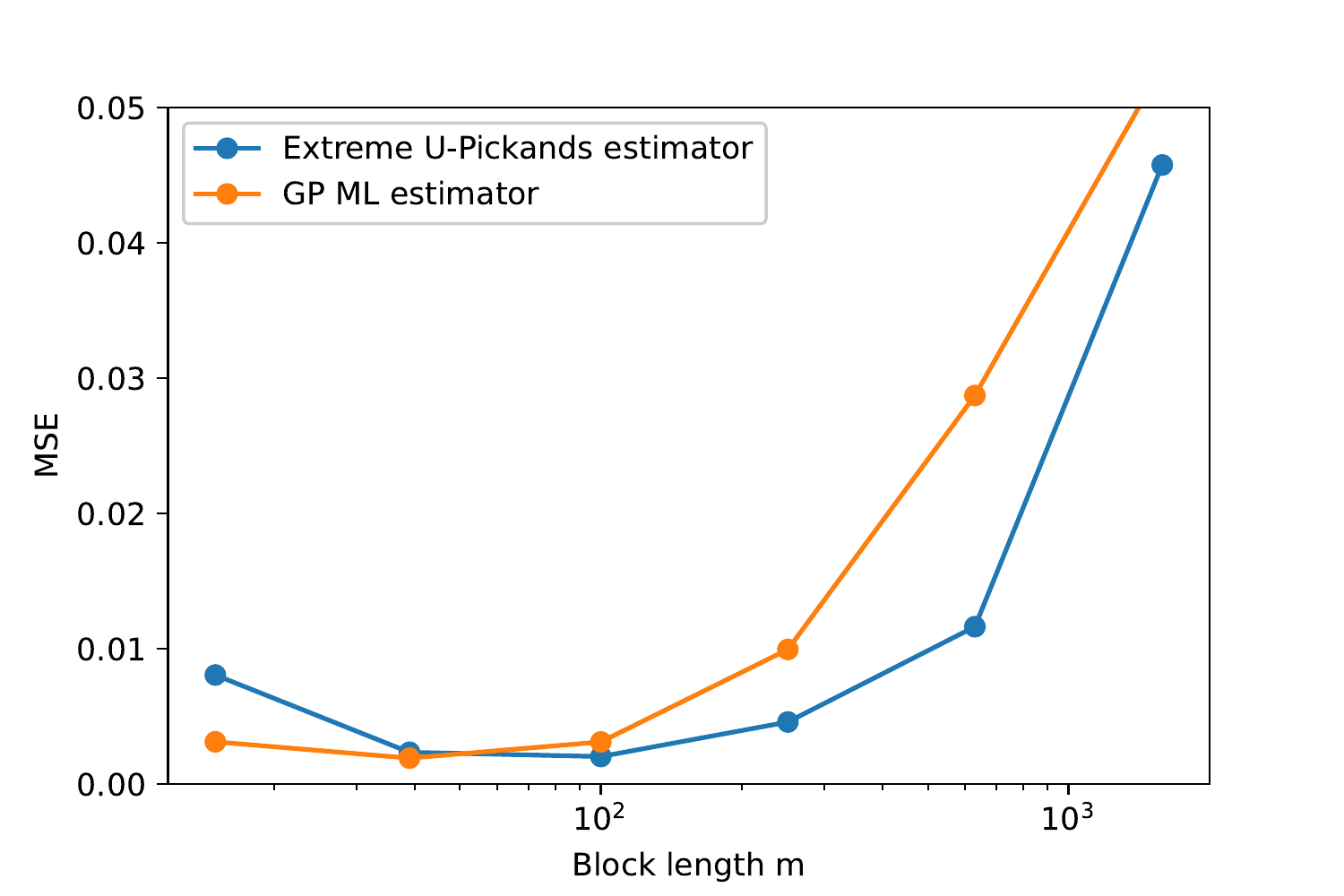}}
	\subfloat[Frech\'et, $\gamma=1/3$, $\rho=-1$]{\includegraphics[height=0.25\textheight, width=0.5\textwidth]{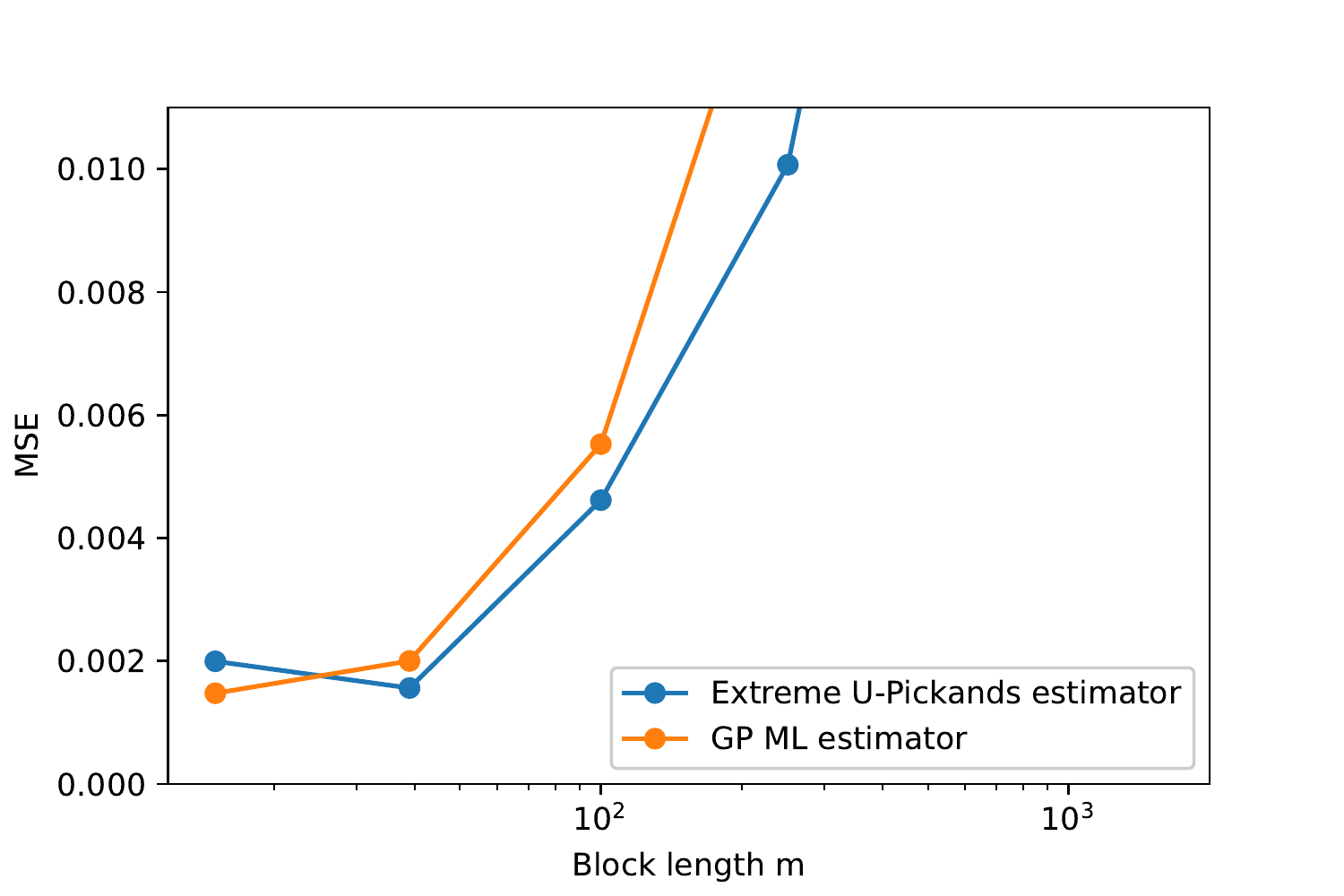}}
\end{figure}

MSE comparisons for observations from a distribution in the DoA of a GEV distribution effectively net out the competing effects of changes in estimator variance (depending on $\gamma$) to changes in estimator bias (depending on the second-order parameter $\rho$ as well as on $\gamma$). For many distributions, $\gamma$ and $\rho$ cannot be controlled separately. This is why we study the Burr distribution to disentangle the effects of $\rho$ and $\gamma$. The Burr($\lambda, \eta$) distribution has distribution function $F(x)=1-\left( 1/(1+x^{\eta}) \right)^{\lambda}$ for $\eta, \lambda, x>0$, with extreme value index $\gamma=(\lambda \eta)^{-1}$ and second-order parameter $\rho=-\lambda^{-1}$.

\begin{figure}
	\caption{Bias of $\UP$ and the GP ML estimator based on $1000$ samples of size $n=10\,000$ from the Burr distribution with stated extreme value index $\gamma$ and second-order parameter $\rho$. Panel~(b) provides a zoom of panel~(a) in the range $-1 \le \rho \le -0.25$.} \label{Bias_Burr}
	\subfloat[$\gamma=0.5$, $\rho \in {[-2, -0.25 ]}$ ]{\includegraphics[height=0.25\textheight, width=0.5\textwidth]{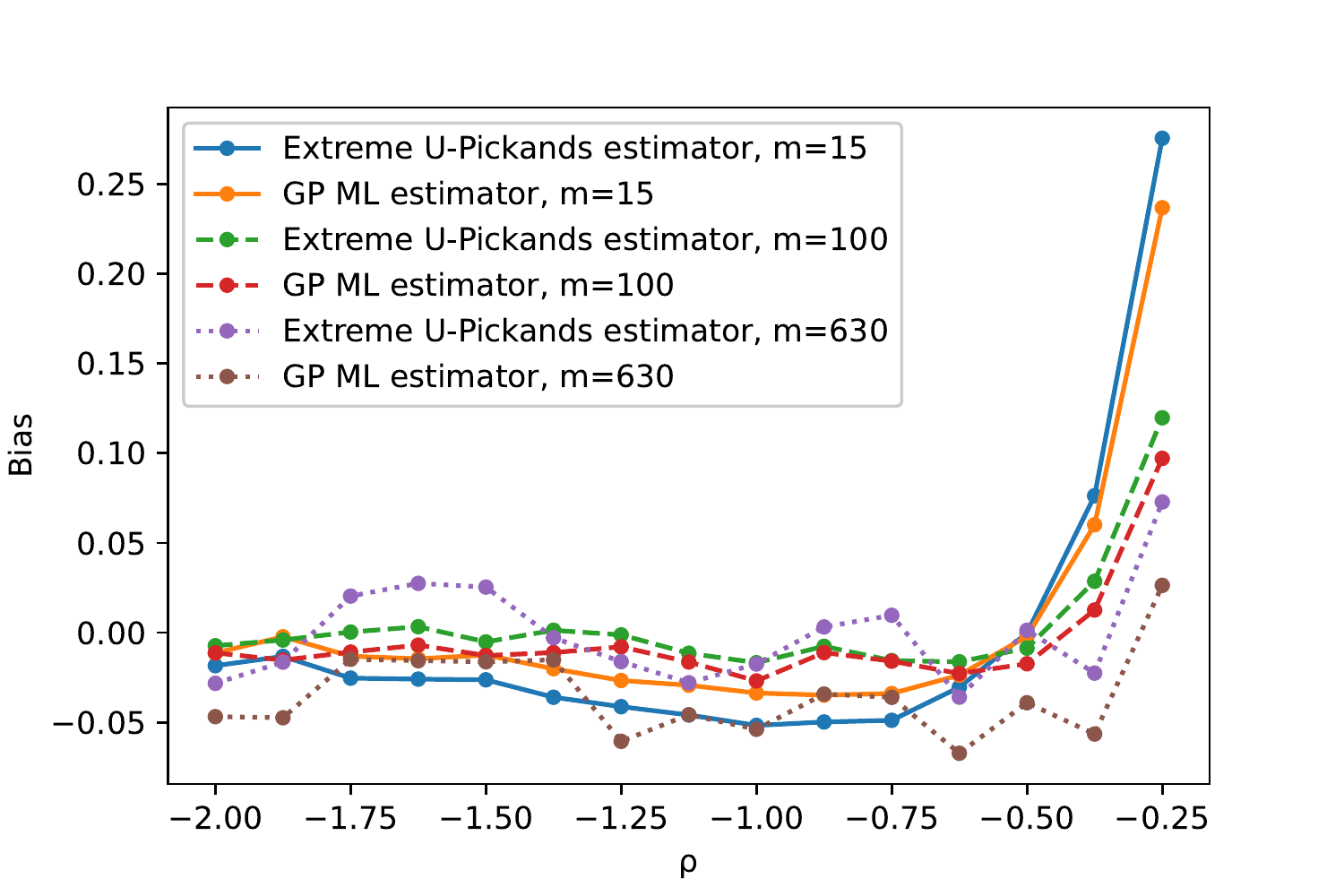}}
	\subfloat[$\gamma=0.5$, $\rho \in {[-1, -0.25 ]}$]{\includegraphics[height=0.25\textheight, width=0.5\textwidth]{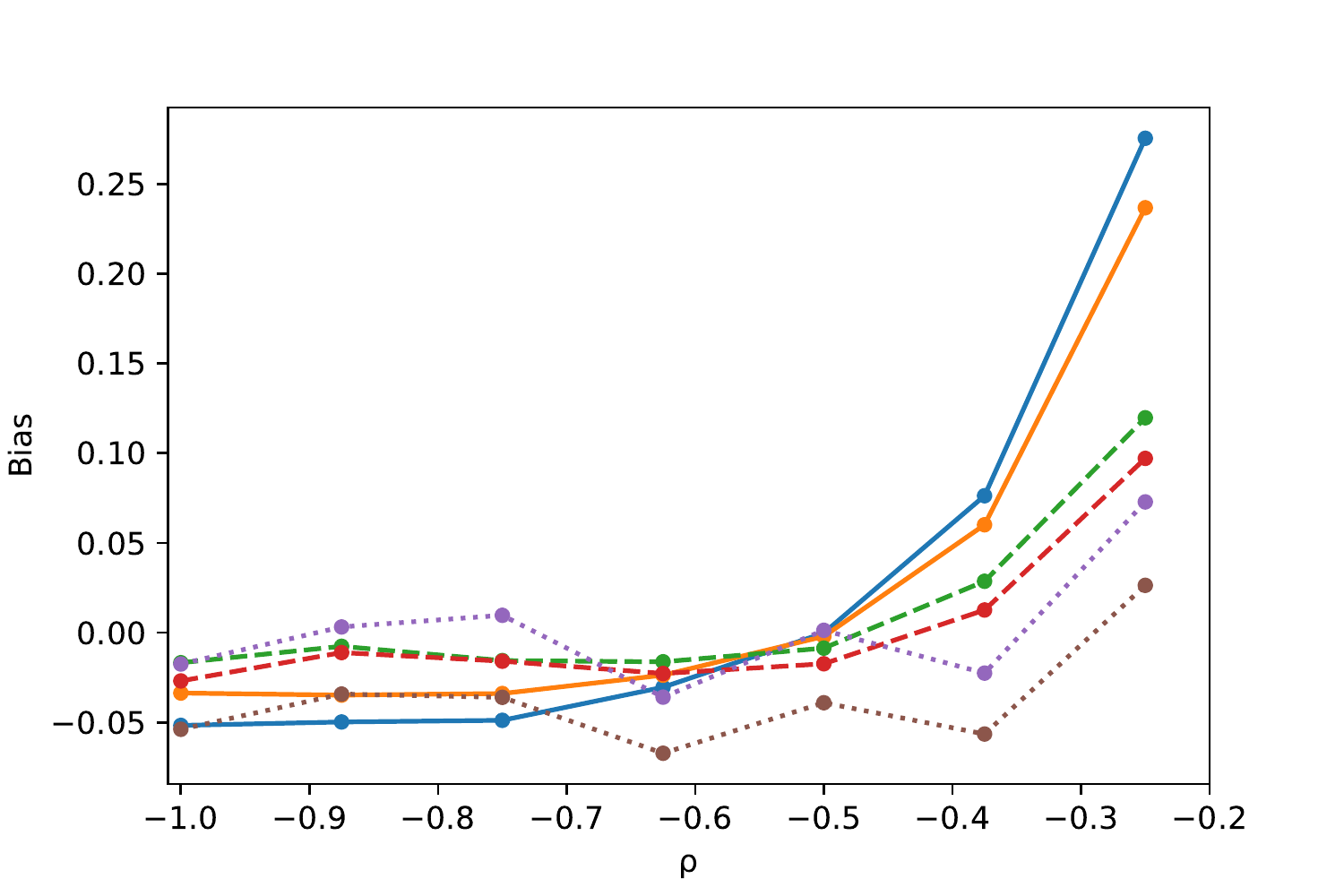}}
\end{figure}

Figure~\ref{Bias_Burr} depicts the bias for several choices of $m$ based on $1000$ simulated datasets of $n=10\,000$ observations from a Burr distribution where the parameters are chosen such that $\gamma=1/2$ is fixed as we vary $\rho$.
As expected, the bias of both estimators approaches $0$ as $\rho$ decreases and increases steeply as $\rho$ approaches $0$. While subplot~(a) of Figure~\ref{Bias_Burr} shows that the bias of $\UP$ is somewhat lower than that of the GP ML estimator when $\rho$ is very negative, subplot~(b) shows that when $\rho$ approaches $0$, the bias of $\UP$ is the higher one for all values of $m$. 

Finally, Figure \ref{single_trajectory} shows that a single trajectory of the estimator $\UP$ as a function of unit increments in $m$ is smooth, in contrast to the characteristically erratic ``Hill plot'' of the GP ML estimator as a function of the top-$k$ order statistics. 
Nonetheless, choosing the block size in practice remains a difficult task even when the single trajectory of the estimator $\UP$ is smooth.

\begin{figure}[h]
	\caption{Single trajectories of the estimators as a function of $m$ or $k$, based on a Student-t($4$) sample of size $n=10\,000$} \label{single_trajectory}
	\subfloat[$\UP$, $m \in {[3{:}10\,000]}$]{\includegraphics[height=0.25\textheight, width=0.5\textwidth]{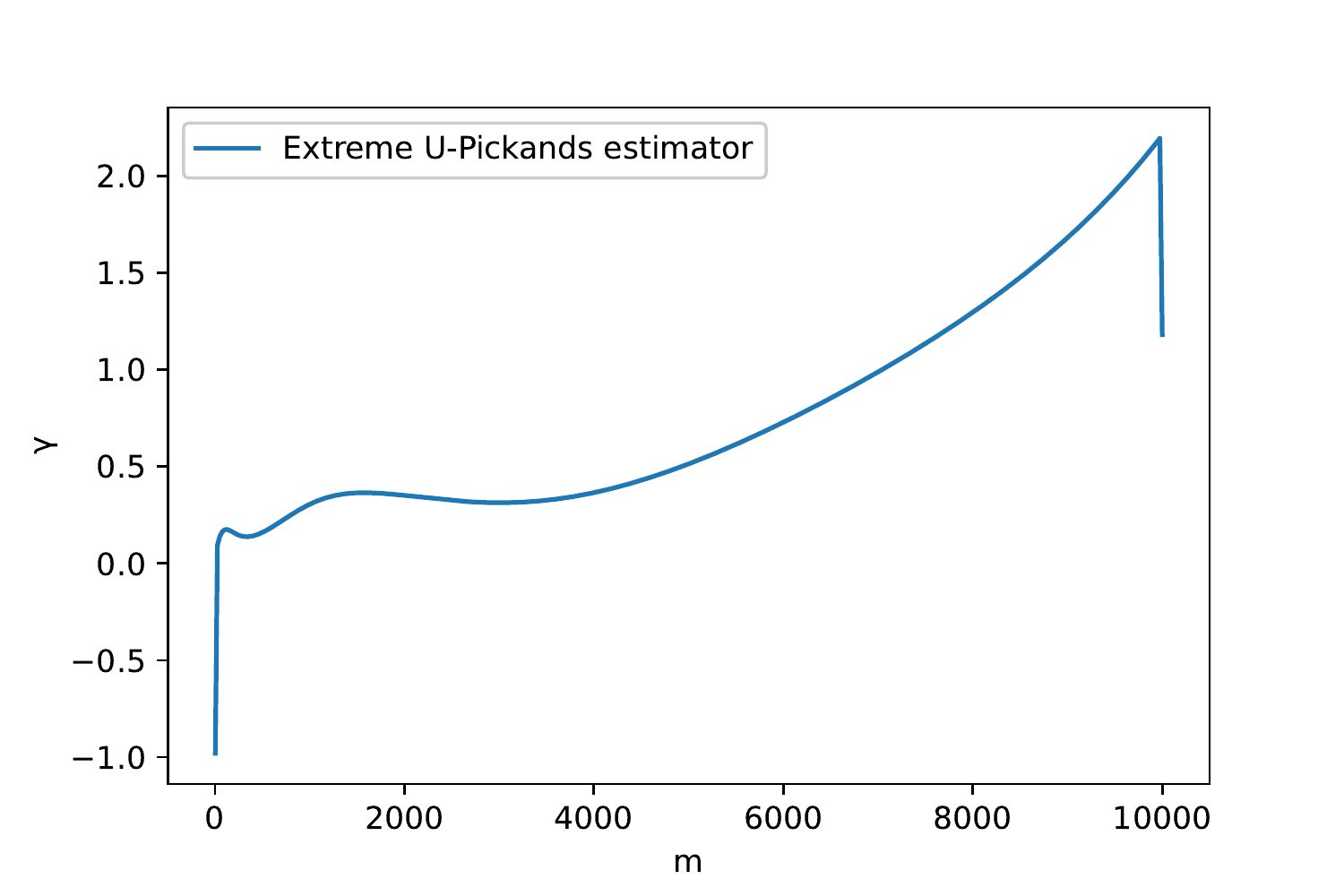}}
	\subfloat[GP ML estimator, $k \in {[15{:}2000]}$ ]{\includegraphics[height=0.25\textheight, width=0.5\textwidth]{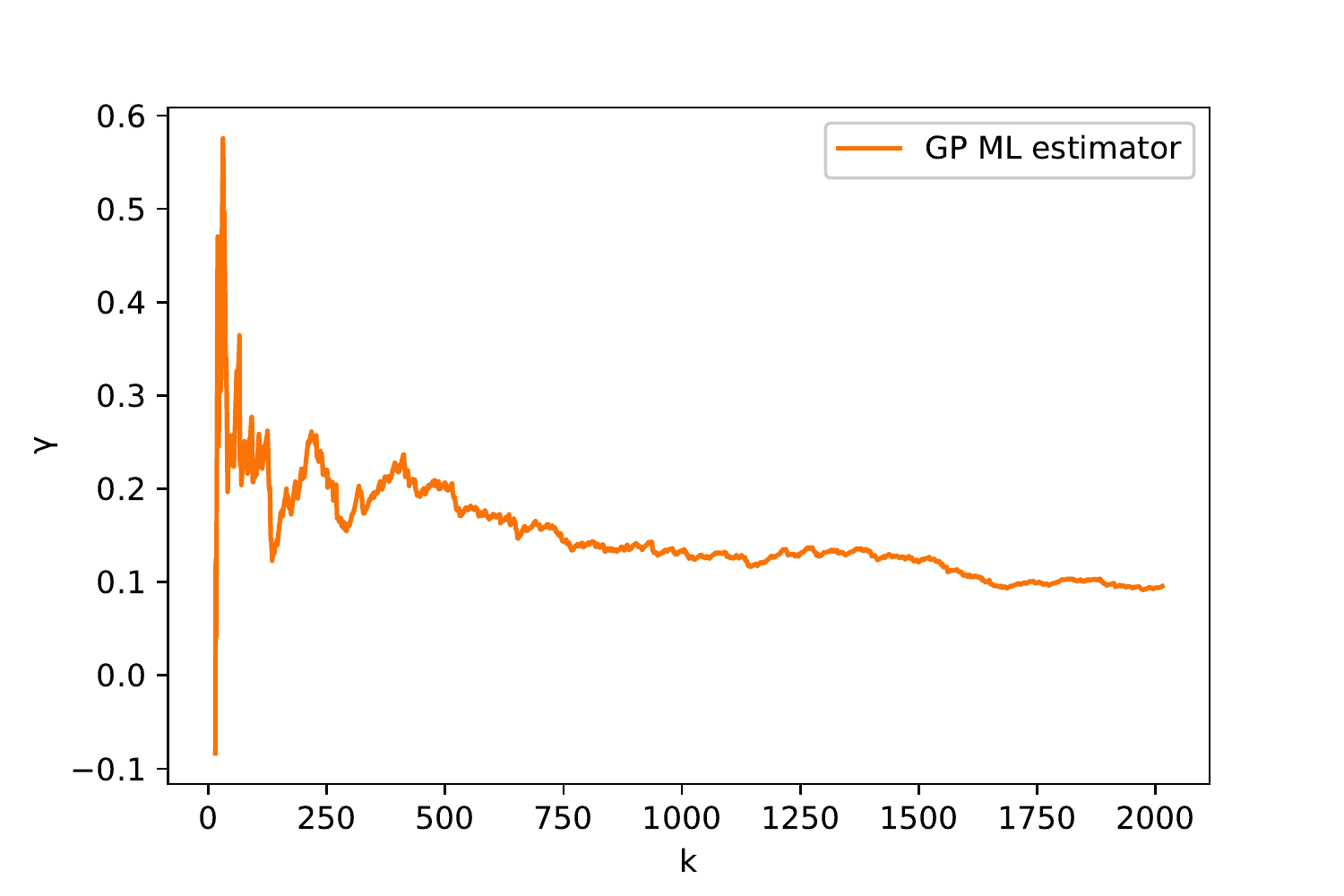}}
\end{figure}

\subsection{Comparison with ABM estimator}
\label{sec:simu:ABM}

We compare the extreme U-Pickands estimator $\UP$ of $\gamma$ to the ABM estimator in \cite{oorschot2020all}. The latter estimator is restricted to the range $\gamma > 0$ and is scale invariant but not location invariant (although it can be made location invariant for instance by first centering the data at the median).
Thinking in terms of disjoint blocks, the ABM estimator based on $m$-blocks uses $n/m$ sample points while $\UP$ uses $3 n/m = n / (m/3)$, three times as many. To ensure a level playing field, we compare the ABM estimator at block size $m$ with the extreme U-Pickands estimator at block size $3m$. In other words, we compare both estimators at equal values of $m/q$, where $m$ is the actual block size used for the estimator and where $q = 1$ for ABM while $q = 3$ for the extreme U-Pickands estimator.\footnote{Note that the recursive formula on page~5 of \cite{oorschot2020all} should be $p_i = p_{i-1} \frac{n-i-m+2}{n-i+1}$ for $i \in [2{:}(n-m+1)]$.} 

For each of a range of distributions, Figure~\ref{fig:ABM} shows the MSE over $100$ samples of size $n = 1\,000$. As the ABM estimator is designed as a quasi-MLE based on the Fr\'{e}chet distribution, its performance is superior for samples drawn from that distribution. For samples drawn from other distributions, the extreme U-Pickands estimator is competitive, despite its simple and explicit definition.

\begin{figure}
	\caption{Comparison of MSE  of ABM estimator \citep{oorschot2020all} at block size $m$ with that of $\UP$ at block size $3m$, based on 100 samples of size $n = 1\,000$ from a number of sampling distributions and for a range of $m$.}
	\label{fig:ABM}
	\subfloat[Fr\'{e}chet(1), $\gamma = 1$]{\includegraphics[width=.45\textwidth]{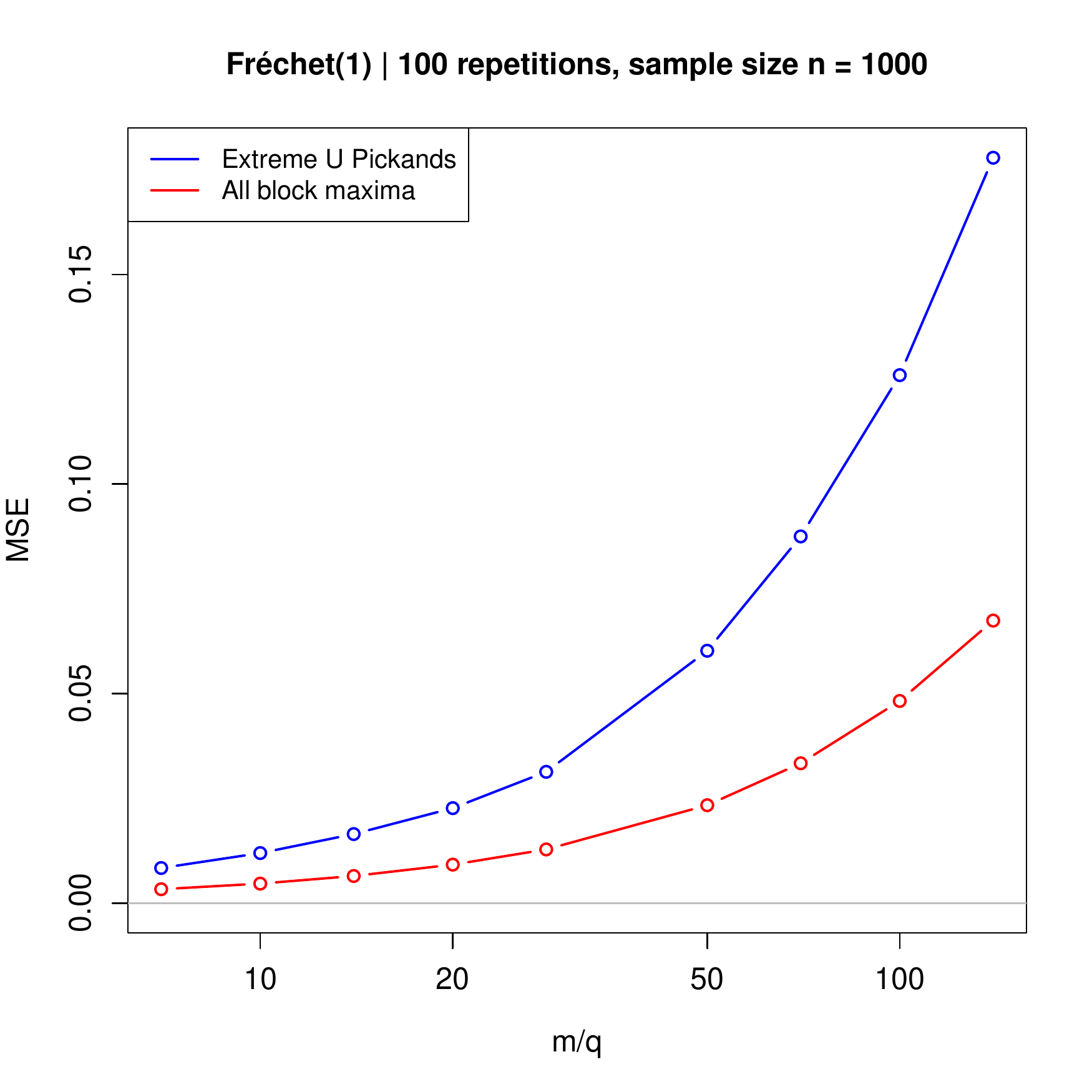}}
	\subfloat[Student(2), $\gamma = 0.5$]{\includegraphics[width=.45\textwidth]{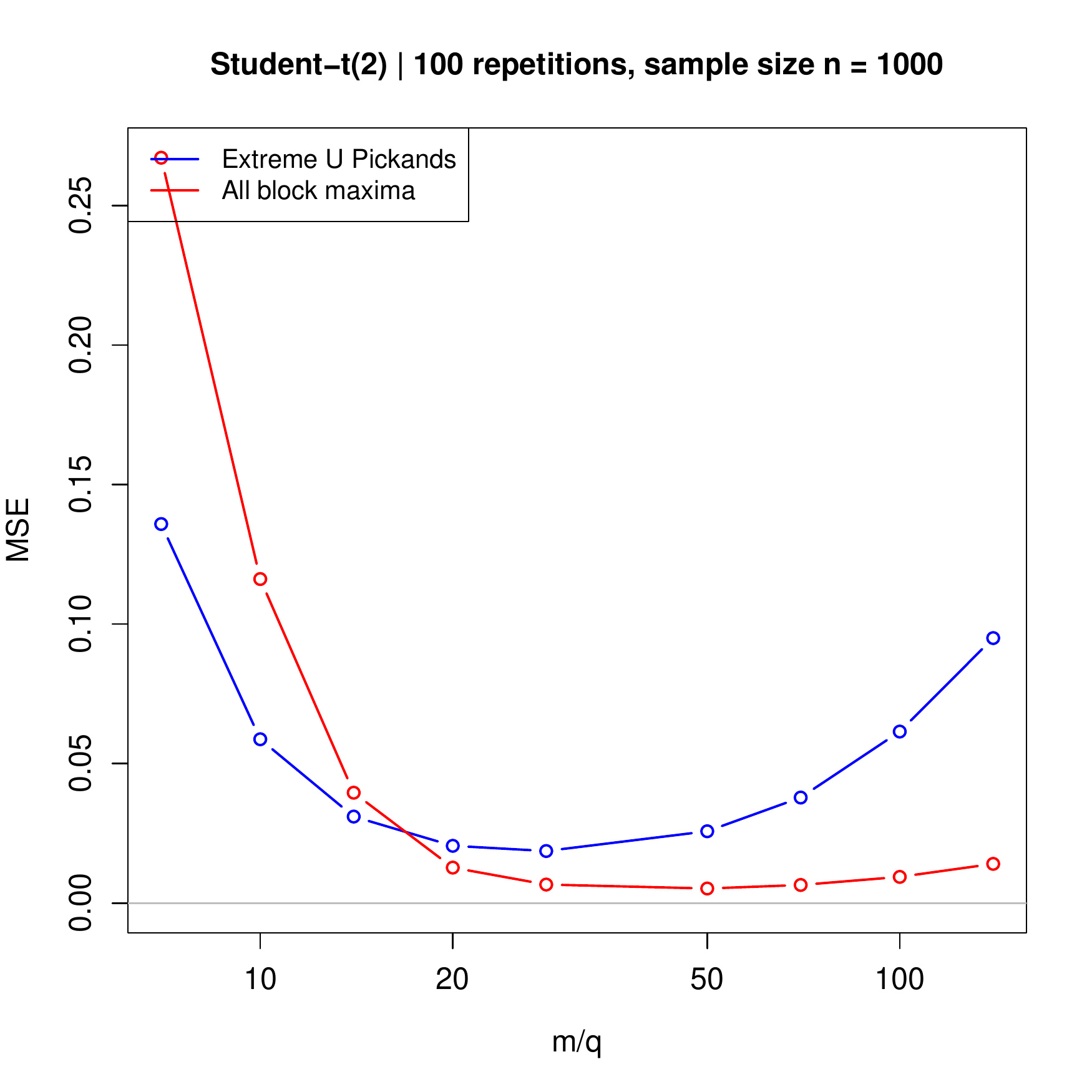}}\\
	\subfloat[GP(0.2), $\gamma = 0.2$]{\includegraphics[width=.45\textwidth]{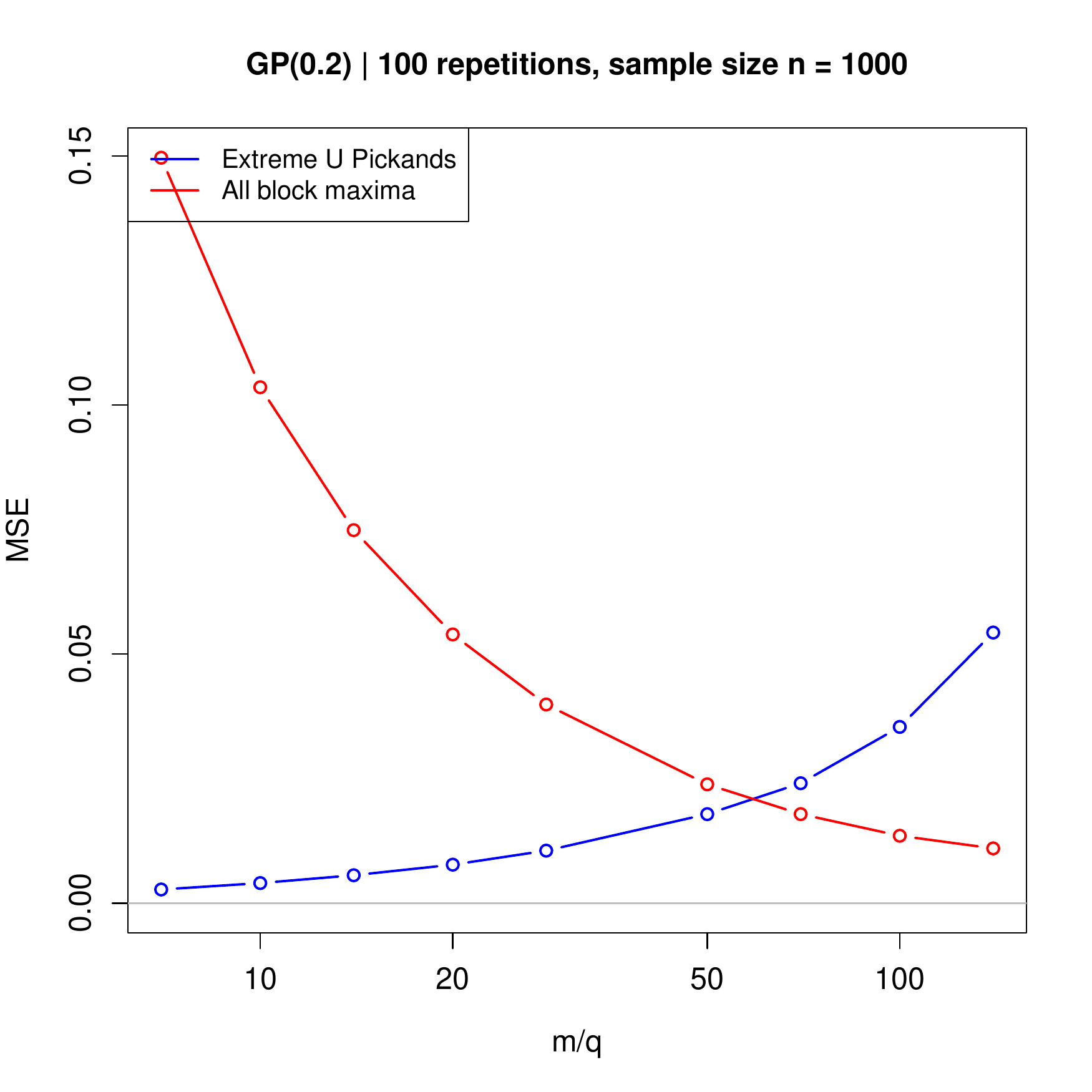}}
	\subfloat[Student(5), $\gamma = 0.2$]{\includegraphics[width=.45\textwidth]{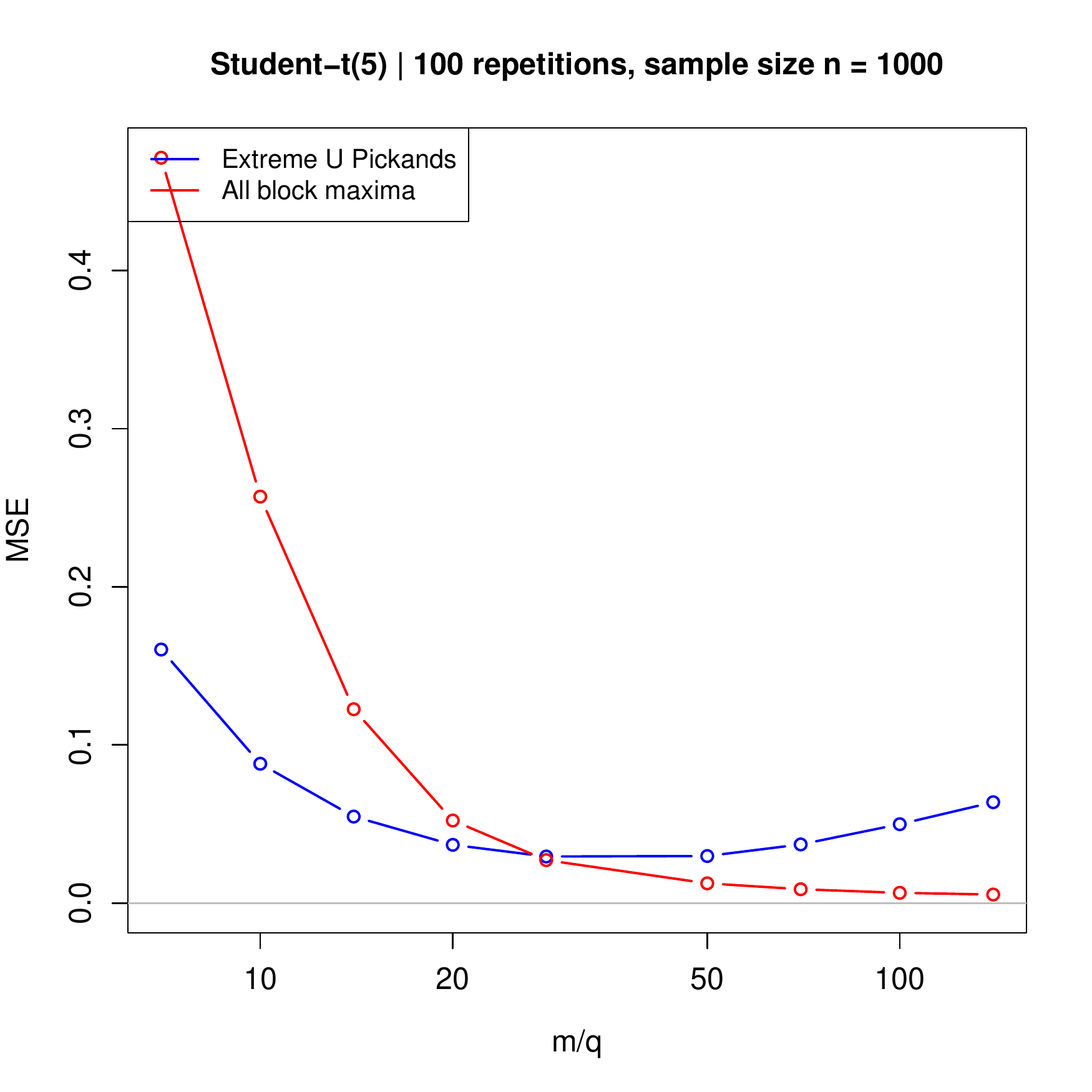}}
\end{figure}

%% file: 6disc.tex
In Theorem \ref{main_GP_thm}, as well as elsewhere in the paper, we have assumed that the data are independently and identically distributed. Since $\gamma$ is a parameter belonging to the marginal distribution of the observations, we expect that Theorem \ref{main_GP_thm} can be adapted to the case where the $Z_i$ all follow the same GP distribution but are weakly dependent. 
If the dependence is weak enough and the sample size is large, one may still find that within a randomly chosen block $Z_I$ the observations are approximately independent, so that $\theta \approx \expec[K_{m}(Z_I)]$ and so that the approximation $\zeta_{m, \ell} \sim \ell \zeta_{m,1}$ underlying Proposition~\ref{Prop:variance} remains valid.
The asymptotic variance will be different, however.

Furthermore, in the simulation study, when data were sampled from the $\GP(\gamma)$ distribution, the estimator $\UPdrie$ turned out to have the lowest variance among all unbiased estimators $\UP$ of $\gamma$, for $m \geq 3$. Since the estimator $\UPdrie$ is also a function of the entire sample, it is an open question whether the estimator $\UPdrie$ is the uniformly minimum variance unbiased estimator for $\gamma$ of the $\GP(\gamma)$ distribution. Because the GP distributions do not constitute an exponential family, this question is left to further research.

Finally, note that a U-statistic may also be calculated based on all exceedances over a high threshold. If the kernel is location-scale invariant, this amounts to applying a U-statistic only to a certain fraction of the largest observations in a sample. Doing so while keeping $m = q$ fixed would yield another kind of extreme U-statistic, perhaps worth studying as well.

%% file: appAprelim.tex
Let $\overset{d}{=}$ denote equality in distribution. For a distribution function $F$, recall its tail quantile function $U$ in \eqref{eq:tqf}.
The tail quantile function of the standard GP distribution with parameter $\gamma \in \reals $ is $h_\gamma$ in \eqref{GP_function}.
Recall that a random variable $Y$ is said to have a standard Pareto distribution if $\prob(Y \le y) = 1 - 1/y$ for $y \geq 1$. The function $h_\gamma$ has the property that if $Y$ follows a standard Pareto distribution, the variable $Z=h_\gamma(Y)$ follows a $\GP(\gamma)$ distribution. Moreover, the function $h_\gamma$ satisfies the functional equation
\[
h_\gamma(xy) = h_\gamma(x) + x^\gamma \, h_\gamma(y),
\qquad x, y > 0.
\]

Let $Y_1, Y_1^*, Y_2, Y_2^*, \ldots $ be independent standard Pareto distributed random variables. For integer $1 \le q \le m$, R\'enyi's representation for exponential order statistics implies that the vector of top-$q$ order statistics of $Y^*_1,\ldots,Y^*_m$ admits the representation
\begin{equation}
	\label{eq:ParetoOSrepr}
	(Y^*_{m-i+1:m})_{i=1}^q \eqd (Y^*_{m-q+1:m} Y_{q-i:q-1})_{i=1}^q
\end{equation}
where $Y_{0:q-1} = 1$. Then, we can represent the vector of top-$q$ order statistics of an independent random sample $Z^*_1,\ldots,Z^*_n$ from the $\GP(\gamma)$ distribution as
\begin{align*}
	(Z^*_{m-i+1:m})_{i=1}^q 
	&\eqd (h_\gamma(Y^*_{m-i+1:m}))_{i=1}^q \\
	&\eqd ( h_\gamma(Y^*_{m-q+1:m} Y_{q-i:q-1}))_{i=1}^q \\
	&= \left(  h_\gamma(Y^*_{m-q+1:m}) + (Y^*_{m-q+1:m})^\gamma  h_\gamma(Y_{q-i:q-1}) \right)_{i=1}^q.
\end{align*}
We have thereby proved the following result.

\begin{lemma}
	\label{lem:GPos}
	Consider the integers $1 \leq q \leq m$ and the scalar $\gamma \in \reals$. The vector of top-$q$ order statistics of an independent $\GP(\gamma)$ random sample $Z^*_1,\ldots,Z^*_m$ admits the representation
	\[
	(Z^*_{m-i+1:m})_{i=1}^q 
	\eqd
	\left( Z^*_{m-q+1:m} + (1 + \gamma Z^*_{m-q+1:m}) Z_{q-i:q-1} \right)_{i=1}^q
	\]
	where $Z_1,\ldots,Z_{q-1}$ is another independent standard $\GP(\gamma)$ random sample, independent from $Z^*_1,\ldots,Z^*_m$, and where $Z_{0:q-1} = 0$.
\end{lemma}

Recall that the kernel $K_m : \reals^m \to \reals$ of the extreme U-statistic is such that, for all integer $m \ge q$ and all $(x_1,\ldots,x_m) \in \reals^m$,
\[
K_m(x_1,\ldots,x_m) = K(x_{m:m},\ldots,x_{m-q+1:m}),
\]
for some (symmetric) function $K : \reals^q \to \reals$. If the kernel is also location-scale invariant, i.e., if the kernel satisfies 
\[
K_m(ax_{m:m}+b,\ldots, ax_{m-q+1:m}+b) = K(x_{m:m},\ldots,x_{m-q+1:m}).
\]
for all $a>0$ and $b\in \reals$, the representation derived in Lemma~\ref{lem:GPos} has a useful consequence.

\begin{lemma}
	\label{lem:GPKm}
	Consider integers $3 \le q \le m$ and $\gamma \in \reals$. If the kernel $K : \reals^q \to \reals$ is location-scale invariant and if $Z^*_1,\ldots,Z^*_m$ are independent standard $\GP(\gamma)$ random variables, then
	\[
	\left(Z^*_{m-q+1:m}, K_m(Z^*_1,\ldots,Z^*_m)\right) 
	\eqd \left(Z^*_{m-q+1:m}, K\bigl((Z_{q-i:q-1})_{i=1}^q \bigr)  \right)
	\]
	with $Z_1,\ldots,Z_{q-1}$ as in Lemma~\ref{lem:GPos}. In particular, $K_m(Z^*_1,\ldots,Z^*_m)$ is independent of $Z^*_{m-q+1:m}$ and its distribution does not depend on $m$.
\end{lemma}

It follows from Lemma~\ref{lem:GPKm}  that
\begin{equation}
	\label{def: expectation}
	\theta
	:= \theta_m
	= \expec \left[ K_m\left( Z^*_{1}, \dots,Z^*_{m} \right) \right]
	= \expec \bigl[ K \bigl( (Z_{q-j:q-1})_{j=1}^{q} \bigr) \bigr] 
	= \mu(\gamma),
\end{equation}
which does not depend on $m$, provided $m \geq q$.

%% file: appBprop22.tex
Recall the bound \eqref{eq:approxvarUnm}, which we state here for convenience again:
\begin{equation}
	\label{eq:approxvarUnm2}
	\left| \var U_{n, Z}^{m} - \frac{m \zeta_{m,1}}{k} \right|
	\le
	\sum_{\ell=1}^m p_{n,m}(\ell) 
	\left| \zeta_{m,\ell} - \ell \zeta_{m,1} \right|.
\end{equation}
Further, recall that 
\begin{equation}
	\label{eq:hyppmf}
	p_{n,m}(\ell) = \binom{n}{m}^{-1} \binom{m}{\ell} \binom{n-m}{m-\ell},
\end{equation}
for integer $1 \le m \le n$ and $0 \le \ell \le m$, represents the probability mass function of the hypergeometric distribution counting the number of successes $\ell$ after $m$ draws without replacement out of an urn containing $n$ balls of which $m$ have the desired colour. To prove Proposition \ref{Prop:variance}, we will decompose the sum over $\ell \in [1{:}m]$ in \eqref{eq:approxvarUnm2} into two pieces, $\ell \le \ell_n$ and $\ell > \ell_n$, with $\ell_n = \oh(m)$ determined below.
\begin{itemize}
	\item 
	For $\ell \le \ell_n$, we will show that $|\zeta_{m,\ell} - \ell \zeta_{m,1}|$ is sufficiently small.
	\item
	For $\ell > \ell_n$, we use tail bounds on the hypergeometric distribution $p_{n,m}(\ell)$.
\end{itemize}

We make these statements precise in the proof below by listing a number of properties that, in combination, imply that the bound \eqref{eq:approxvarUnm2} is $\oh(1/k)$ as $n \to \infty$.

\begin{proof}[Proof of Proposition \ref{Prop:variance}]
	By \eqref{eq:approxvarUnm2}, we have, for any $b > 0$,
	\begin{multline*}
		\left| \var U_{n}^{m} - \frac{m}{k} \zeta_{m,1} \right|
		\le 
		\max_{\ell\in[1:\ell_n]} 
		\left| \ell^{-1} \zeta_{m,\ell} - \zeta_{m,1} \right| 
		\cdot \sum_{\ell=1}^{\ell_n} p_{n,m}(\ell) \ell \\
		+ \max_{\ell\in[1:m]} \ell^{-b} \left|\zeta_{m,\ell} - \ell \zeta_{m,1}\right|
		\cdot \sum_{\ell=\ell_n+1}^m p_{n,m}(\ell) \ell^{b}.
	\end{multline*}
	The expectation of the hypergeometric distribution in \eqref{eq:hyppmf} is $m \cdot m/n = m/k$. The previous decomposition implies the following break-down of \eqref{eq:wish} into three sub-goals:
	As $n \to \infty$,
	\begin{align}
		\label{eq:unifzeta1m}
		\max_{\ell\in[1:\ell_n]} 
		\left| \ell^{-1} \zeta_{m,\ell} - \zeta_{m,1} \right| 
		&= \oh(1/m), \\
		\label{eq:zetaO1}	
		\max_{\ell\in[1:m]} \ell^{-b} \left| \zeta_{m,\ell} - \ell \zeta_{m,1} \right| 
		&= \Oh(1), \\
		\label{eq:hyptail1k}
		\sum_{\ell=\ell_n+1}^m p_{n,m}(\ell) \ell^b
		&= \oh(1/k),  
	\end{align}
	where we choose $b$ such that $1/2 \leq b\leq 1$. 
	
	If the above sub-goals are attained, we obtain that
	\[
	\left| \var U_{n}^{m} - \frac{m}{k} \zeta_{m,1} \right|
	\le \oh(1/m) m/k + \Oh(1) \oh(1/k) = \oh(1/k), \qquad \text{ as } n \to \infty.
	\]
	The proof of the three sub-goals \eqref{eq:unifzeta1m}, \eqref{eq:zetaO1} and \eqref{eq:hyptail1k} for the choice
	\begin{align}\label{choice_ell}
		\ell_n = \left\lfloor \left( \frac{m}{n} + \left(\frac{\ln n}{2m}\right)^{1/2} \right) m \right \rfloor,
	\end{align}
	with $\lfloor x \rfloor$ the integer part of $x \in \reals$,
	is the content of Lemmas \ref{lem:zetam1bound}, \ref{lem:hyptail1k} and \ref{subgoal_3} below.
\end{proof}
We remark that the choice of $\ell_n$ is constrained by \eqref{eq:unifzeta1m} and \eqref{eq:hyptail1k}. On the one hand, for \eqref{eq:unifzeta1m} to hold, we will need that $\ell_n = \oh(m)$ as $n \to \infty$. On the other hand, for \eqref{eq:hyptail1k} to hold, we will need that $\ell_n \to \infty$ sufficiently fast. Both criteria can be met for the choice of $\ell_n$ in \eqref{choice_ell}, provided $m$ is not too small or large as per Condition \ref{degree_sequence}. 

The proof of sub-goal~\eqref{eq:unifzeta1m} is somewhat involved and is developed in Section~\ref{sec:unifzeta1m} via a series of lemmas. First, we establish sub-goals \eqref{eq:zetaO1} and \eqref{eq:hyptail1k}.

Recall that for the independent $\GP(\gamma)$ random variables $Z_1,\ldots ,Z_{2m-1}$,
\[
\zeta_{m,\ell} = \expec \left[ 
K_m(Z_{[1:m]}) \,
K_m(Z_{[(m-\ell+1):(2m-\ell)]})
\right] - \theta^2,
\]
for $\ell \in [1{:}m]$. For integer $m \ge q$ and real $a > 0$, define 
\[
\| K \|_a :=  \bigl( \expec\bigr[ \bigl|
K \bigl((Z_{q-j:q-1})_{j=1}^{q}\bigr)\bigr|^a 
\bigr] \bigr)^{1/a},
\]
which, by Lemma~\ref{lem:GPKm}, does not depend on $m$. 

\begin{lemma}[Sub-goal~\eqref{eq:zetaO1}]
	\label{lem:zetam1bound}
	If Condition \ref{Condition_kernel_scale-loc_integrability} is satisfied, then, for every integer $m \ge q$, we have, for any $a >2$,
	\begin{align}
		\label{eq:zetam1bound}
		\zeta_{m,1} &\le \left(2q \right)^{1-2/a} \| K \|^2_a \cdot m^{2/a-1}.
	\end{align}
	As a consequence, we can take any $b \in (2/a, 1)$ to ensure that  \eqref{eq:zetaO1} holds.
\end{lemma}

\begin{proof}
	To ensure \eqref{eq:zetaO1}, it is sufficient to have
	\begin{align}
		\label{eq:zetaO1a}
		\sup_{m \ge q} \max_{\ell \in [1:m]} \zeta_{m,\ell} &< \infty, \\
		\label{eq:zetaO1b}
		\sup_{m \ge q} m^{1-b} \zeta_{m,1} &< \infty,
	\end{align}
	provided $0\leq b \leq 1$. By H\"older's inequality, Lemma~\ref{lem:GPKm} and Condition~\ref{Condition_kernel_scale-loc_integrability}, 
	\[
		\sup_{m \ge q} \max_{\ell \in [1:m]} \zeta_{m,\ell} \leq \| K \|^2_2-\theta^2< \infty,
	\]
	such that \eqref{eq:zetaO1a} is guaranteed.
	
	For \eqref{eq:zetaO1b},	consider the difference
	\[
	\Delta_m
	= K_m(Z_{[1:m]}) \, K_m(Z_{[m:(2m-1)]}) - 
	K_{m-1}(Z_{[1:(m-1)]}) \, K_{m-1}(Z_{[(m+1):(2m-1)]}),
	\]
	capturing the effect of omitting $Z_m$ from $Z_1,\ldots,Z_{2m-1}$.
	It follows that
	\begin{align*}
		\zeta_{m,1}
		&= \expec \left[ K_m(Z_{[1:m]}) K_m(Z_{[m:(2m-1)]}) \right] - \theta^2= \expec(\Delta_m). 
	\end{align*}
	Further, if $Z_{m}$ does not belong to the top-$q$ of $Z_{[1:m]}$ nor to the one of $Z_{[m:(2m-1)]}$, $K_m(Z_{[1:m]}) = K_{m-1}(Z_{[1:(m-1)]})$ and $K_m(Z_{[m:(2m-1)]}) = K_{m-1}(Z_{[(m+1):(2m-1)]})$, which implies that $\Delta_m = 0$. The probability that $Z_{m}$ belongs to the top-$q$ of $Z_{[1:m]}$ or $Z_{[m:(2m-1)]}$ is not larger than $2q/m$. We find, by applying H\"older's inequality twice, for any $a > 2$,
	\begin{align*}
		\left| \expec(\Delta_m) \right|
		&\le \expec(|\Delta_m|)
		\le  \left( \expec\left[\left|  K_m(Z_{[1:m]}) K_m(Z_{[m:(2m-1)]})  \right|^{a/2} \right] \right)^{2/a}  \left(2q/m \right)^{1-2/a}  \\
		&\le \| K \|^2_a \cdot \left(2q/m \right)^{1-2/a},
	\end{align*}
	which is the upper bound in \eqref{eq:zetam1bound}. By multiplying the upper bound with $m^{1-b}$ for some $b \in (2/a, 1)$, we obtain~\eqref{eq:zetaO1}.
\end{proof}

\begin{lemma}[Sub-goal \eqref{eq:hyptail1k}]
	\label{lem:hyptail1k}
	If Condition~\ref{degree_sequence} holds, the sequence $\ell_n$ in~\eqref{choice_ell} satisfies $\ell_n \to \infty$ and  $\ell_n = \oh(m)$ as $n \to \infty$, and sub-goal~\eqref{eq:hyptail1k} holds.
\end{lemma}

\begin{proof}
	The statements regarding the asymptotic order of $\ell_n$ follow immediately.
	Let $H_n$ be a hypergeometric random variable with probability mass function $p_{n,m}$ in \eqref{eq:hyppmf}. We have $\expec(H_n) = m/k$. Put $t = (\ell_n+1)/m - 1/k$ for $t \ge 0$ to be determined. We have
	\begin{align*}
		k \sum_{\ell=\ell_n+1}^m p_{n,m}(\ell) \ell^b 
		&\le k m^{b}\sum_{\ell=\ell_n+1}^m p_{n,m}(\ell)\\
		&= k m^{b} \prob[H_n \ge (k^{-1} + t) m ] \\
		&\le \exp \left( - 2 m t^2 - \ln k^{-1} +b \ln m \right),
	\end{align*}
	see  \cite{hoeffding1963probability} or \cite{chvatal1979tail}. Since $t \ge ((2m)^{-1} \ln n)^{1/2}$ and $b<1$, we get
	\[
	2m t^2 + \ln 1/k - b \ln(m)
	\ge \ln n - \ln k -b \ln m = (1-b)\ln m \to \infty,
	\]
	as $n \to \infty$. The lemma follows.
\end{proof}

\subsection{Proof of sub-goal \eqref{eq:unifzeta1m}}
\label{sec:unifzeta1m}

It remains to prove the first sub-goal \eqref{eq:unifzeta1m} for the choice of $\ell_n$ above. Recall that $\ell_n=\oh(m)$ and $\ell_n \to \infty$ as $n \to \infty$. Throughout, since we handle $\ell<\ell_n$, we assume $\ell, q, m$ are positive integers with $m \ge \ell + q$ and $\ell \ge 2$ (if $\ell=1$ there is nothing to prove).

\begin{lemma}[Sub-goal \eqref{eq:unifzeta1m}]\label{subgoal_3}
	If Conditions~\ref{Condition_kernel_scale-loc_integrability} and~\ref{degree_sequence} are satisfied, then sub-goal~\eqref{eq:unifzeta1m} holds, i.e., as $n \to \infty$,
	\begin{align*}
		\max_{\ell \in [1:\ell_n]} 
		\left| \ell^{-1} \zeta_{m,\ell} - \zeta_{m,1} \right|
		=\oh(1/m). 
	\end{align*}
\end{lemma}
\begin{proof}[Proof of sub-goal \eqref{eq:unifzeta1m}]
	Consider the counting variable
	\[
	N_{\ell,m} 
	= \sum_{i=m-\ell+1}^m \1_{i,\ell,m},
	\] where
	\begin{multline}
		\1_{i,\ell,m} 
		= \1 \bigl\{
		\{ \text{$Z_i$ is in the top-$q$ of $Z_{[1:m]}$} \} \\
		\cup \{\text{$Z_i$ is in the top-$q$ of $Z_{[(m-\ell+1):(2m-\ell)]}$} \}
		\bigr\}.
		\label{eq:1ilm}
	\end{multline}
	Define
	\[
	\Xi_{\ell,m} = K_m(Z_{[1:m]}) \, K_m(Z_{[(m-\ell+1):(2m-\ell)]}),
	\]
	and write $\expec[\Xi; A] = \expec[\Xi \1_A]$ for a random variable $\Xi$ and an event $A$. Then,
	\begin{align*}
		\zeta_{m,\ell} 
		&= \expec[ \Xi_{\ell,m} ; N_{\ell,m} = 0 ]
		+ \expec[ \Xi_{\ell,m} ; N_{\ell,m} = 1 ] 
		+ \expec[ \Xi_{\ell,m} ; N_{\ell,m} \ge 2 ] - \theta^2 \\
		\text{and} \quad
		\zeta_{m,1}
		&= \expec[\Xi_{1,m} ; N_{1,m} = 0] + \expec[\Xi_{1,m} ; N_{1,m} = 1] - \theta^2.
	\end{align*}
	We obtain the following decomposition of the left-hand side of \eqref{eq:unifzeta1m} into three terms:
	\begin{align}
		\left| \zeta_{m,\ell} - \ell \zeta_{m,1} \right|
		&\le \left| \expec[ \Xi_{\ell,m} ; N_{\ell,m} = 0 ] - \ell \expec[\Xi_{1,m} ; N_{1,m} = 0] + (\ell - 1) \theta^2 \right| \nonumber\\
		&\quad {}
		+ \left| \expec[ \Xi_{\ell,m} ; N_{\ell,m} = 1 ] - \ell \expec[\Xi_{1,m} ; N_{1,m} = 1] \right| \nonumber\\
		&\quad {}
		+ \left| \expec[ \Xi_{\ell,m} ; N_{\ell,m} \ge 2 ] \right|\nonumber\\
		&:=\mathbb{I}_1+\mathbb{I}_2+\mathbb{I}_3\label{eq:dcmp}
	\end{align}
	According to Lemmas~\ref{lem:N0GPD}, \ref{lem:N1} and \ref{lem:N2} below, we have
	\begin{equation}
		\label{eq:Ibound}
		\mathbb{I}_j  \lesssim \ell \left( \ell/ m^2 \right)^{1-2/a}
	\end{equation}
	for $j \in \{1,2,3\}$ and any $a>2$, where $a_n \lesssim b_n$ if and only if there exists a constant $c>0$ such that $a_n \leq c b_n$ for all~$n$. Recall the choice of $\ell_n$ in \eqref{choice_ell}.
	Together with Condition~\ref{degree_sequence}, the bounds on $\mathbb{I}_j$ will then imply, as $n \to \infty$ and for sufficiently large $a$, that
	\begin{align*}
		\max_{\ell \in [1:\ell_n]} 
		\left| \ell^{-1} \zeta_{m,\ell} - \zeta_{m,1} \right|
		&\lesssim  \max_{\ell \in [1:\ell_n]}  \left( \ell/ m^2 \right)^{1-2/a}\\
		&\le n^{2/a-1} +(2^{-1}\ln(n))^{1/2-1/a}m^{3/a-3/2}=\oh(1/m).
	\end{align*}
	We conclude that sub-goal~\eqref{eq:unifzeta1m} is attained.  
\end{proof}

We now list Lemmas~\ref{lem:N0GPD}, \ref{lem:N1} and \ref{lem:N2}, providing the bounds~\eqref{eq:Ibound}. The proofs of these lemmas rely on counting variable inequalities which are developed afterwards in Lemmas~\ref{lem:N} to~\ref{lem:EBmlm1mlm}.

\begin{lemma}[Handling $\mathbb{I}_1$ in \eqref{eq:dcmp}]
	\label{lem:N0GPD}
	If Condition \ref{Condition_kernel_scale-loc_integrability} is satisfied, then
	\[
	\expec[ \Xi_{\ell,m}; N_{\ell,m} = 0] = \prob(N_{\ell,m} = 0) \, \theta^2.
	\]
	In addition,
	\[
	\mathbb{I}_1\le \theta^2  \cdot \frac{q^2\ell(\ell-1)}{(m-1)(m-\ell+1)}.
	\]
\end{lemma}
\begin{proof}
	Let $r=m-\ell$. If $N_{\ell,m} = 0$, the top-$q$ of $Z_{[1:m]}$ is the same as the top-$q$ of $Z_{[1:(m-\ell)]}$ and the top-$q$ of $Z_{[(m-\ell+1):(2m-\ell)]}$ is the same as the one of $Z_{[(m+1):(2m-\ell)]}$. The event $\{N_{\ell,m} = 0\}$ is a function of the random variables $Z_{[(r+1):m]}$ and the $q$-largest order statistics of $Z_{[1:r]}$ and $Z_{[(m+1):(m+r)]}$. These statements taken together with a slight extension of the properties of the top order statistics of a generalised Pareto sample in Lemma~\ref{lem:GPKm} imply
	\begin{align*}
		\expec[\Xi_{\ell,m}; N_{\ell,m} = 0]
		&= \expec[ K_r(Z_{[1:r]}) K_r(Z_{[(m+1):(m+r)]}); N_{\ell,m} = 0]\\
		&=\prob(N_{\ell,m} = 0) \, \theta^2.
	\end{align*}
	Thanks to that equality, we find
	\begin{align*}
		\lefteqn{
			\left| \expec[ \Xi_{\ell,m} ; N_{\ell,m} = 0 ] - \ell \expec[\Xi_{1,m} ; N_{1,m} = 0] + (\ell - 1) \theta^2 \right| 
		} \\
		&= \left| \prob(N_{\ell,m} = 0) - \ell \prob(N_{1,m} = 0) + \ell - 1 \right| \theta^2 \\
		&= \left| \ell \prob(N_{1,m} \ge 1) - \prob(N_{\ell,m} \ge 1) \right| \theta^2.
	\end{align*}
	Apply Lemma~\ref{lem:PN1diff} to conclude.
\end{proof}
In case $N_{\ell,m} = 1$, exactly one of the $\ell$ indicators $\1_{i,\ell,m}$ in \eqref{eq:1ilm} for $i \in [(m-\ell+1){:}m]$ is equal to one while the other $\ell-1$ indicators are zero. We find
\begin{align}
	\label{eq:Nlm1Bilm}
	\{ N_{\ell,m} = 1 \} 
	&= \bigcup_{i=m-\ell+1}^m B_{i,\ell,m} \qquad \text{where} \\
	\nonumber
	B_{i,\ell,m} 
	&= \{ \1_{i,\ell,m} = 1 \} \cap \bigcap_{j \in [(m-\ell+1):m] \setminus \{i\}} \{ \1_{j,\ell,m} = 0\},\end{align}
and the union over $i$ being disjoint.
Note that $B_{i,\ell,m}$ is the event that $Z_i$ is in the top-$q$ of $Z_{[1:m]}$ or $Z_{[(m-\ell+1):(2m-\ell)]}$ while the other variables $Z_j$ for $j \in [(m-\ell+1){:}m] \setminus \{i\}$ are not.
Further, $\{ N_{1, m} = 1\} = \{ \1_{m,1,m} = 1 \}$. 

\begin{lemma}[Handling $\mathbb{I}_2$ in \eqref{eq:dcmp}]
	\label{lem:N1}
	If Condition \ref{Condition_kernel_scale-loc_integrability} is satisfied, we have
	\[
	\mathbb{I}_2
	\le  4(5q^2-q)^{1-2/a} \|K\|^2_a \ell \left(\frac{(\ell-1)}{(m-\ell+1)^2} \right)^{1-2/a},
	\]
	for all $a >2$, provided that $m$ is sufficiently large.
\end{lemma}
\begin{proof}
	By symmetry and since the union in \eqref{eq:Nlm1Bilm} is disjoint,
	\begin{align*}
		\expec[ \Xi_{\ell,m}; N_{\ell,m} = 1 ]
		&= \sum_{i=m-\ell+1}^m \expec[ \Xi_{\ell,m}; B_{i,\ell,m} ] \\
		&= \ell \expec[ \Xi_{\ell,m}; B_{m,\ell,m} ].
	\end{align*}
	On $B_{m,\ell,m}$, the top-$q$ order statistics of $Z_{[(m-\ell+1):(2m-\ell)]}$ are the same as those of $Z_{[m:(2m-\ell)]}$, which implies $K_m(Z_{[(m-\ell+1):(2m-\ell)]}) = K_{m-\ell+1}(Z_{[m:(2m-\ell)]})$ and thus
	\[
	\expec[ \Xi_{\ell,m}; B_{m,\ell,m} ]
	= \expec\left[ K_m(Z_{[1:m]}) \, K_{m-\ell+1}(Z_{[m:(2m-\ell)]}); B_{m,\ell,m} \right].
	\]
	On the right-hand side, we replace the indicator of the event $B_{m,\ell,m}$ first by $\1_{m,\ell,m}$ and then by $\1_{m,1,m}$, using Lemmas~\ref{lem:EBmlm1mlm} and \ref{lem:E11diff}, respectively. We find, for $a >2$, by the triangle inequality and H\"older's inequality,
	\begin{align*}
		\lefteqn{
			\left|
			\expec[ \Xi_{\ell,m}; B_{m,\ell,m} ]
			- \expec\left[ K_m(Z_{[1:m]}) \, K_{m-\ell+1}(Z_{[m:(2m-\ell)]})  \1_{m,1,m} \right]
			\right|
		} \\
		&\le \|K\|^2_a \expec\left[\left|\1_{B_{m,\ell,m}} - \1_{m,\ell,m}\right|\right]^{1-2/a} 
		+ \|K\|^2_a \expec\left[\left|\1_{m,\ell,m}-\1_{m,1,m}\right|\right]^{1-2/a} \\
		&\le \|K\|^2_a  \left[ \left( 
		2(2q^2-q) \frac{\ell-1}{(m-1)(m-\ell+1)} \right)^{1-2/a}
		+\left( (q+1)q \frac{\ell-1}{(m+1)m}
		\right)^{1-2/a} \right]\\
		&\le 2(5q^2-q)^{1-2/a}  \|K\|^2_a \left( \frac{\ell - 1}{(m-1)(m-\ell+1)} \right)^{1-2/a} .
	\end{align*}
	The top-$q$ order statistics of $Z_{[m:(2m-\ell)]}$ are the same as those of $Z_{[m:(2m-1)]}$ unless at least one of the variables $Z_j$ for $j \in [(2m-\ell+1){:}(2m-1)]$ is larger than the $q$-largest order statistic of $Z_{[m:(2m-\ell)]}$. Let $C_{\ell,m}$ denote the latter event. It follows that, for $a>2$,
	\begin{align*}
		&\left|
		\expec\left[ K_m(Z_{[1:m]}) \, K_{m-\ell+1}(Z_{[m:(2m-\ell)]})  \1_{m,1,m} \right]\right.\\
		&\left.\;\;\;\;\;
		-
		\expec\left[ K_m(Z_{[1:m]}) \, K_{m}(Z_{[m:(2m-1)]}) \1_{m,1,m} \right]
		\right| 
		\\
		&\le \expec\left[
		\left| K_m(Z_{[1:m]}) \right|
		\left| K_{m-\ell+1}(Z_{[m:(2m-\ell)]}) - K_{m}(Z_{[m:(2m-1)]}) \right|
		\1_{m,1,m}
		\right] \\
		&\le 2 \|K\|^2_a \expec\left[\1_{m,1,m}; C_{\ell,m}\right]^{1-2/a} .
	\end{align*}
	Further, we have
	\begin{align*}
		&\expec\left[\1_{m,1,m}; C_{\ell,m}\right] \\
		=&
		\prob \biggl( \left\{\text{$Z_m$ is in the top-$q$ of $Z_{[1:m]}$ and/or $Z_{[m:(2m-1)]}$}\right\} \\
		&\cap \bigcup_{j=2m-\ell+1}^{2m-1} \left\{
		\text{$Z_j$ is larger than the $q$-largest order statistic of $Z_{[m:(2m-\ell)]}$}
		\right\} \biggr).
	\end{align*}
	The latter probability is bounded by $\ell-1$ times the sum
	\begin{align*}
		&\prob\left(\text{$Z_m$ is in the top-$q$ of $Z_{[1:m]}$ and} \right.\\
		&\left.\;\;\;\;\;\text{     $Z_{2m-\ell+1}$ is in the top-$q$ of $Z_{[(m+1):(2m-\ell+1)]}$}\right) \\
		+& \prob\left(\text{$Z_m$ and $Z_{2m-\ell+1}$ are both in the top-$q$ of $Z_{[m:(2m-\ell+1)]}$}\right)
	\end{align*}
	and thus by
	\[
	(\ell-1) \left(
	\frac{q}{m} \cdot \frac{q}{m-\ell+1} + \frac{q(q-1)}{(m-\ell+2)(m-\ell+1)}
	\right) 
	\le (2q^2-q) \frac{\ell-1}{(m-\ell+1)^2}.
	\]
	Since,
	\[
	\expec\left[ K_m(Z_{[1:m]}) \, K_{m}(Z_{[m:(2m-1)]}) \1_{m,1,m} \right]
	= \expec\left[ \Xi_{1,m}; N_{1,m} = 1\right],
	\]
	a combination of the previous inequalities yields
	\begin{align*}
		\lefteqn{\left| 
			\expec[ \Xi_{\ell,m}; N_{\ell,m} = 1 ] 
			- \ell \expec[ \Xi_{1,m}; N_{1,m} = 1 ] 
			\right|}
		\\
		&= \ell \left| \expec[ \Xi_{\ell,m}; B_{m,\ell,m}] - \expec\left[ K_m(Z_{[1:m]}) \, K_{m}(Z_{[m:(2m-1)]}) \1_{m,1,m} \right] \right| \\
		&\le \ell \|K\|^2_a \left[
		2(5q^2-q)^{1-2/a}   \left( \frac{\ell - 1}{(m-1)(m-\ell+1)} \right)^{1-2/a} \right.\\
		&\left.\qquad\qquad\qquad
		{} + 2 (2q^2-q)^{1-2/a} \left( \frac{\ell-1}{(m-\ell+1)^2} \right)^{1-2/a}
		\right] \\
		&\le 4(5q^2-q)^{1-2/a} \|K\|^2_a \ell \left(\frac{(\ell-1)}{(m-\ell+1)^2} \right)^{1-2/a}. 
		% 		\qedhere
	\end{align*}
\end{proof}
\begin{lemma}[Handling $\mathbb{I}_3$ in \eqref{eq:dcmp}]
	\label{lem:N2}
	If Condition \ref{Condition_kernel_scale-loc_integrability} is satisfied, then
	\[
	\mathbb{I}_3
	\le \|K\|^2_a \left( \frac{\ell(\ell-1)}{2} \right)^{1-2/a} \left( \frac{2(2q^2-q)}{(m-1)(m-\ell+1)} \right)^{1-2/a}
	\]
	for all $a > 2$.
\end{lemma}
\begin{proof}
	Since the kernel satisfies Condition \ref{Condition_kernel_scale-loc_integrability}, we have by H\"older's inequality
	\[
	\left| \expec[ \Xi_{\ell,m} ; N_{\ell,m} \ge 2 ] \right|
	\le \expec[ |\Xi_{\ell,m}| ; N_{\ell,m} \ge 2 ] \le \|K\|^2_a \prob(N_{\ell,m} \ge 2)^{1-2/a}
	\]
	for $a >2$. By Lemma~\ref{lem:N}, we have
	\[
	\prob(N_{\ell,m} \ge 2) \le \frac{\ell(\ell-1)}{2} \expec[\1_{m-1,\ell,m} \1_{m,\ell,m}].
	\]
	In view of Lemma~\ref{lem:E11}, we conclude
	\[
	\left| \expec[ \Xi_{\ell,m} ; N_{\ell,m} \ge 2 ] \right|
	\le \|K\|^2_a \left( \frac{\ell(\ell-1)}{2} \right)^{1-2/a} \left( \frac{2(2q^2-q)}{(m-1)(m-\ell+1)} \right)^{1-2/a}.
	\qedhere
	\]
\end{proof}
We now develop the required counting variable inequalities. In doing so, the following preliminary lemma will be useful a number of times. No originality is claimed.

\begin{lemma}
	\label{lem:N}
	Consider a counting variable $N = \sum_{i=1}^n \1_i$ with indicator variables $\1_1,\ldots,\1_n$ satisfying $\expec[\1_i \1_j] \le \eps$ for all $i, j \in [1{:}n]$ and $i \ne j$. Then
	\begin{equation*}
		\begin{aligned}
			2 \prob(N \ge 2) 
			&\le \expec(N) - \prob(N = 1) \\
			&\le 2 \left( \expec(N) - \prob(N \ge 1) \right) \\
			&\le n(n-1)\eps. 
		\end{aligned}
	\end{equation*}
\end{lemma}

\begin{proof}
	A case-by-case analysis reveals
	\[ 
	\1\{N \ge 2\} 
	\le \frac{1}{2}(N - \1\{N=1\})
	\le N - \1\{N \ge 1\} 
	\le \frac{1}{2} N(N-1)
	\]
	and thus
	\[
	\prob(N \ge 2)
	\le \frac{1}{2} \left( \expec(N) - \prob(N=1) \right)
	\le \expec(N) - \prob(N \ge 1) 
	\le \frac{1}{2} \left(\expec(N^2) -  \expec(N)\right).
	\]
	But since $\1_i^2 = \1_i$, we have
	\[
	\expec(N^2) = \sum_{i=1}^n \sum_{j=1}^n \expec[\1_i \1_j]
	= \sum_{i=1}^n \expec[\1_i] + \sum_{i,j \in [1:n], i \ne j} \expec[\1_i \1_j]
	\le \expec(N) + n(n-1) \eps,
	\]
	and thus
	\[
	\frac{1}{2} \left(\expec(N^2) -  \expec(N)\right) 
	\le \frac{n(n-1)}{2} \eps,
	\]
	yielding the stated inequalities.
\end{proof}

Recall the indicator variables $\1_{i,\ell,m}$ in \eqref{eq:1ilm}.
\begin{lemma}
	\label{lem:E11}
	We have
	\begin{equation*}
		\expec[\1_{m-1,\ell,m} \1_{m,\ell,m}]
		\le \frac{2(2q^2-q)}{(m-1)(m-\ell+1)}.
	\end{equation*}
\end{lemma}

\begin{proof}
	By symmetry,
	\begin{align*}
		\lefteqn{
			\expec[\1_{m-1,\ell,m} \1_{m,\ell,m}]
		} \\
		&= \prob\left(\text{$Z_{m-1}$ and $Z_m$ are both in the top-$q$ of $Z_{[1:m]}$ or $Z_{[(m-\ell+1):(2m-\ell)]}$}\right) \\
		&\le 2 \prob\left(\text{$Z_{m-1}$ and $Z_m$ are both in the top-$q$ of $Z_{[1:m]}$}\right) \\
		&\quad {} +
		2 \prob\left(\text{$Z_{m-1}$ is in the top-$q$ of $Z_{[1:m]}$}\right.\\
		&\:\: \left. \text{\hspace{1.1cm}        and $Z_{m}$ is in the top-$q$ of $Z_{[(m-\ell+1):(2m-\ell)]}$}\right).
	\end{align*}
	The first probability on the right-hand side is equal to $q(q-1)/[m(m-1)]$ while the second one is bounded by
	\begin{multline*}
		\prob\left(\text{$Z_{m-1}$ is in the top-$q$ of $Z_{[1:(m-1)]}$ and $Z_{m}$ is in the top-$q$ of $Z_{[m:(2m-\ell)]}$}\right) \\ 
		= \frac{q}{m-1} \cdot \frac{q}{m-\ell+1}.
	\end{multline*}
	Add the two bounds and note that $1/m <1/(m-\ell+1)$ to conclude the proof.
\end{proof}

\begin{lemma}
	\label{lem:E11diff}
	We have
	\begin{equation*}
		\expec\left[\left| \1_{m,\ell,m}  - \1_{m,1,m} \right|\right]
		\le (q+1)q \frac{\ell-1}{(m+1)m}.
	\end{equation*}
\end{lemma}

\begin{proof}
	For events $A$ and $B$, we have $\left|\1_A - \1_B\right| = \1_{A \symdif B}$ where the symmetric difference $A \symdif B = (A \setminus B) \cup (B \setminus A)$ is the event that exactly one of $A$ and $B$ occurs.
	It follows that $\left| \1_{m,\ell,m}  - \1_{m,1,m} \right|$ is the indicator of the event that exactly one of the two events
	\begin{align*}
		&\{ \text{$Z_m$ is in the top-$q$ of $Z_{[1:m]}$ or $Z_m$ is in the top-$q$ of $Z_{[(m-\ell+1):(2m-\ell)]}$}\} \: \text{and} \\
		&\{ \text{$Z_m$ is in the top-$q$ of $Z_{[1:m]}$ or $Z_m$ is in the top-$q$ of $Z_{[m:(2m-1)]}$}\}
	\end{align*}
	occurs. In that case, $Z_m$ is either in the top-$q$ of $Z_{[(m-\ell+1):(2m-\ell)]}$ but not in the top-$q$ of $Z_{[m:(2m-1)]}$ or the other way around, and thus
	\begin{align*}
		&\lefteqn{
			\expec\left[\left| \1_{m,\ell,m}  - \1_{m,1,m} \right|\right]
		} \\
		&\le \prob\left(
		\text{$Z_m$ is in the top-$q$ of $Z_{[(m-\ell+1):(2m-\ell)]}$ but not in that of $Z_{[m:(2m-1)]}$}
		\right) \\
		&{} +
		\prob\left(\text{$Z_m$ is in the top-$q$ of $Z_{[m:(2m-1)]}$ but not in that of $Z_{[(m-\ell+1):(2m-\ell)]}$}
		\right) \\
		&= 2 \prob\left(
		\text{$Z_m$ is in the top-$q$ of $Z_{[m:(2m-1)]}$ but not in that of $Z_{[(m-\ell+1):(2m-\ell)]}$}
		\right).
	\end{align*}
	For the probability of the latter event, note that there always needs to be at least one of the $\ell-1$ random variables of  $Z_{[(m-\ell+1):(m-1)]}$ larger than $Z_m$ for $Z_m$ not to be in the top-$q$ of $Z_{[(m-\ell+1):(2m-\ell)]}$. Thus, by the union bound, the probability is bounded by $(\ell-1)$ times the probability that $Z_{m-1}$ and $Z_m$ are both in the top-$(q+1)$ of $Z_{[(m-1):(2m-1)]}$ while $Z_{m-1}$ is still larger than $Z_m$. It follows that
	\[
	\left| \expec[\1_{m,\ell,m} - \1_{m,1,m}] \right|
	\le 2 (\ell-1) \cdot \frac{(q+1)q}{2(m+1)m} 
	= (q+1)q \frac{\ell-1}{(m+1)m}. \qedhere
	\] 
\end{proof}

\begin{lemma}
	\label{lem:PN1diff}
	We have
	\[
	\left| \prob(N_{\ell,m} \ge 1) - \ell \prob(N_{1,m} = 1) \right|
	\le 3q^2 \cdot \frac{\ell(\ell-1)}{(m-1)(m-\ell+1)}.
	\]
\end{lemma}

\begin{proof}
	In view of Lemma~\ref{lem:N}, we have
	\[
	\left| \prob(N_{\ell,m} \ge 1) - \ell \expec[\1_{m,\ell,m}] \right|
	\le \frac{\ell(\ell-1)}{2} \expec[\1_{m-1,\ell,m} \1_{m,\ell,m}]
	\]
	and thus, since $N_{1,m} = \1_{m,1,m}$,
	\begin{align*}
		\left| \prob(N_{\ell,m} \ge 1) - \ell \prob(N_{1,m} = 1) \right|&
		\le \frac{\ell(\ell-1)}{2} \expec[\1_{m-1,\ell,m} \1_{m,\ell,m}]\\
		&\:\:\:\:
		+ \ell \expec\left[\left|\1_{m,\ell,m}  - \1_{m,1,m} \right|\right] .
	\end{align*}
	The two expectations on the right-hand side were treated in Lemmas~\ref{lem:E11} and~\ref{lem:E11diff}, respectively. Combine the bounds to obtain that
	\begin{align*}
		\lefteqn{\left| \prob(N_{\ell,m} \ge 1) - \ell \prob(N_{1,m} = 1) \right|} \\
		&\le \frac{\ell(\ell-1)}{2} \cdot \frac{2(2q^2-q)}{(m-1)(m-\ell+1)}
		+ \ell \cdot (q+1)q \frac{\ell-1}{(m+1)m} \\
		&\le  \frac{3q^2\ell(\ell-1)}{(m-1)(m-\ell+1)}.
	\end{align*}
\end{proof}

\begin{lemma}
	\label{lem:EBmlm1mlm}
	We have
	\begin{equation*}
		\expec\left[\left|\1_{B_{m,\ell,m}} - \1_{m,\ell,m}\right|\right]
		\le 2(2q^2-q) \frac{\ell-1}{(m-1)(m-\ell+1)}.
	\end{equation*}
\end{lemma}

\begin{proof}
	We have $\1_{m,\ell,m} \ge \1_{B_{m,\ell,m}}$ while the difference between the two indicators is equal to the indicator of the event that $\1_{m,\ell,m} = 1$ holds together with $\1_{j,\ell,m} = 1$ for some $j \in [(m-\ell+1){:}(m-1)]$. By the union bound, we obtain
	\[
	\expec\left[\left|\1_{B_{m,\ell,m}} - \1_{m,\ell,m}\right|\right]
	\le (\ell-1) \expec[ \1_{m-1,\ell,m} \1_{m,\ell,m} ].
	\]
	Apply Lemma~\ref{lem:E11} to conclude.
\end{proof}

%% file: appCprop23.tex
\begin{proof}[Proof of Proposition \ref{Prop:Hajek_distribution}] 
	The proof is divided in two parts. The first part will deal with the expression for the asymptotic variance in~\eqref{eq:sigmaK} and the second part with the asymptotic normality of the H\'ajek projection in~\eqref{Part2}.
	
	\paragraph{Part 1: asymptotic variance.}
	
	Write $\overline{K}_m := K_m - u(\gamma)$. We start by calculating the limit
	\[
	\lim_{r \to \infty} r \, \zeta_{r+1,1}, \quad \text{where} \quad
	\zeta_{r+1,1} = \expec \left[
	\overline{K}_{r+1}(Z_1,\ldots,Z_r,Z)
	\cdot
	\overline{K}_{r+1}(Z_1',\ldots,Z_r',Z)
	\right]
	\]
	and $Z,Z_1,Z_1',Z_2,Z_2',\ldots$ are independent $\GP(\gamma)$ random variables. The expectation defining $\zeta_{r+1,1}$ is finite by Condition \ref{Condition_kernel_scale-loc_integrability} and the Cauchy--Schwarz inequality. We will prove that 
	\begin{equation}
		\label{eq:sigmagK2}
		\lim_{r \to \infty} r \, \zeta_{r+1,1}=\sigma_{\gamma,K}^2 = \int_0^\infty
		\left( \expec \left[
		\overline{K}_{q+1} \bigl(
		(Z_{q-j:q-1})_{j=1}^q, h_\gamma(S_q/x)
		\bigr)
		\right] \right)^2 \, \diff x<\infty,
	\end{equation}
	where $Z_{0:q-1} = 0$ and $Z_{1:q-1} \le \ldots \le Z_{q-1:q-1}$ are the order statistics of an independent $\GP(\gamma)$ random sample $Z_1,\ldots,Z_{q-1}$, while $S_q = E_1+\cdots+E_q$ with $E_1,\ldots,E_q$ an independent unit-exponential random sample, independent of $Z_1,\ldots,Z_{q-1}$.
	
	To prove \eqref{eq:sigmagK2}, we will show that we can rewrite $r \, \zeta_{r+1,1}$ as
	\begin{equation}
		\label{eq:rzeta_int}
		r \, \zeta_{r+1,1} = \int_{0}^{\infty} \left( \expec \left[
		\overline{K}_{q+1} \bigl(
		(h_\gamma(Y_{q-j:q-1}))_{j=1}^q, h_\gamma(r U_{q:r} / x)
		\bigr) \cdot \1 \{ r U_{q:r} > x \}
		\right] \right)^2 \diff x,
	\end{equation}
	where $Y_{0:q-1} := 1$, $Y_{1:q-1} \le \ldots \le Y_{q-1:q-1}$ are the order statistics of a standard Pareto random sample of size $q-1$ and $U_{1:r} \le \ldots \le U_{r:r}$ are the order statistics of a uniform$(0, 1)$ random sample of size $r$, the two samples being independent.  Since $r U_{q:r}$ converges in distribution to $S_q$, the remainder of the proof then consists in justifying why the limit can be interchanged with the expectation and the integral. We will do this in two steps: first we will show that the expectation inside the integral in \eqref{eq:rzeta_int} converges, i.e., for all $x > 0$, we have the convergence
	\begin{multline}
		\label{eq:expconv}
		\lim_{r \to \infty} \expec \left[
		\overline{K}_{q+1} \bigl(
		(h_\gamma(Y_{q-j:q-1}))_{j=1}^q, h_\gamma(rU_{q:r}/x)
		\bigr) \cdot \1 \{ rU_{q:r} > x \}
		\right] \\
		= \expec \left[
		\overline{K}_{q+1} \bigl(
		(h_\gamma(Y_{q-j:q-1}))_{j=1}^q, h_\gamma(S_q/x)
		\bigr) \cdot \1 \{ S_q > x \}
		\right],
	\end{multline}
	where $S_q$ is as in \eqref{eq:sigmagK2}. Second, we will apply Lebesque's dominated convergence theorem to show that the integral itself converges.

	We start with the proof of \eqref{eq:rzeta_int}. Condition on $Z$ to get
	\begin{align*}
		r \, \zeta_{r+1,1}
		&= \int_0^\infty \expec \left[ \overline{K}_{r+1}(Z_1,\ldots,Z_r,z)
		\cdot
		\overline{K}_{r+1}(Z_1',\ldots,Z_r',z)
		\right]
		\cdot r \cdot \diff \prob(Z \le z) \\
		&= \int_0^\infty \left( \expec \left[ \overline{K}_{r+1}(Z_1,\ldots,Z_r,z) \right] \right)^2
		\cdot r \cdot \diff \prob(Z \le z) \\
		&= \int_0^\infty \left( \expec \left[ \overline{K}_{q+1}\bigl((Z_{r-j+1:r})_{j=1}^q,z\bigr) \right] \right)^2
		\cdot r \cdot \diff \prob(Z \le z), 
	\end{align*}
	in terms of the ascending order statistics $Z_{1:r} \le \ldots \le Z_{r:r}$ of $Z_1,\ldots,Z_r$. Here we used the fact that the kernel $K_m$ only depends on the top-$q$ values of its argument. Note that if $Y$ is a standard Pareto random variable, then the distribution of $h_\gamma(Y)$ is $\GP(\gamma)$. Let $Y_{1:r} \le \ldots \le Y_{r:r}$ be the ascending order statistics of a standard Pareto random sample $Y_1,\ldots,Y_r$. It follows that
	\begin{align*}
		r \, \zeta_{r+1,1}
		&= \int_1^\infty \left( \expec \left[ 
		\overline{K}_{q+1} \bigl( 
		(h_\gamma(Y_{r-j+1:r}))_{j=1}^q, h_\gamma(y)
		\bigr)
		\right]\right)^2 \cdot r \frac{\diff y}{y^2} \\
		&= \int_{1/r}^\infty \left( \expec \left[ 
		\overline{K}_{q+1} \bigl( 
		(h_\gamma(Y_{r-j+1:r}))_{j=1}^q, h_\gamma(r t)
		\bigr)
		\right]\right)^2 \frac{\diff t}{t^2},
	\end{align*}
	where we have substituted $y = rt$ in the last step. The vector of top-$q$ standard Pareto order statistics satisfies the distributional representation
	\[
	(Y_{r-j+1:r})_{j=1}^q \eqd
	(Y_{q-j:q-1} / U_{q:r})_{j=1}^q
	\]
	where $Y_{0:q-1} := 1$ and where $U_{1:r} \le \ldots \le U_{r:r}$ denote the ascending order statistics of an independent uniform$(0, 1)$ sample $U_1,\ldots,U_r$, independent of $Y_1,\ldots,Y_r$; note that $1/U_{q:r}$ has the same distribution as $Y_{r-q+1:r}$. The function $h_\gamma$  satisfies the functional equation
	\[
	\forall x, y > 0, \qquad
	h_\gamma(y/x) = \frac{h_\gamma(y) - h_\gamma(x)}{x^\gamma}.
	\]
	By location-scale invariance of the kernel,
	\begin{align*}
		\lefteqn{\overline{K}_{q+1} \bigl( 
			(h_\gamma(Y_{q-j:q-1} / U_{q:r}))_{j=1}^q, h_\gamma(r t)
			\bigr)}\\
		&= \overline{K}_{q+1} \bigl(
		(h_\gamma(Y_{q-j:q-1}))_{j=1}^q, U_{q:r}^\gamma h_\gamma(r t) + h_\gamma(U_{q:r})
		\bigr) \\
		&= \overline{K}_{q+1} \bigl(
		(h_\gamma(Y_{q-j:q-1}))_{j=1}^q, h_\gamma(r U_{q:r} t)
		\bigr).
	\end{align*}
	Following the substitution $x = 1/t$, we obtain
	\begin{align*}
		r \, \zeta_{r+1,1}
		&= \int_{1/r}^\infty \left( \expec \left[
		\overline{K}_{q+1} \bigl(
		(h_\gamma(Y_{q-j:q-1}))_{j=1}^q, h_\gamma(r U_{q:r} t)
		\bigr)
		\right] \right)^2 \frac{\diff t}{t^2} \\
		&= \int_{0}^r \left( \expec \left[
		\overline{K}_{q+1} \bigl(
		(h_\gamma(Y_{q-j:q-1}))_{j=1}^q, h_\gamma(r U_{q:r} / x)
		\bigr)
		\right] \right)^2 \diff x.
	\end{align*}
	By independence and Fubini's theorem, the expectation inside the integral is equal to
	\begin{multline*}
		\expec \left[
		\overline{K}_{q+1} \bigl(
		(h_\gamma(Y_{q-j:q-1}))_{j=1}^q, h_\gamma(r U_{q:r} / x)
		\bigr)
		\right] \\
		=
		\int_0^1 \expec \left[
		\overline{K}_{q+1} \bigl(
		(h_\gamma(Y_{q-j:q-1}))_{j=1}^q, h_\gamma(r u / x)
		\bigr) 
		\right] \, \diff \prob(U_{q:r} \le u).
	\end{multline*}
	The expectation inside the integral on the right-hand side is zero as soon as $r u / x \le 1$, because in that case
	\begin{multline*}
		\expec \left[
		\overline{K}_{q+1} \bigl(
		(h_\gamma(Y_{q-j:q-1}))_{j=1}^q, h_\gamma(r u / x)
		\bigr) 
		\right]\\
		=
		\expec \left[
		\overline{K}_{q} \bigl( (h_\gamma(Y_{q-j:q-1}))_{j=1}^q \bigr)
		\right] =
		\expec \left[
		\overline{K}_{q} \bigl( (Z_{q-j:q-1})_{j=1}^q \bigr)
		\right]
		= 0.
	\end{multline*}
	Therefore, we obtain
	\begin{multline*}
		\expec \left[
		\overline{K}_{q+1} \bigl(
		(h_\gamma(Y_{q-j:q-1}))_{j=1}^q, h_\gamma(r U_{q:r} / x)
		\bigr)
		\right] \\
		= 
		\expec \left[
		\overline{K}_{q+1} \bigl(
		(h_\gamma(Y_{q-j:q-1}))_{j=1}^q, h_\gamma(r U_{q:r} / x)
		\bigr) \cdot \1 \{ r U_{q:r} > x \}
		\right]. 
	\end{multline*}
	Equation \eqref{eq:rzeta_int} follows as the integral is equal to $0$ for values of $x$ higher than $r$.
	
	Next, we show \eqref{eq:expconv} and eventually \eqref{eq:sigmagK2}. The well-known representation of uniform order statistics in terms of partial sums of unit-exponential random variables, i.e., $U_{q:r} \overset{d}=S_q/S_{r+1}$ where, for any integer $a$, $S_a = E_1+\cdots+E_a$ is the sum of independent standard exponential random variables independent of $Y_1,\ldots,Y_{q-1}$, implies the weak convergence
	\begin{equation}
		\label{eq:rUrS}
		r U_{q:r} \stackrel{d}{\longrightarrow} S_q = E_1+\cdots+E_q,
		\qquad r \to \infty.
	\end{equation}
	The question is whether in \eqref{eq:expconv}, we can switch the limit and the expectation.
	
	Under a suitable Skorokhod construction, we can assume that \eqref{eq:rUrS} holds almost surely. By continuity of the kernel, the integrand inside the expectation on the left-hand side then converges almost surely to the integrand on the right-hand side. To show uniform integrability in \eqref{eq:expconv} and eventually the validity of switching the limit and integral in \eqref{eq:sigmagK2}, it is sufficient to show that the expectation of the square is bounded by an integrable function $c(x)$, i.e.,
	\[
	\limsup_{r \to \infty}
	\expec \left[
	\overline{K}_{q+1}^2 \bigl(
	(h_\gamma(Y_{q-j:q-1}))_{j=1}^q, h_\gamma(rU_{q:r}/x)
	\bigr) \cdot \1 \{ rU_{q:r} > x \}
	\right]
	< c(x) < \infty,
	\]
	where $c(x)$ is a function on $[0,+\infty)$ such that $\int_0^{+\infty}c(x)\diff x<\infty$. Then not only the relation \eqref{eq:expconv} follows immediately, but also the relation \eqref{eq:sigmagK2} follows from the Lebesgue dominance convergence theorem.
	
	To do so, we rely on a change-of-measure, replacing $rU_{q:r}/x$ on the event $\{r U_{q:r} > x\}$ by a standard Pareto random variable, and compensating by the likelihood ratio. For $0 < x < r$, let $p_{r,x}$ be the probability density function of $r U_{q:r}/x$ and let $Y$ be a standard Pareto random variable, independent of $Y_1,\ldots,Y_{q-1}$. We have
	\begin{align*}
		\lefteqn{
			\expec \left[
			\overline{K}_{q+1}^2 \bigl(
			(h_\gamma(Y_{q-j:q-1}))_{j=1}^q, h_\gamma(rU_{q:r}/x)
			\bigr) \cdot \1 \{ rU_{q:r} > x \}
			\right]
		} \\
		&= \int_1^\infty 
		\expec\left[
		\overline{K}_{q+1}^2 \bigl(
		(h_\gamma(Y_{q-j:q-1}))_{j=1}^q, h_\gamma(y)
		\bigr) \right] \cdot p_{r,x}(y) \, \diff y \\
		&= \int_1^\infty 
		\expec\left[
		\overline{K}_{q+1}^2 \bigl(
		(h_\gamma(Y_{q-j:q-1}))_{j=1}^q, h_\gamma(y)
		\bigr) \right] \cdot p_{r,x}(y) \cdot y^2 \, \frac{\diff y}{y^2} \\
		&=
		\expec \left[
		\overline{K}_{q+1}^2 \bigl(
		(h_\gamma(Y_{q-j:q-1}))_{j=1}^q, h_\gamma(Y)
		\bigr) \cdot p_{r,x}(Y) \cdot Y^2
		\right] \\
		&= 	\expec \left[
		\overline{K}_{q}^2 \bigl(
		h_\gamma(Y_1), \ldots, h_{\gamma}(Y_{q-1}), h_\gamma(Y)
		\bigr) \cdot p_{r,x}(Y) \cdot Y^{2-\epsilon} \cdot Y^{\epsilon}
		\right],
	\end{align*}
	for some $0<\epsilon<1/2$.
	
	Recall that the distribution of $(h_\gamma(Y_1), \ldots, h_\gamma(Y_{q-1}), h_\gamma(Y))$ is equal to the one of an independent $\GP(\gamma)$ random sample of size $q$. 
	In view of Condition \ref{Condition_kernel_scale-loc_integrability}, and the fact that $\expec[Y^{2\epsilon}]<\infty$, it is sufficient to show that
	\begin{equation*}
		\sup_{r > 2q} \sup_{y \ge 1} p_{r,x}(y) \cdot y^{2-\epsilon} \le \tilde c(x),
	\end{equation*}
	for some finite function $\tilde c(x)$ such that $\int_0^{+\infty}\tilde c(x)\diff x<\infty$. Since the distribution of $U_{q:r}$ is $\operatorname{Beta}(q, r+1-q)$, we have, for $y \in [1, \infty)$,
	\begin{align}
		\nonumber
		p_{r,x}(y) 
		&= \frac{\diff}{\diff y} \prob(r U_{q:r} / x \le y) \\
		\nonumber
		&= \frac{\diff}{\diff y} \prob(U_{q:r} \le xy/r) \\
		\nonumber
		&= (x/r) \cdot \frac{\Gamma(r+1)}{\Gamma(q) \Gamma(r+1-q)} \cdot (xy/r)^{q-1} \cdot (1 - xy/r)^{r-q} \cdot \1\{xy < r\} \\
		\label{eq:prxy}
		&= \frac{x^q}{(q-1)!} \cdot \frac{r!}{(r-q)! r^q} \cdot y^{q-1} \cdot (1 - xy/r)^{r-q} \cdot \1\{xy < r \}.
	\end{align}
	As $1 - h \le e^{-h}$ for real $h$, we get, for $r > 2q $, the inequalities
	\begin{equation*}
		p_{r,x}(y)
		\le \frac{x^q}{(q-1)!} \cdot y^{q-1} \cdot e^{-(r-q)xy/r}
		\le \frac{x^q}{(q-1)!} \cdot y^{q-1} \cdot e^{-xy/2}.
	\end{equation*}
	Hence,
	\[   p_{r,x}(y) \cdot y^{2-\epsilon} \leq \frac{x^q}{(q-1)!} \cdot y^{q+1-\epsilon} \cdot e^{-xy/2}, \quad \text{ for all } r \geq q, \text{ all } x>0 \text{ and all } y \geq 1.
	\]
	With straightforward calculation, we get that the upper bound is maximized at $y=1 \vee2(q+1-\epsilon)/x$. It follows that we can choose
	\begin{align*}
		\tilde{c}(x)=
		\begin{cases}
			\dfrac{(2(q+1-\epsilon))^{q+1-\epsilon}}{(q-1)!e^{q+1-\epsilon}} \cdot x^{-1+\epsilon} &\text{ if $0<x\leq 2(q+1-\epsilon)$,}
			\\[1em]
			\dfrac{x^{q}}{(q-1)!} \cdot e^{-x/2} &\text{ if $x>2(q+1-\epsilon)$,}
		\end{cases}
	\end{align*}
	which satisfies $\int_{0}^{\infty} \tilde{c}(x) \, \diff x<\infty$. This concludes the proof of \eqref{eq:expconv} and \eqref{eq:sigmagK2}.
	
	\paragraph{Part 2: asymptotic normality.}
	
	It remains to show that the H\'ajek projection $\hat{U}_{n,Z}^m$ in \eqref{eq:Hajek} of $U_{n,Z}^m - \theta$ is asymptotically normal.
	Note that $\hat{U}_{n, Z}^m$ is a centered row sum of a triangular array of row-wise i.i.d.\ random variables, with total variance $\zeta_{m, 1} m^2/n$. If Lyapunov's condition is met, following the central limit theorem, we have
	\begin{equation*}
		\frac{k}{\sqrt{n \zeta_{m,1}}} \hat{U}_{n,Z}^m \td N(0, 1), \qquad n \to \infty.
	\end{equation*}
	Therefore, we need to show that for some $\delta>0$ we have
	\begin{equation}
		\label{eq:Lyapunov}
		\lim_{m \to \infty} \frac{1}{\left(\sqrt{n\zeta_{m,1}}\right)^{2+\delta}} \cdot n
		\expec \left[ \bigl|\hat{K}_m(Z)\bigr|^{2+\delta}\right]  = 0.
	\end{equation}
	First, by Jensen's inequality,
	\[
	\bigl| \hat{K}_m(z) \bigr|^{2+\delta}
	\le \expec \left[
	\left| \overline{K}_m(Z_1,\ldots,Z_{m-1},z) \right|^{2+\delta}
	\right].
	\]	
	Second, by Fubini's theorem, location-scale invariance of the kernel and the stability property of GP order statistics,
	\[
	\expec \left[ \bigl| \hat{K}_m(Z) \bigr|^{2+\delta} \right]
	\le \expec \left[
	\left| \overline{K}_m(Z_1,\ldots,Z_m) \right|^{2+\delta}
	\right]
	= \expec \left[
	\left| \overline{K}_q(Z_1,\ldots,Z_q) \right|^{2+\delta}
	\right],
	\]
	which implies that the term $\expec [ | \hat{K}_m(Z) |^{2+\delta} ]$ is uniformly bounded over all $m$. 
	
	Since $\lim_{m \to \infty} m \zeta_{m,1} = \sigma_{K}^2(\gamma) > 0$, the relation~\eqref{eq:Lyapunov} holds provided we can choose $\delta>0$ such that
	\[ 
	\lim_{n\to\infty} \left(\frac{m}{n}\right)^{(2+\delta)/2} n = 0.
	\]
	From Condition \ref{degree_sequence}, this can be achieved by choosing $\delta>2/\epsilon$.
\end{proof}

%% file: appDprop3.tex
We start by listing two preliminary lemmas.

\begin{lemma}{(Potter's bounds; see \cite[Proposition B.1.9]{deHaanFerreira2006extreme}).} \label{Potter}
	Let $f\in RV_{\alpha}$. For all $\delta_1, \delta_2>0$ there exists $t_0=t_0(\delta_1, \delta_2)$ such that for any positive $t$ and $x$ satisfying $t \geq t_0$ and $tx \geq t_0$, 
	\[
	(1-\delta_1)x^{\alpha}\min \left( x^{\delta_2}, x^{-\delta_2} \right) 
	\leq \frac{f(tx)}{f(t)} 
	\leq (1+\delta_1)x^{\alpha}\max \left( x^{\delta_2}, x^{-\delta_2} \right).
	\]
\end{lemma}

\begin{lemma} \label{uniform_integrability}
	Let $Y^*_{1:m}\leq Y^*_{2:m} \leq \ldots \leq Y^*_{m:m}$ be the order statistics of an independent random sample of Pareto$(1)$ random variables. For integer $q \ge 1 $ and for a real number $\nu<q$, the sequence $\{(Y_{m-q+1:m}^*/m)^\nu : m \ge m_0 \}$ is uniformly integrable for some $m_0$ depending on $q$ and $\nu$.
\end{lemma}

\begin{proof}[Proof of Lemma \ref{uniform_integrability}]
	Note that $Y^*_{m-q+1:m}\overset{d}{=} \frac{S_{m+1}}{S_q}$, where $S_l=\sum_{i=1}^{l}E_i$ for independent unit-exponential random variables $E_i$. Since it is only the distribution of $Y^*_{m-q+1:m}$ that matters, we can change probability spaces to ensure that the equality in distribution is in fact an equality between random variables.
	
	We first deal with $(Y_{m-q+1:m}^*/m)^{\nu}$ for $\nu \leq 0$. We have
	\begin{align*}
		(Y_{m-q+1:m}^*/m)^{\nu} \leq 
		\left(\frac{S_q}{(m+1)^{-1}\sum_{j=1}^{m+1}E_j}  \right)^{-\nu}.
	\end{align*}
	Note that $H_m=(m+1)^{-1}\sum_{i=j}^{m+1} E_j $ is distributed as a Gamma$(m+1, m+1)$ random variable with value $m+1$ for both the shape and rate parameter. By Cauchy's inequality,
	\begin{align*}
		\expec\left[(Y_{m-q+1:m}^*/m)^{2\nu}\right] \leq 
		\sqrt{ \expec[ H_m^{2\nu} ]} \sqrt{ \expec[ S_q^{-2\nu} ]}.
	\end{align*}
	We have $\expec[ S_{q}^{-2\nu} ]< \infty$ for all $\nu \leq 0$ while by direct calculation, provided $m>-2\nu-1$,
	\begin{align*}
		\expec[ H_m^{2\nu} ] = \frac{(m+1)^{-2\nu} \Gamma(m+2\nu+1)}{\Gamma(m+1)} \to 1 \text{ as } m \to \infty.
	\end{align*}
	Therefore, there exists $m_0(\nu)$ such that the sequence $\{ ( Y^*_{m-q+1:m}/m )^{\nu}, m \ge m_0(\nu) \}$ is bounded in $\mathcal{L}^2$ and thus uniformly integrable.
	
	Next we consider $( Y^*_{m-q+1:m}/m )^{\nu}$ for $0 < \nu< q$.
	Define $p_0=(3\nu+q)/4$ and note that $\nu < p_0 < (\nu+q)/2$. Also, define $p_1= (\nu+q)/(2p_0)$, which satisfies $p_1 > 1$, and define $z > 1$ by $1/z+1/p_1=1$. 
	For sufficiently large integer $m$, we find, using H\"older's inequality and the inequality $1/m \le 2 / (m+1)$, that
	\begin{align*}
		\expec\left[(Y^*_{m-q+1:m}/m)^{p_0}\right] 
		&\leq 2^{p_0} \expec \left[( H_m/S_q )^{p_0} \right]\\
		&\leq 2^{p_0} \left( \expec[H_m^{zp_0}] \right)^{1/z}  \left(\expec[S_q^{-p_1p_0}]\right)^{1/p_1}.
	\end{align*}
	From $p_0p_1 = (\nu+q)/2 < q$ it follows that $\expec[ S_q^{-p_0p_1} ] < \infty$. Similar to the case where $\nu<0$, for any $a>0$ we have $\sup_{m \ge 1} \expec[H_m^{a}]  <\infty$.
	It follows that 
	\[
	\sup_{m \ge q} \expec\left[( Y^*_{m-q+1:m}/m )^{p_0} \right]< \infty.
	\]
	Since $p_0 > \nu > 0$, the sequence $\{( Y^*_{m-q+1:m}/m )^{\nu}, m \ge q\}$ is uniformly integrable.
\end{proof}

\begin{proof}[Proof of Theorem~\ref{Prop:bias}]
	Following Theorem~\ref{Main-XU-result}, we only need to show that 
	Condition~\ref{biasterm_with_limit_BK} holds with the function $B_K(\rho,\gamma)$ given in \eqref{eq:BK}.

	Recall from the discussion in Section \ref{DoA_bias} that
	\begin{align*}
		\sqrt{k}\left(  { U}^{m}_{n}-{ U}^{m}_{n,Z} \right)
		&=\sqrt{k}   \binom{n}{m}^{-1}  \sum_{I \subset [n], |I| = m} \left( K_m\left(X_I\right)-  K_m\left(Z_I\right) \right)  \\
		&=\sqrt{k}A(m )   \binom{n}{m}^{-1}  \sum_{I \subset [n], |I| = m} \xi^I_m\\
		&=:\sqrt{k}A(m ) U^m_{n, \xi},
	\end{align*}
	where $ \xi^I_m=K_m\left(X_I\right)-  K_m\left(Z_I\right)$. In what follows, we write $\xi_m$ to indicate a generic $\xi^I_m$ where the index set is not important.
	
By Hoeffding's decomposition,
\begin{align*}
	\var  \left( U^m_{n, \xi} \right)
	= \sum_{\ell=1}^m p_{n,m}(\ell) \, \Psi_{m,\ell}
\end{align*}
with $p_{n,m}(\ell)$ in \eqref{eq:hyppmf}
and $\Psi_{m,\ell}=\cov (\xi^I_m, \xi^{I'}_m )$ for index sets $I, I' \subset [n]$ such that $|I'|=|I|=m$ and $|I \cap I'|=\ell$. Below we will show the limit relation
\begin{equation}
	\label{eq:varto0}
	\sum_{\ell=1}^m p_{n,m}(\ell) \Psi_{m,\ell} \to 0, \qquad \text{as } n \to \infty.
\end{equation}
Then it follows from Chebychev's inequality that
\[
U^m_{n, \xi} - \expec[\xi_m] \pto 0, \qquad \text{as } n \to \infty.
\]
In addition, we will show that
\begin{equation}
	\label{eq:Etophi}
	\lim_{m \to \infty} \expec[\xi_m] =\expec \left[\lim_{m \to \infty}\xi_m \right] = B_{K}(\rho, \gamma),
\end{equation}
with $B_K(\rho,\gamma)$ defined as in \eqref{eq:BK}. Clearly, the last two limit relations imply that $U^m_{n,\xi}$ converges in probability to $B_{K}(\rho, \gamma)$, yielding Condition~\ref{biasterm_with_limit_BK}. It remains to show \eqref{eq:varto0}--\eqref{eq:Etophi}.
	
	We start by proving \eqref{eq:Etophi}. By exploiting the scale and location invariance of $K$ and Lemma \ref{lem:GPKm} for the $\GP(\gamma)$ distributed random variables, $\xi_m^I$ is equal in distribution to
	\begin{equation*}
		\xi_m := \frac{K\left( \{ h(Y_{q-j: q-1};Y^{*}_{m-q+1:m}) \}_{j=1}^q \right)   -K\left(\{h_{\gamma}(Y_{q-j:q-1}\right)\}_{j=1}^q) }{A(m)},
	\end{equation*}
	where $Y_{0: q-1}=1$, $Y_{q-1:q-1} \geq Y_{q-2:q-1} \geq \ldots \geq Y_{1: q-1}$ are Pareto(1) order statistics independent of $Y^*_{m-q+1:m}$ and $h(x;t)$ is the function defined as
	\[
	h(x;t)=\frac{U(tx)-U(t)}{t \, U'(t)}.
	\] 
Under Condition~\ref{derivative_tail_quantile}, following Theorem~2.3.12 in~\citep{deHaanFerreira2006extreme}, we have, for $x>0$,
$$\lim_{t \to \infty}h(x;t)=\lim_{t \to \infty}\frac{U(tx)-U(t)}{t \, U'(t)}=h_\gamma(x), $$ 
and 
\begin{equation}
\label{eq:Hgammarho}
\lim_{t \to \infty}\frac{h(x;t)-h_\gamma(x)}{A(t)}
=  \int_{1}^{x}s^{\gamma-1}\int_{1}^{s}u^{\rho-1} \, \diff u \, \diff s
=: H_{\gamma, \rho}(x).
\end{equation}
	
By Condition~\ref{kernel_derivatives} the function $K$ is differentiable and all $q$ first-order partial derivatives $\dot{K}_1,\ldots,\dot{K}_q$ are continuous. The mean value theorem for several variables implies that there exists a point $\hat{Z}_{m}$ on the line segment between the points $ \{h(Y_{q-j:q-1}; Y^{*}_{m-q+1:m}) \}_{j=1}^{q} $ and $  \{Z_{q-j:q-1} \}_{j=1}^{q}=\{h_{\gamma}\left( Y_{q-j:q-1}\right)\}_{j=1}^q$ such that as $m \to \infty$,
	\begin{align} 
		\nonumber
		\xi_m
		&=\frac{A(Y^*_{m-q+1:m})}{A(m)}
		\sum_{j=1}^{q} \dot{K}_j(\hat{Z}_{m}) 
		\frac{h(Y_{q-j: q-1}; Y^{*}_{m-q+1:m} ) - h_{\gamma}\left(Y_{q-j:q-1} \right)  }{A(Y^*_{m-q+1:m})} \\ 
		&\to S_q^{-\rho} \sum_{j=1}^{q} \dot{K}_j\left( \{h_{\gamma}\left( Y_{q-j:q-1}\right) \}_{j=1}^q\right)
		H_{\gamma, \rho}(Y_{q-j: q-1})
		\label{xi}
	\end{align}
	almost surely, where $S_q$ is a properly constructed Erlang($q$) random variable independent of $Y_{q-1:q-1} \geq Y_{q-2:q-1} \geq \ldots \geq Y_{1: q-1}$. To see this, note that in a suitable Skorokhod construction, $m^{-1} Y^*_{m-q+1:m} \to S_q^{-1}$ almost surely. In addition, $ A(m x) / A(m)\to x^{\rho}$ as $m\to\infty$. Therefore, it follows by the extended continuous mapping theorem that $A(Y^*_{m-q+1:m}) / A(m) \to S_q^{-\rho}$ as $m \to \infty$.
	
	In view of \eqref{xi}, we show that the sequence $\{\xi_m\}_{m \ge m_0}$ is uniformly integrable for sufficiently large integer $m_0$. In particular, we will show that
	\[
	\sup_{m \ge m_0} \expec[ |\xi_m|^{1+\epsilon/2} ] < \infty
	\]
	for some positive integer $m_0$, with $\epsilon > 0$ as given in Condition~\ref{kernel_moment}.
	Write $\eta = 1+\epsilon/2$.
	We deal with $\xi_m$ on the event 
	\[
	D_m=\{ Y^*_{m-q+1:m} > t_0  \} 
	\]
	and its complement $D_m^c$ for $t_0>1$ to be defined later.
	
	Firstly, we handle $\expec( |\xi_m|^{\eta} ; D_m^c )$. By independence of $Y_{m-q+1:m}^*$ and $Y_1,\ldots,Y_{q-1}$, we have, by the Cauchy--Schwarz inequality, for any $t_0>1$,
	\begin{multline*}
		\expec( |\xi_m|^{\eta} ; D_m^c )
		\le \frac{2^\eta \expec[|K_m(X_I)|^{\eta} + |K_m(Z_I)|^{\eta}; D_m^c]}{|A(m)|^{\eta} } \\
		\le \frac{2^\eta}{|A(m)|^{\eta}} \left\{[\expec(|K_m(X_I)|^{2\eta})]^{1/2}  + [\expec(|K_m(Z_I)|^{2\eta})]^{1/2} \right\} [\prob(D_m^c)]^{1/2} \to 0,
	\end{multline*}
	as $m \to \infty$. To obtain the final step, we use the following three facts: $\expec(|K_m(Z_I)|^{2\eta})$ is finite and does not depend on $m$, $ \sup_{m \ge q} \expec(|K_m(X_I)|^{2\eta}) =O(m^a)$ for some $a>0$ by Condition \ref{kernel_moment} and $\prob(D_m^c)$ tends to zero much faster than any power of $|A(m)|$ (exponentially versus polynomially).

	Next, we handle $\expec( |\xi_m|^{\eta} ; D_m )$. Consider again the expansion~\eqref{xi} of $\xi_m$ by the mean value theorem with
	\[ 
	\hat{Z}_m=(1-\hat{a})\{h(Y_{q-j:q-1}; Y^{*}_{m-q+1:m}) \}_{j=1}^{q} +\hat{a} \{h_\gamma(Y_{q-j:q-1}) \}_{j=1}^{q}
	\]
	for some random $0 < \hat{a} < 1$. Since $q=3$, we abbreviate $(\hat{Z}_{3-j:2})_{j=1}^3:=\hat{Z}_m$, where $\hat{Z}_{2:2} > \hat{Z}_{1:2} > \hat{Z}_{0:2} = 0$.
	
	Fix $\delta>0$ to be determined later. Since $|A|$ is regularly varying, the inequality in Lemma~\ref{Potter} implies that there exists some $t_0(\delta)>0$ such that
	\begin{align*}
		\left|\xi_m \right|^{\eta} \mathbb{I}_{D_m}& \leq (1+\delta) \left( \frac{Y^*_{m-2:m}}{m} \right)^{\eta(\rho \pm \delta)}   \left| \sum_{j=1}^{3} \dot{K}_{j} (\hat{Z}_m) \frac{ \tilde{h}(Y_{3-j:2})}{A(Y^*_{m-2:m})}    \right| ^{\eta}
	\end{align*}
	where $(x)^{\pm \delta}  $ is short-hand for $\max \{x^{\delta},x^{-\delta}\}$ and where $\tilde{h}(Y_{i:2})=  h(Y_{i: 2}; Y^{*}_{m-2:m} )-h_{\gamma}\left(Y_{i:2} \right)$. With straightforward calculation, we obtain
\[
	\dot{K}_1(x) = \frac{g'\left( \ln\left(\frac{x_1-x_2}{x_2-x_3}\right) \right)}{x_1 - x_2}, 
	\dot{K}_3(x) = \frac{g'\left( \ln\left(\frac{x_1-x_2}{x_2-x_3}\right) \right)}{x_2 - x_3},
	\dot{K}_2(x) = - \dot{K}_1(x)-\dot{K}_3(x).
\]
	
	By exploiting the derivative structure, we have
	\begin{align*}
		\lefteqn{\sum_{j=1}^{3} \dot{K}_{j} (\hat{Z}_m) \tilde{h}(Y_{3-j:2})} \\
		&=
		\dot{K}_1(\hat{Z}_m) \left( \tilde{h}(Y_{2:2}) - \tilde{h}(Y_{1:2}) \right)
		+
		\dot{K}_3(\hat{Z}_m) \left( \tilde{h}(Y_{0:2}) - \tilde{h}(Y_{1:2}) \right)	\\
		&= g' \left(\ln \left(
		\tfrac{\hat{Z}_{2:2} - \hat{Z}_{1:2}}{\hat{Z}_{1:2}-\hat{Z}_{0:2}}
		\right)\right)
		\left(
		\frac{\tilde{h}(Y_{2:2}) - \tilde{h}(Y_{1:2})}{\hat{Z}_{2:2} - \hat{Z}_{1:2}}
		+
		\frac{\tilde{h}(Y_{0:2}) - \tilde{h}(Y_{1:2})}{\hat{Z}_{1:2}-\hat{Z}_{0:2}}
		\right).
	\end{align*}
	It follows that
	\begin{multline*}
		\left|\xi_m \right|^{\eta} \mathbb{I}_{D_m}
		\leq (1+\delta) \left( \frac{Y^*_{m-2:m}}{m} \right)^{\eta(\rho \pm \delta)}  \left|  g'\left( \ln\left(\frac{
			\hat{Z}_{2:2}-\hat{Z}_{1:2}}{\hat{Z}_{1:2}-\hat{Z}_{0:2}}\right) \right) \right|^{\eta} \\ \cdot \left| \sum_{j=1}^{2}   \frac{\tilde{h}(Y_{3-j:2})-\tilde{h}(Y_{2-j:2})   }{\left( \hat{Z}_{3-j:2}-\hat{Z}_{2-j:2}\right)A(Y^*_{m-2:m})}  \right|^{\eta}.
	\end{multline*}

	By Lemma~\ref{Galton_kernel} below, any positive power of $$  \left|\sum_{j=1}^{2} \frac{\tilde{h}(Y_{3-j:2})-\tilde{h}(Y_{2-j:2})    } {\left(\hat{Z}_{3-j:2}-\hat{Z}_{2-j:2} \right)A(Y^*_{m-2:m})}\right| $$ can be bounded by an integrable random variable not depending on $m$, provided $Y^*_{m-2:m}>t_0$ for some sufficiently large $t_0$.	Consequently, to show that $\sup_{m>m_0} \expec |\xi_m|^{\eta} \mathbb{I}_{D_m}< \infty$ by the H\"older's inequality, it suffices to prove that
	\begin{align*}
		\sup_{m>m_0} \expec \left[ \left| g'\left( \ln\left(\frac{
			\hat{Z}_{2:2}-\hat{Z}_{1:2}}{\hat{Z}_{1:2}-\hat{Z}_{0:2}}\right) \right) \right|^{r\eta} \left( \frac{Y^*_{m-2:m}}{m} \right)^{r \eta(\rho\pm \delta)} \right] <\infty
	\end{align*}
	for some $r>1$. 
	
	Lemma~\ref{uniform_integrability} shows that $\left( Y^*_{m-2:m}/m \right)^{r\eta(\rho\pm \delta)}$ is uniformly integrable for any $r>1$, with choosing $\delta$ such that $\rho \pm \delta  \leq 3/(r\eta)$. Therefore, it remains to prove that there exists some $r>1$ such that
	\begin{equation}
		\label{eq:boundg}
		\sup_{m>m_0}\expec\left| g'\left( \ln\left(\frac{
			\hat{Z}_{2:2}-\hat{Z}_{1:2}}{\hat{Z}_{1:2}-\hat{Z}_{0:2}}\right) \right) \right|^{r \eta}< \infty.
	\end{equation}
	Note that the relation \eqref{eq:boundg} holds trivially if part~(b) of Condition~\ref{kernel_derivatives} holds. We show that it also holds if part~(a) of Condition~\ref{kernel_derivatives} holds. Write $\frac{b}{c}:=\frac{
		{Z}_{2:2}-{Z}_{1:2}}{{Z}_{1:2}-{Z}_{0:2}}$ and $\frac{d}{e}:=\frac{
		h(Y_{2:2};Y^{*}_{m-2:m} )-h(Y_{1:2};Y^{*}_{m-2:m} )}{h(Y_{1:2};Y^{*}_{m-2:m} )-h(Y_{0:2};Y^{*}_{m-2:m} )}$.
	Then $\frac{
		\hat{Z}_{2:2}-\hat{Z}_{1:2}}{\hat{Z}_{1:2}-\hat{Z}_{0:2}}=\frac{\hat{a} b+(1-\hat{a})d}{\hat{a} c+(1-\hat{a}) e}$ is bounded by $\frac{b}{c}$ and $\frac{d}{e}$. By the assumed monotonicity of $g'$, we have
	\begin{multline*}
		\left| g'\left( \ln\left(\frac{\hat{Z}_{2:2}-\hat{Z}_{1:2}}{\hat{Z}_{1:2}-\hat{Z}_{0:2}} \right) \right) \right|^{\eta r} \leq \left| g'\left( \ln\left(\frac{
			{Z}_{2:2}-{Z}_{1:2}}{{Z}_{1:2}-{Z}_{0:2}}\right) \right) \right|^{\eta r} \\
		+ \left| g'\left( \ln\left(\frac{
			h(Y_{2:2};Y^{*}_{m-2:m} )-h(Y_{1:2};Y^{*}_{m-2:m} )}{h(Y_{1:2};Y^{*}_{m-2:m} )-h(Y_{0:2};Y^{*}_{m-2:m} )}\right) \right) \right|^{\eta r}. 
	\end{multline*}
	Together with Condition \ref{kernel_derivatives}, it completes the proof of the relation \eqref{eq:boundg}.
	
	By combining the uniform integrability of $ \left|\xi_m\right|^{\eta} \mathbb{I}_{D_m}$ and $ \left|\xi_m\right|^{\eta} \mathbb{I}_{D_m^c}$, we conclude that $\xi_m$ is uniformly integrable. It follows that
	\begin{multline*}
		\lim_{m \to \infty} \expec [\xi_m]  = \expec \left[\lim_{m \to \infty} \xi_m\right] \\
		=  \expec\left[ S_3^{-\rho} \right]  \expec \left[ \sum_{j=1}^{3}\dot{K}_{(j)}\left( \{h_{\gamma}\left( Y_{3-j:2}\right) \}_{j=1}^3\right)H_{\gamma, \rho}\left(Y_{3-j,:2} \right) \right] = B_{K}(\rho, \gamma),
	\end{multline*}
	which completes the proof of \eqref{eq:Etophi}.
	
	Finally, we turn to prove \eqref{eq:varto0}. Note that all steps above in conjunction with Lemma \ref{Galton_kernel} can be adapted to prove uniform integrability of $\xi^2_m$, from which it follows that 
	\[
	\lim_{m \to \infty} \expec [\xi^2_m] 
	= \expec \left[ S^{-2\rho}_3 \right] \expec\left[ \left( \sum_{j=1}^{3}\dot{K}_{(j)}\left( \{h_{\gamma}\left( Y_{3-j:2}\right) \}_{j=1}^3\right)H_{\gamma, \rho}\left(Y_{3-j:2} \right) \right)^2 \right]
	=:\phi.
	\]
	
	By classical theory on U-statistics \citep[Chapter 5.2]{serfling2009approximation}, we have
	\begin{align*}
		\sum_{\ell=1}^m p_{n,m}(\ell) \Psi_{m,\ell} \leq \frac{m}{n}\Psi_{m, m},
	\end{align*}
	where $\Psi_{m,m}$ is the variance of $\xi_m$. The relation \eqref{eq:varto0} follows immediately from $\expec[\xi_m] \to B_{K}(\rho, \gamma)$ and $\expec[\xi^2_m] \to \phi$ as $m \to \infty$.
\end{proof}

\begin{lemma} \label{Galton_kernel}
	Recall  $\tilde{h}(Y_{i:2})= h(Y_{i:2}; Y^{*}_{m-2:m} )-h_{\gamma}\left(Y_{i:2} \right) $ and $\hat{Z}_m=(1-\hat{a})\{h(Y_{3-j:2}; Y^{*}_{m-2:m}) \}_{j=1}^{3} +\hat{a} \{h_\gamma(Y_{3-j:2}) \}_{j=1}^{3}$ for some $0 < \hat{a} < 1$, where $\hat{a}$ is a random variable. Recall that $\{ \hat{Z}_{3-j:2}\}_{j=1}^3:=\hat{Z}_m$ where $\hat{Z}_{2:2} \geq \hat{Z}_{1:2} \geq \hat{Z}_{0:2}$ and ${D_m}=\{ Y^*_{m-q+1:m} > t_0  \}$. For sufficiently large $t_0>1$, we have
	\[
	\sup_{m \ge 3} \sum_{j=1}^{2}\left| \frac{\tilde{h}(Y_{3-j:2})-\tilde{h}(Y_{2-j:2})    } {\left(\hat{Z}_{3-j:2}-\hat{Z}_{2-j:2} \right)A(Y^*_{m-2:m})}\right|\mathbb{I}_{D_m} \leq B
	\]
	for a random variable $B$ satisfying $\expec[ |B|^r ] < \infty$ for any $r>0$.
\end{lemma}

\begin{proof}
	For brevity, we only deal with the first term in the summation:
	\[
	\left| \frac{\tilde{h}(Y_{2:2})-\tilde{h}(Y_{1:2})    } {(\hat{Z}_{2:2}-\hat{Z}_{1:2}) \, A(Y^*_{m-2:m})}\right|.
	\]
	The second term can be handled using the same inequalities even though $Y_{0:2}=1$.
	
	Recall the function $A(t)$ from Condition~\ref{derivative_tail_quantile} as
	\[
	A(t)=\frac{t \, U''(t)}{U'(t)}-\gamma+1.
	\]
	Following the proof of Theorem~2.3.12 in \cite{deHaanFerreira2006extreme}, we can write
	\[
	\frac{\tilde{h}(Y_{2:2})-\tilde{h}(Y_{1:2})   }{A(Y^*_{m-2:m})}= \int_{Y_{1:2}}^{Y_{2:2}} s^{\gamma-1} \int_{1}^{s} \frac{A(Y^*_{m-2:m}u)}{A(Y^*_{m-2:m})}\frac{U'(Y^*_{m-2:m}u)}{U'(Y^*_{m-2:m})} u^{-\gamma}\, \diff u \, \diff s.
	\]
	By regular variation of the absolute value of $A$ and $U'$ with indices $\rho$ and $\gamma-1$ respectively and using Potter's bounds in Lemma \ref{Potter}, we get that for any $\delta_2>0$ there exists a $t_0(\delta_2)$ such that if $Y^*_{m-2:m}>t_0(\delta_2)$,
	\begin{align*}
		\lefteqn{\int_{Y_{1:2}}^{Y_{2:2}} s^{\gamma-1} \int_{1}^{s} \frac{A(Y^*_{m-2:m}u)}{A(Y^*_{m-2:m})}\frac{U'(Y^*_{m-2:m}u)}{U'(Y^*_{m-2:m})} u^{-\gamma} \, \diff u \, \diff s}
		\\
		&\leq \int_{Y_{1:2}}^{Y_{2:2}} s^{\gamma-1} \int_{1}^{s} (1+\delta_2) u^{\rho-1+\delta_2} \, \diff u \, \diff s\\
		&= \frac{1+\delta_2}{\rho+\delta_2} \left( h_{\gamma+\rho+\delta_2}\left(Y_{2:2}\right)- h_{\gamma+\rho+\delta_2}\left(Y_{1:2}\right)  \right) - \frac{1+\delta_2}{\rho+\delta_2}\left(h_\gamma(Y_{2:2})-h_\gamma(Y_{1:2})\right),
	\end{align*}
	provided $\rho+\delta_2 \neq 0$, and similarly,
	\begin{multline*}
		\int_{Y_{1:2}}^{Y_{2:2}} s^{\gamma-1} \int_{1}^{s} \frac{A(Y^*_{m-2:m}u)}{A(Y^*_{m-2:m})}\frac{U'(Y^*_{m-2:m}u)}{U'(Y^*_{m-2:m})} u^{-\gamma} \, \diff u \, \diff s \\
		\geq \frac{1-\delta_2}{\rho-\delta_2} 
		\left( h_{\gamma+\rho-\delta_2}\left(Y_{2:2}\right)
		- h_{\gamma+\rho-\delta_2}\left(Y_{1:2}\right)  \right) 
		- \frac{1-\delta_2}{\rho-\delta_2}
		\left(h_\gamma(Y_{2:2})-h_\gamma(Y_{1:2})\right).
	\end{multline*} 
	Accordingly, it suffices to prove that, for any $r>0$ and $\gamma \in \reals$, the term
	\begin{align*}
		&\left| \frac{h_{\gamma}(Y_{2:2})- h_{\gamma}(Y_{1:2})}{\hat{Z}_{2:2}-\hat{Z}_{1:2}} \right|^r 
	\end{align*}
	can be bounded by an integrable function not depending on $m$.

	Note that
	\begin{multline*}
		\hat{Z}_{2:2}-\hat{Z}_{1:2}  =(1-\hat{a})\left(  h(Y_{2:2}; Y^{*}_{m-2:m}) -h(Y_{1:2}; Y^{*}_{m-2:m})  \right) \\ +\hat{a} \left( h_\gamma(Y_{2:2})   - h_\gamma(Y_{1:2}) \right).
	\end{multline*}
	For any $\delta_1>0$, there exists $t_0(\gamma, \delta_1)>1$ such that, if $Y^*_{m-2:m}>t_0(\gamma, \delta_1)$, following the Potter's bounds in Lemma~\ref{Potter}, we have that
	\begin{multline*}
		h(Y_{2:2}; Y^{*}_{m-2:m}) -h(Y_{1:2}; Y^{*}_{m-2:m})
		= \int_{Y_{1:2}}^{Y_{2:2}}\frac{U'(t Y^*_{m-2:m})}{U'(Y^*_{m-2:m})} \, \diff t \\
		\geq (1-\delta_1) \int_{Y_{1:2}}^{Y_{2:2}}t^{\gamma-1 -\delta_1} \, \diff t
		= (1-\delta_1) \left( h_{\gamma-\delta_1}(Y_{2:2}) - h_{\gamma-\delta_1}(Y_{1:2})\right),
	\end{multline*}
	provided $\gamma-\delta_1 \neq 0$. Define $\tilde{Y} = Y_{2:2} / Y_{1:2}$. Recalling that $0<\hat{a}<1$, we get that
	\begin{align*}
		&\hat{Z}_{2:2}-\hat{Z}_{1:2} \\
		\geq&  (1-\hat{a}) (1-\delta_1)\left( h_{\gamma-\delta_1}(Y_{2:2}) - h_{\gamma-\delta_1}(Y_{1:2})\right)+\hat{a} \left(h_\gamma(Y_{2:2})-h_\gamma(Y_{1:2})  \right)\\
		\geq&  (1-\hat{a}) (1-\delta_1)\left( h_{\gamma-\delta_1}(Y_{2:2}) - h_{\gamma-\delta_1}(Y_{1:2})\right)+\hat{a}(1-\delta_1) \left(h_\gamma(Y_{2:2})-h_\gamma(Y_{1:2})  \right)\\
		\geq& (1-\delta_1)Y_{2:2}^{-\delta_1}\left( h_{\gamma}(Y_{2:2}) - h_{\gamma}(Y_{1:2})\right),
	\end{align*}
	where the last inequality follows from the fact that for $0 < x < y < \infty$, for $\gamma \in \reals$ and for $0 \le \delta < \infty$, we  have
	\begin{align*}
		y^\delta \left( h_{\gamma-\delta}(y) - h_{\gamma-\delta}(x) \right)
		&= y^\delta \int_x^y t^{\gamma-\delta-1} \, \diff t = \int_x^y t^{\gamma - 1} (y/t)^\delta \, \diff t \\
		&\ge \int_x^y t^{\gamma - 1} \, \diff t = h_\gamma(y) - h_\gamma(x). 
	\end{align*}
	We conclude with the bound
	\[
	\left| \frac{h_{\gamma}(Y_{2:2})- h_{\gamma}(Y_{1:2})}{\hat{Z}_{2:2}-\hat{Z}_{1:2}} \right| \leq \frac{ h_{\gamma}(Y_{2:2})-h_{\gamma}(Y_{1:2})}{ (1-\delta_1)Y_{2:2}^{-\delta_1}\left( h_{\gamma}(Y_{2:2}) - h_{\gamma}(Y_{1:2})\right)}\leq \frac{ Y_{2:2}^{\delta_1}}{(1-\delta_1)},
	\]
	from which the lemma follows by choosing $\delta_1<(2 r)^{-1}$.
\end{proof}

%% file: appEsec4.tex
To show Proposition~\ref{unbiased_estimator}, we use known formulas for the expectation of exponentiated Generalized Pareto random variables and their order statistics in terms of the digamma function $\psi$.

\begin{lemma}
	\label{moment_EXP_GP}
	Let $\gamma \in \reals$ and let $Z_1, Z_2$ be two independent $\GP(\gamma)$ random variables, with order statistics $Z_{1:2} \le Z_{2:2}$.
	We have
	\begin{align*}
		\expec[ \ln Z_1 ]
		&=
		\begin{cases}
			\ln(1/\gamma)+\psi(1) - \psi(1/\gamma) \qquad &\text{ if } \gamma>0 \\
			\ln(-1/\gamma) +\psi(1) - \psi(1-1/\gamma) \qquad &\text{ if } \gamma<0 \\
			\psi(1) \qquad &\text{ if } \gamma=0 ,
		\end{cases}
		\\
		\expec[ \ln Z_{2:2} ]
		&=
		\begin{cases}
			\ln(1/\gamma)+\psi(1) +\psi(2/\gamma) - 2\psi(1/\gamma) \qquad &\text{ if } \gamma>0  \\
			\ln(-1/\gamma) +\psi(1) +\psi(1-2/\gamma)- 2\psi(1-1/\gamma) \qquad &\text{ if } \gamma<0 \\
			\psi(1)+\ln(2) \qquad &\text{ if } \gamma=0 ,
		\end{cases}
		\\
		\expec[ \ln Z_{1:2} ]
		&=
		\begin{cases}
			\ln(1/\gamma)+\psi(1) -\psi(2/\gamma) \qquad &\text{ if } \gamma>0  \\
			\ln(-1/\gamma) +\psi(1) -\psi(1-2/\gamma) \qquad &\text{ if } \gamma<0 \\
			\psi(1)-\ln(2)   \qquad &\text{ if } \gamma=0.
		\end{cases}
	\end{align*} 
	As a consequence, for all $\gamma \in \reals$, we have
	\begin{equation}
		\label{eq:u12}
		\expec \left[ \KP(Z_{2:2}, Z_{1:2}, 0) \right] = \gamma.
	\end{equation}
\end{lemma} 

\begin{proof}
	The formulas for $\expec[\ln Z_1]$, $\expec[\ln Z_{1:2}]$ and $\expec[\ln Z_{2:2}]$ follow directly from the moments of exponentiated GP random variables and their order statistics, see \cite[Section~2.3]{lee2019exponentiated}.
	
	Let $Y_1, Y_2$ be two independent standard Pareto random variables and let $Y_{1:2} \leq Y_{2:2}$ denote their order statistics as. 
	Recall that we can represent the $\GP(\gamma)$ random variables $Z_1, Z_2$ and their corresponding order statistics $Z_{1:2} \leq Z_{2:2}$ as $Z_i=h_\gamma(Y_i)$, $Z_{1:2}=h_{\gamma}(Y_{1:2})$ and $Z_{2:2}=h_{\gamma}(Y_{1:2})$. But then
	\[
	Z_{2:2} - Z_{1:2} 
	= h_\gamma(Y_{2:2}) - h_\gamma(Y_{1:2})
	= Y_{1:2}^\gamma h_\gamma(Y_{2:2} / Y_{1:2}).
	\]
	Since $Y_{2:2} / Y_{1:2}$ is equal in distribution to $Y_1$, we find
	\begin{align*}
		u_1(\gamma)
		&:= \expec \left[ 
			\ln \left( \frac{Z_{2:2} - Z_{1:2}}{Z_{1:2}} \right) 
		\right] \\
		&= \expec[\ln Y_{1:2}^\gamma] + \expec[\ln Z_1] - \expec[\ln Z_{1:2} ] \\
		&=
		\begin{cases}
			\gamma/2 -\psi(1/\gamma)  + \psi(2/\gamma) &\text{ if } \gamma>0 \qquad \\
			\gamma/2- \psi(1-1/\gamma)+\psi(1-2/\gamma) &\text{ if } \gamma<0 \qquad\\
			\ln(2) &\text{ if } \gamma=0. \qquad
		\end{cases}
	\end{align*}
	More straightforwardly, we have
	\begin{align*}
		u_2(\gamma)
		&:= \expec \left[ \ln \left( \frac{Z_{2:2}}{Z_{1:2}} \right) \right] 
		= \expec [ \ln Z_{2:2} ] - \expec[ \ln Z_{1:2} ] \\
		&=
		\begin{cases}
			2 \left(\psi(2/\gamma)  - \psi(1/\gamma)\right) &\text{ if } \gamma>0 \qquad \\
			2 \left( \psi(1-2/\gamma)-\psi(1-1/\gamma) \right) &\text{ if } \gamma<0 \qquad\\
			2\ln(2) &\text{ if } \gamma=0. \qquad
		\end{cases}
	\end{align*}
	It follows that $\expec[ \KP(Z_{2:2}, Z_{2:1}, 0)] = 2 \, u_1(\gamma) - u_2(\gamma) = \gamma$, as required.
\end{proof}

\begin{proof}[Proof of Proposition~\ref{unbiased_estimator}]
	We have
	\begin{align*}
		\expec \left[ \KPm(Z_1,\ldots,Z_m) \right]
		&= \expec \left[ \KP(Z_{m:m}, Z_{m-1:m}, Z_{m-2:m}) \right] \\
		&= \expec \left[ \KP(Z_{2:2}, Z_{1:2}, 0) \right] 
		= \gamma,
	\end{align*}
	where we used Lemma~\ref{lem:GPKm} and the location-scale invariance of $\KP$ in the second step and Lemma~\ref{moment_EXP_GP} in the last step.
\end{proof}

\begin{proof}[Proof of Theorem~\ref{Prop_kernel_density}]
	We apply Theorem~\ref{Prop:bias}. To this end, we need to show that Conditions~\ref{Condition_kernel_scale-loc_integrability}, \ref{degree_sequence}, \ref{derivative_tail_quantile}, \ref{balance_bias}, \ref{kernel_derivatives} and \ref{kernel_moment} hold. Conditions \ref{degree_sequence}, \ref{derivative_tail_quantile}, and \ref{balance_bias} hold by assumption.
	
	One can check that the kernel $\KP$ is differentiable, has continuous first-order partial derivatives and is of the form \eqref{eq:Kg} for an auxiliary function $g$ that satisfies $|g'(x)|=|2-e^x/(e^x+1)| \leq 2$. Therefore, Condition~\ref{kernel_derivatives} holds too.	
	It remains to show that $\KPm$ satisfies Conditions~\ref{Condition_kernel_scale-loc_integrability} and~\ref{kernel_moment}.
	
	Condition~\ref{Condition_kernel_scale-loc_integrability} can be verified as follows:
	for any $p>0$, with the same notation as in the proof of Lemma~\ref{moment_EXP_GP},
	\begin{multline*}
		\expec \left[\left| \KPm(Z_1,\ldots,Z_m) \right|^p \right]=\mathbb{E} \left[ \left| 2\ln\left( \frac{Z_{2:2}-Z_{1:2}}{Z_{1:2}} \right) -\ln \left( \frac{Z_{2:2}}{Z_{1:2}}  \right) \right|^p \right]\\
		= \expec \left[ \left| 2 \ln \left[ Y_{1:2}^\gamma\right] +2\ln \left[ h_\gamma(Y_{2:2}/Y_{1:2}) \right] -\ln \left[h_\gamma(Y_{1:2})   \right] -\ln\left[h_{\gamma}(Y_{2:2}) \right]  \right|^p \right]< \infty,
	\end{multline*}
	because exponential random variables and exponentiated GP random variables (and its order statistics) have finite moments of all orders \cite[Corollary~3]{lee2019exponentiated}.  
	
	To show Condition~\ref{kernel_moment}, we rely on Condition~\ref{cond:f}. Note that $\KPm$ is a linear combination of the logarithms of the spacings $X_{m:m}-X_{m-1:m}$, $X_{m:m}-X_{m-2:m}$ and $X_{m-1:m}-X_{m-2:m}$. We have 
	\begin{multline*} 
		-\ln (X_{m:m}-X_{m-2:m}) \1_{\{  X_{m:m}-X_{m-2:m}< 1 \}} \\
		\leq -\ln(X_{m:m}-X_{m-1:m}) \1_{\{  X_{m:m}-X_{m-1:m}< 1  \}} 
	\end{multline*}
	and for $ 0 \leq j<i \leq 2$,
	\begin{multline*}
		\ln (X_{m-j:m}-X_{m-i:m}) \1_{\{  X_{m-j:m}-X_{m-i:m}> 1 \}} \\
		\leq \ln (X_{m:m}-X_{m-2:m}) \1_{\{  X_{m:m}-X_{m-2:m}> 1 \}} .
	\end{multline*}
	Consequently, to verify Condition~\ref{kernel_moment} we only have to show that
	\begin{align} 
		\label{small_gap} 
		\expec \left[
		\left(-\ln(X_{m-j:m}-X_{m-j-1:m})\right)^p 
		\1_{\{  X_{m-j:m}-X_{m-j-1:m}< 1  \}} 
		\right] = \Oh(m^a)
	\end{align}
	for $j \in \{0,1\}$, and 
	\begin{align} 
		\label{large_gap} 
		\expec \left[ 
		(\ln(X_{m:m}-X_{m-2:m}))^p
		\1_{\{  X_{m:m}-X_{m-2:m}> 1 \}} 
		\right] = \Oh(m^a) 
	\end{align} 
	for some $p>2$ and $a>0$. 
	
	First,  for the expectation in~\eqref{small_gap}, note that the distribution function $\tilde{F}_j$ of the spacing $X_{m-j:m}-X_{m-j-1:m}$ for $j \in \{0, 1\}$ is given by
	\[
	\tilde{F}_{j}(r)=1-\int_{-\infty}^{+\infty}(m-j-1) \binom{m}{m-j-1}f(x)F^{m-j-2}(x)\left(1-F(x+r)\right)^{j+1}\diff x.
	\]
	Moreover, note that $\left|(1-F(x))^2-(1-F(a))^2\right|\leq 4 \left|F(x)-F(a) \right|$ for all $x, a \in \mathcal{E}$ and $\tilde{F}_{j}(0)=0$. We get, for any $0<x_0<1$ and $r_0=(-\ln x_0)^p$, that
	\begin{align*}
		&\expec\left[(-\ln(X_{m-j:m}-X_{m-j-1:m}))^p\1_{\{  X_{m-j:m}-X_{m-j-1:m}< x_0  \}} \right]\\
		&=\int_{r_0}^{\infty} \prob(X_{m-j:m}-X_{m-j-1:m}<e^{-r^{1/p}}) \, \diff r\\
		&=\int_{r_0}^{\infty} \left(\tilde{F}_j( e^{-r^{1/p}}) -\tilde{F}_j\left( 0 \right) \right)\diff r\\
		&\leq (m-j-1) \binom{m}{m-j-1} \\
		&\:\:\:\: {} \cdot \int_{r_0}^{\infty}  \int_{-\infty}^{+\infty} f(x) {F}^{m-j-2}(x)\left|\bigl(1-F(x+e^{-r^{1/p}})\bigr)^{j+1}-\left(1-F(x)\right)^{j+1}  \right| \diff x \, \diff r\\
		& \leq  c \binom{m}{m-j-1}  \int_{r_0}^{\infty} \, r^{-(1+\epsilon)} \diff r,
	\end{align*}
	for some constant $c>0$.
	Here the last step follows from~\eqref{conditie:CDF} in Condition~\ref{cond:f}. With straightforward calculation, we get that as $m\to\infty$,
	\[
		\expec\left[(-\ln(X_{m-j:m}-X_{m-j-1:m}))^p\1_{\{  X_{m-j:m}-X_{m-j-1:m}< 1  \}} \right]=\Oh(m^2).
	\]
	
	Secondly, for the expectation in \eqref{large_gap}, consider the following five events:
	\begin{align*}
		A   &:= \{X_{m:m}-X_{m-2:m}>1\},\\
		A_1 &:= \{X_{m:m}>1,X_{m-2:m}<-1\},\\
		A_2 &:= \{X_{m:m}\leq 1,X_{m-2:m}<-1\},\\
		A_3 &:= \{X_{m:m}>1,X_{m-2:m}\geq -1\},\\
		A_4 &:= \{X_{m:m}\leq 1,X_{m-2:m}\geq -1\}.
	\end{align*}
	The following inequalities can be verified in a straightforward way:
	\begin{align*}
		\left(\ln\left(X_{m:m}-X_{m-2:m}\right) \right)^p\1_{A\cap A_1}&\leq 2^p\left[\left(\ln\left(X_{m:m}\right) \right)^p+\left(\ln\left(-X_{m-2:m}\right) \right)^p+1\right]\1_{A\cap A_1},\\
		\left(\ln\left(X_{m:m}-X_{m-2:m}\right) \right)^p\1_{A\cap A_2}&\leq 2^p\left[\left(\ln\left(-X_{m-2:m}\right) \right)^p+1\right]\1_{A\cap A_2},\\
		\left(\ln\left(X_{m:m}-X_{m-2:m}\right) \right)^p\1_{A\cap A_3}&\leq 2^p\left[\left(\ln\left(X_{m:m}\right) \right)^p+1\right]\1_{A\cap A_3},\\
		\left(\ln\left(X_{m:m}-X_{m-2:m}\right) \right)^p\1_{A\cap A_4}&\leq \1_{A\cap A_4}.
	\end{align*}
	Therefore, we only need to show that the following two expectations
	\[
	\expec \left[\left(\ln\left(X_{m:m}\right) \right)^p\1_{\{  X_{m:m}> 1  \}} \right] \text{\ \ and\ \ }\expec \left[\left(\ln\left(-X_{m-2:m}\right) \right)^p\1_{\{  X_{m-2:m}< -1  \}}\right],
	\]
	are of the order $\Oh(m^a)$ for some constant $a>0$.
	
	We start from the second expectation concerning $X_{m-2:m}$. There exists $x_0>1$ such that for all $m>3$, the variable $\left(-X_{m-2:m}\right)\1_{\{  X_{m-2:m}< -x_0  \}}$ is stochastically dominated by $\left(-X\right)\1_{\{  X< -x_0  \}}$, where $X$ has distribution $F$. Therefore, in view of Condition~\ref{cond:f}, we have 
	\begin{multline*}
		\expec \left[(\ln(-X_{m-2:m}))^p \1_{\{  X_{m-2:m}< -1  \}}\right] \\
		\leq 
		\expec \left[(\ln(-X))^p\1_{\{X<-x_0\}}\right]
		+(\ln(x_0))^p
		=\Oh(1).
	\end{multline*}
	
	Finally, for the first expectation concerning $X_{m:m}$, choose $0<\theta<1/\gamma$.  Note that there exists $x_0=x_0(p)>1$ such that $(\ln(x))^p$ is a concave function in $x > x_0$. In view of Jensen's inequality, we have
	\[
	\expec \left[ (\ln(X_{m:m}))^p \1_{\{  X_{m:m}> x_0  \}} \right]
	\leq \left(\frac{1}{\theta}\ln\left(\expec\left[X_{m:m}^\theta\1_{\{  X_{m:m}> x_0  \}} \right]\right)\right)^p.
	\]
	For $\theta<1/\gamma$, Theorem~5.3.2 in \citep{deHaanFerreira2006extreme} yields that as $m\to\infty$, 
	\[
	\expec\left[X_{m:m}^\theta\1_{\{  X_{m:m}> 1  \}} \right]
	= \left(U(m)\right)^{\theta} = \oh(m),
	\]
	which is sufficient to bound the first expectation by $\Oh(m^a)$ with any $a>0$. Combining the two expectations completes the proof of \eqref{large_gap} and thus that of Condition~\ref{kernel_moment}.
\end{proof}

\begin{proof}[Proof of Lemma~\ref{weights_estimator}]
	Put $S_{ij} = \ln(X_{n-i+1:n} - X_{n-j+1:n})$ for integer $1 \le i < j \le n$. For $I \subset [n]$ with $|I| = m$, both $K_{m,1}(X_I)$ and $K_{m,2}(X_I)$ are linear combinations of such log-spacings $S_{ij}$. For a given pair $(i, j)$ of indices, the number of such subsets $I$ in which $S_{ij}$ appears can be counted:
	\begin{itemize}
		\item There are $\binom{n-j}{m-2}$ sets $I$ such that $X_{n-i+1:n}$ and $X_{n-j+1:n}$ are respectively the largest and 2nd largest order statistics of $X_I$.
		\item There are $(i-1) \binom{n-j}{m-3}$ sets $I$ such that $X_{n-i+1:n}$ and $X_{n-j+1:n}$ are respectively the 2nd and 3rd largest order statistics of $X_I$.
		\item There are $(j-i-1) \binom{n-j}{m-3}$ sets $I$ such that $X_{n-i+1:n}$ and $X_{n-j+1:n}$ are respectively the largest and 3rd largest order statistics of $X_I$.
	\end{itemize}
	Since $n-j+1$ cannot be lower than $m-2$ (because there need to be at least $m-3$ observations lower than $X_{n-j+1:n}$), we find
	\[
	\UP = 2 U_{n,1}^m - U_{n,2}^m
	\]
	with 
	\begin{align*}
		U_{n,1}^m &= \binom{n}{m}^{-1}
		\sum_{j=2}^{n-m+3} \sum_{i=1}^{j-1}
		\left[\binom{n-j}{m-2} - (i-1) \binom{n-j}{m-3}\right] 
		S_{ij}, \\
		U_{n,2}^m &= \binom{n}{m}^{-1} 
		\sum_{j=2}^{n-m+3} \sum_{i=1}^{j-1}
		\left[ (j-i-1) \binom{n-j}{m-3} - (i-1) \binom{n-j}{m-3} \right]
		S_{ij}.
	\end{align*}
	It follows that
	\[
	\UP = \binom{n}{m}^{-1} 
	\sum_{j=2}^{n-m+3} 
	\left[ 2 \binom{n-j}{m-2} - (j-2) \binom{n-j}{m-3} \right] 
	\sum_{i=1}^{j-1} S_{ij},
	\]
	which can be further simplified to the stated formula.
\end{proof}